\documentclass[reqno]{amsart}

\usepackage{amsthm}
\usepackage{appendix}
\usepackage{amsmath, amssymb,mathptmx}
\usepackage[mathcal]{euscript}
\DeclareMathAlphabet{\mathcal}{OMS}{cmsy}{m}{n}
\usepackage{graphicx}

\newcommand{\qq}{\gamma}
\newcounter{margcount} %Zähler erzeugen
\newcommand{\ignore}[1]{{}}

\renewcommand{\AA}{{\mathcal A}}
\newcommand{\BB}{{\mathcal B}}

\newcommand{\DD}{{\mathcal D}}

\newcommand{\KK}{{\mathcal K}}
\newcommand{\LL}{{\mathcal L}}
\newcommand{\MM}{{\mathcal M}}
\newcommand{\NN}{{\mathcal N}}
\newcommand{\OO}{{\mathcal O}}
\newcommand{\PP}{{\mathcal P}}

\DeclareMathOperator*{\dist}{dist} 
\DeclareMathOperator*{\graph}{graph}

\DeclareMathOperator*{\supp}{supp}

\newcommand{\R}{\mathbb{R}}
\newcommand{\N}{\mathbb{N}}
\newcommand{\Z}{\mathbb{Z}}
\newcommand{\C}{\mathbb{C}}

\newcommand{\eps}{\epsilon}
\renewcommand{\phi}{\varphi}

\newcommand{\ddt}{\, \frac{d}{dt}}
\renewcommand{\iint}{\int\!\!\!\!\int}

\newcommand{\rring}[1]%
{\hspace{-0.4ex}{\stackrel{\raisebox{-0.1ex}[-5ex][-1ex]{\text{\ \tiny oo }}}{#1}}\hspace{-0.3ex}}

\def\XXint#1#2#3{{\setbox0=\hbox{$#1{#2#3}{\int}$}
\vcenter{\hbox{$#2#3$}}\kern-.5\wd0}}

\newcommand{\upref}[2]{\hspace{-0.8ex}\stackrel{\eqref{#1}}{#2}}
\newcommand{\lupref}[2]{\hspace{0ex} \stackrel{\eqref{#1}}{#2}}

\newcommand{\lupupref}[3]{\hspace{0ex}\stackrel{\eqref{#1},\eqref{#2}}{#3}}

\newcommand{\nco}[1]{\|#1\|_{L^\infty}}

\newcommand{\nlt}[1]{\|#1\|_{L^2}}
\newcommand{\nltL}[2]{\|#1\|_{L^2({#2})}}

\newcommand{\cciL}[1]{C_c^{\infty}(#1)}

\newcommand{\ciiL}[1]{C^{\infty}(#1)}

\newcommand{\NNN}[2]{\|#1\|_{#2}}

\newcommand{\XStfe}[1]{X_{0}^{#1}}

\newcommand{\XNtfe}[2]{\|#1\|_{\XStfe{#2}}}

\newcommand{\XSLtfe}[2]{\XStfe{#1}({#2})}

\newcommand{\XNLtfe}[3]{\|#1\|_{\XSLtfe{#2}{#3}}}

\newcommand{\PXSLtfe}[2]{T\XSLtfe{#1}{#2}}
\newcommand{\PXNLtfe}[3]{\|#1\|_{\PXSLtfe{#2}{#3}}}

\newcommand{\PXStfe}[1]{{TX}_{0}^{#1}}

\newcommand{\PXNtfe}[2]{\|#1\|_{\PXStfe{#2}}}

\newcommand{\XS}[1]{X_\eps^{#1}}
\newcommand{\XM}[2]{[#1]_{\XS{#2}}}
\newcommand{\XN}[2]{\|#1\|_{\XS{#2}}}
\newcommand{\XP}[3]{\langle #1, #2 \rangle_{\XS{#3}}}

\newcommand{\XSL}[2]{\XS{#1}({#2})}
\newcommand{\XML}[3]{[#1]_{\XSL{#2}{#3}}}
\newcommand{\XNL}[3]{\|#1\|_{\XSL{#2}{#3}}}

\newcommand{\XSo}[1]{\mathring X_\eps^{#1}}

\newcommand{\XSoo}[1]{\rring{X}_\eps^{\raisebox{0ex}[0ex][0ex]{\hspace{0.1ex}$\scriptstyle #1$}}}

\newcommand{\XNoo}[2]{\|#1\|_{\XSoo{#2}}}

\newcommand{\XSLo}[2]{\mathring X_\eps^{#1}{(#2)}}

\newcommand{\XSLoo}[2]{\rring{X}_\eps^{\raisebox{0ex}[0ex][0ex]{\hspace{0.1ex}$\scriptstyle #1$}}({#2})}

\newcommand{\TXS}[1]{L^2(\XS{#1})}
\newcommand{\TXM}[2]{[#1]_{\TXS{#2}}}
\newcommand{\TXN}[2]{\|#1\|_{\TXS{#2}}}

\newcommand{\TXSL}[2]{L^2(\XSL{#1}{#2})}
\newcommand{\TXNL}[3]{\|#1\|_{\TXSL{#2}{#3}}}

\newcommand{\CXS}[1]{C^0(\XS{#1})}

\newcommand{\CXN}[2]{\|#1\|_{\CXS{#2}}}

\newcommand{\CXSL}[2]{C^0(\XSL{#1}{#2})}
\newcommand{\CXNL}[3]{\|#1\|_{\CXSL{#2}{#3}}}

\newcommand{\PXS}[1]{T\XS{#1}}
\newcommand{\PXSo}[1]{T\XSo{#1}}
\newcommand{\PXSoo}[1]{T\XSoo{k}}

\newcommand{\PXN}[2]{\|#1\|_{\PXS{#2}}}

\newcommand{\XSeps}[2]{X_{#2}^{#1}}

\newcommand{\PXSeps}[2]{T\XSeps{#1}{#2}}
\newcommand{\PXNeps}[3]{\|#1\|_{\PXSeps{#2}{#3}}}

\newcommand{\YSeps}[2]{Y_{#2}^{#1}}

\newcommand{\PYSeps}[2]{T\YSeps{#1}{#2}}
\newcommand{\PYNeps}[3]{\|#1\|_{\PYSeps{#2}{#3}}}

\newcommand{\PXSL}[2]{T\XSL{#1}{#2}}

\newcommand{\PXNL}[3]{\|#1\|_{\PXSL{#2}{#3}}}

\newcommand{\PXSLoo}[2]{T\XSLoo{#1}{#2}}

\newcommand{\PXSN}[1]{TX_{\eps}^{#1} \cap L^2(X_\eps^0)}

\newcommand{\PXNN}[2]{\|#1\|_{\PXSN{#2}}}

\newcommand{\PXSNL}[2]{(TX_{\eps}^{#1} \cap L^2(X_\eps^0)({#2})}

\newcommand{\PXNNL}[3]{\|#1\|_{\PXSNL{#2}{#3}}}

\newcommand{\XSSoo}[1]{\rring{X}_\eps^{\raisebox{0ex}[0ex][0ex]{\hspace{0.1ex}$\scriptstyle #1$*}}}

\newcommand{\PXSSoo}[1]{T\XSSoo{k}}

\newcommand{\YS}[1]{Y_\eps^{#1}}
\newcommand{\YM}[2]{[#1]_{\YS{#2}}}
\newcommand{\YN}[2]{\|#1\|_{\YS{#2}}}

\newcommand{\YSL}[2]{\YS{#1}({#2})}

\newcommand{\YSoo}[1]{\rring{Y}_\eps^{\raisebox{0ex}[0ex][0ex]{\hspace{0.1ex}$\scriptstyle #1$}}}

\newcommand{\PYS}[1]{T\YS{#1}}

\newcommand{\PYN}[2]{\|#1\|_{\PYS{#2}}}

\newcommand{\PYSL}[2]{T\YSL{#1}{#2}}

\newcommand{\PYNL}[3]{\|#1\|_{\PYSL{#2}{#3}}}

\newcommand{\ZS}[1]{Z_\eps^{#1}}
\newcommand{\ZM}[2]{[#1]_{\ZS{#2}}}
\newcommand{\ZN}[2]{\|#1\|_{\ZS{#2}}}

\newcommand{\ZSL}[2]{\ZS{#1}(#2)}
\newcommand{\ZNL}[3]{\|#1\|_{\ZSL{#2}{#3}}}

\newcommand{\ZSoo}[1]{\rring{Z}_\eps^{\raisebox{-0ex}[0ex][0ex]{\hspace{0.1ex}$\scriptstyle #1$}}}
\newcommand{\ZNoo}[2]{\|#1\|_{\ZSoo{#2}}}
\newcommand{\ZMoo}[2]{[#1]_{\ZSoo{#2}}}

\newcommand{\ZSMoo}[1]{\rring{Z}_-^{\raisebox{0ex}[0ex][0ex]{\hspace{0.1ex}$\scriptstyle #1$}}}

\newcommand{\ZMMoo}[2]{[#1]_{\ZSMoo{#2}}}

\newcommand{\ZSP}[1]{Z^{#1}_+}
\newcommand{\ZNP}[2]{\|#1\|_{\ZSP{#2}}}

\newcommand{\DLL}{{\delta \mathcal L}}

\renewcommand{\medskip}{}

\newcommand{\roundup}[1]{\lceil #1 \rceil}                        % aufrunden

\newcommand{\EP}[1]{e^{\frac 12 \eps |#1| (1-v)}}
\newcommand{\EEP}[1]{K^{#1}_\eps}
\newcommand{\DI}[1]{\ d\hspace{-0.3ex}\Im (#1)}
\newcommand{\SP}{\sup_{v \in (0,1)}}
\newcommand{\lam}{\lambda}
\newcommand{\Ome}{\Omega}

\newtheorem{theorem}{Theorem}[section]

\newtheorem{proposition}[theorem]{Proposition}

\newtheorem{lemma}[theorem]{Lemma}
\newtheorem{corollary}[theorem]{Corollary}

\newcommand{\feps}{{\textstyle \frac 1\eps}}
\newcommand{\cnltL}[2]{\|#1\|_{\Re \lam = #2}}

\def \e{\eps} 
\def \l{\lambda}
\def \Lam{\Lambda}
\def \gam{\gamma}
\def \bet{\beta}

\def \EEEE{E}
\def \alp{\alpha} 
 \usepackage{hyperref}
 \hypersetup{
   colorlinks=true,
   anchorcolor=black,
   menucolor=black,
   citecolor=black,
   filecolor=black,
   linkcolor=black,
   urlcolor=black
   hyperlinkcolor=black,
 }
\numberwithin{equation}{section}

\begin{document}

\title[Well-posedness of Darcy's flow and  Lubrication approximation  ]
{Darcy's  flow  with prescribed contact angle -- Well-posedness
  and lubrication approximation}

%\titlerunning{Darcy flow}        % if too long for running head

\author{Hans Kn\"upfer} \address{Institut f\"ur Angewandte Mathematik,
  Universit\"at Bonn, Endenicher Strasse 60, 53111 Bonn, Germany, phone: +49 228 7362297} %
\email{hans.knuepfer@hcm.iam.uni-bonn.de}
\author{Nader Masmoudi} %
\address{Courant Institute, New York University, Mercer
  Street 251, New York, NY 10012, USA, phone: +1 212 998 3211}
\email{masmoudi@cims.nyu.edu}

\begin{abstract}
  We consider the spreading of a thin two-dimensional droplet on a solid
  substrate. We use a model for viscous fluids where the 
  evolution is governed by
  Darcy's Law. At the triple point where air and liquid meet the solid
  substrate, the liquid assumes a constant, non-zero contact angle ({\it partial
    wetting}). We show local and  global  well-posedness of this free boundary
  problem in the presence of the moving contact point.  Our estimates are
  uniform in the contact angle assumed by the liquid at the contact point. In
  the so-called lubrication approximation (long-wave limit) we show that the
  solutions converge to the solution  of a one-dimensional degenerate parabolic
  fourth order equation which belongs to a family of thin-film equations. The
  main technical difficulty is to describe the evolution of the non-smooth
  domain and to identify suitable spaces that capture the transition to the
  asymptotic model uniformly in the small parameter $\eps$.
\end{abstract}

\maketitle

\tableofcontents
\addtocontents{toc}{\protect\setcounter{tocdepth}{1}}.

\vspace{-9ex}
\section{Introduction and model} %\label{sec-intro}

In the past years, the theory of fluid systems in the presence of a free
boundary has been developed by many important works. Usually, in these problems,
the interface (or free boundary) separates two phases of the fluid system.
Among the large literature, such work has been addressed e.g. in
\cite{Lindblad05,Lannes05,CS07,SZ08,ZZ08,CCG11} for local existence results,
\cite{Wu11,GMS11prep} for global existence results, \cite{CCFGL12} for the study
of blow-up, \cite{AM05,MR12prep} for asymptotic limits.  In this paper, we are
interested in the situation of a fluid evolution in the presence of a contact
point where three phases meet, namely a contact point between air, liquid and
solid, see Fig. \ref{fig-darcy}. One example is the flow in a Hele-Shaw cell,
where the liquid touches the lateral boundary of the Hele-Shaw cell.  \medskip

The Hele-Shaw model describes the evolution of a liquid between the plates of a
Hele-Shaw cell.  In general, the surface tension-driven Hele-Shaw flow is given
by
\begin{align} \label{darcy-fullspace} \left \{
    \begin{aligned}
      &\Delta p \ = \ 0 \qquad\qquad && \text{in $\Omega(t)$}, \\
      &p \ = \ \gamma \kappa && \text{on $\partial \Omega(t)$}, \\
      &V \ = \ \nabla p \cdot n && \text{on $\partial \Omega(t)$},
    \end{aligned} 
  \right.
\end{align}
where the evolving domain $\Omega(t) \subset \R^2$ describes the region occupied
by the fluid. The velocity of the fluid is described by Darcy's Law $U = -
\nabla p$. In particular the normal velocity $V$ of the fluid interface is
described by $V = \nabla p \cdot n$. The parameter $\gamma$ describes the
surface tension between air and liquid. Next to its interpretation as the flow
in a Hele-Shaw cell, the fluid evolutions governed by Darcy's Law appear in a
wide range of physical models. One example is the flow of a liquid through a
porous medium, see \cite{Bear-Book}. Other situations which can be modeled by
\eqref{darcy-fullspace} are crystal growth or dissolution, directional
solidification or melting, electrochemical machining or forming
\cite{BenamarPelceTabeling-1991, Rubinstein-1971, McGeoughRassmussen-1974}. In
the last two decades, well-posedness of \eqref{darcy-fullspace} has been
investigated: Short-time existence and regularity of solutions of
\eqref{darcy-fullspace} have been proved in \cite{DuchonRobert-1984,
  KawaradaKoshigoe-1991, EscherSimonett-1997-1,EscherSimonett-1997-2} and
Prokert \cite{Prokert-1998}. Global existence for initial data close to the
sphere has been shown in \cite{ConstantinPugh-1993}. The case of zero surface
tension, $\gamma = 0$ has been considered e.g.  in
\cite{SiegelCaflishHowison-2004, Ambrose-2004}.

\medskip

Clearly, the normal component of the velocity is zero at the liquid-solid
interface.  We assume that the Hele-Shaw cell is described by the half-space $H
= \R \times (0,\infty)$. At the point where air, liquid and solid meet, we
assume that the liquid assumes a static (microscopic) contact angle. This
contact angle $\theta$ is determined by Young's Law \cite{Young-1805},
i.e. $\gam \cos \theta \ = \ \gamma_{SG}-\gamma_{SL}$ where the parameter
describe the surface tensions between the three phases: $\gamma$ (air, liquid),
$\gamma_{SL}$ (solid, liquid) and $\gamma_{SG}$ (air, solid).  This leads to the
following model:
\begin{align} \label{darcy} %
  \left \{
    \begin{aligned}
      &\Delta p \ = \ 0 \qquad && \text{in } \Omega(t), \\
      &p  \ =  \ \gam \kappa && \text{on } \partial \Omega(t)  \cap H, \\
      &p_y \ = \ 0 && \text{on } \partial \Omega(t)  \cap \partial H, \\
      &\gam \cos \theta \ = \ \gamma_{SG}-\gamma_{SL} \qquad\qquad&& \text{on
      } \partial(\partial \Omega(t) \cap \partial H),
    \end{aligned} 
  \right.
\end{align}
see Fig. \ref{fig-darcy}. The evolution then can also be interpreted as the
spreading of a droplet on a plate. A well-posedness result for \eqref{darcy} in
H\"older spaces has been given by Bazalyi and Friedman in
\cite{BazaliyFriedman-2005-1,BazaliyFriedman-2005-2}.  However, in their
analysis the conditions on the initial data are too restrictive to allow for
movement of the triple point (and thus for spreading of the droplet).  Our first
main result is a well-posedness result for this free boundary problem in a much
wider class of weighted Sobolev spaces. In particular, our result seems to be
the first result which allows for movement of the triple point.

\medskip

\begin{figure}
  \centering
  \includegraphics[width=8cm]{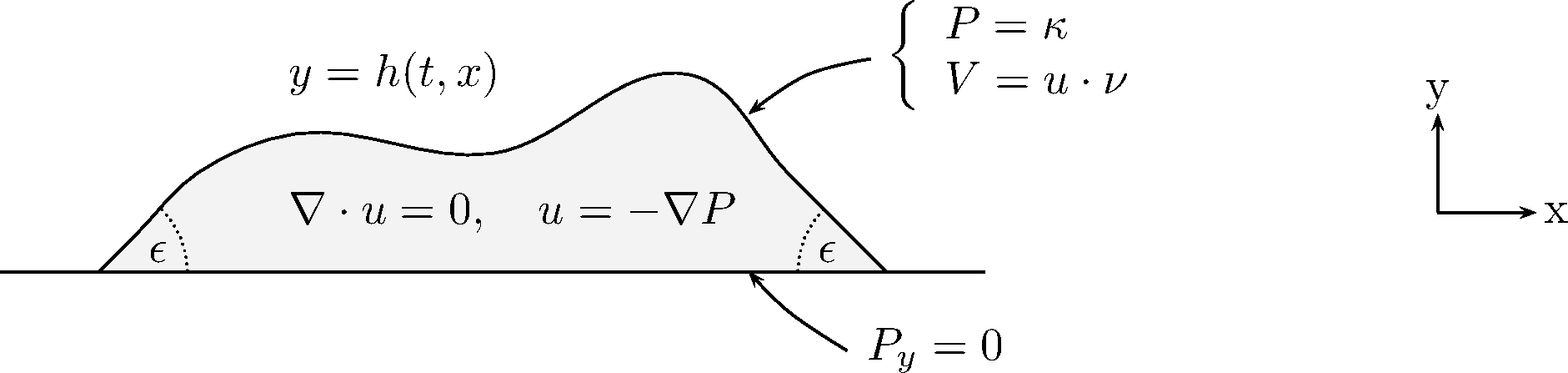}
  \vspace{-2ex}
  \caption{Darcy flow on solid substrate}
  \label{fig-darcy}
\end{figure}

The second aim of this work is to show the convergence of solutions to a reduced
model in the so called lubrication approximation regime or long wave
approximation. More precisely, let us assume that typical vertical length scales
are of order $\eps$ while horizontal length scales are of order $1$. In
particular, the angle assumed at the contact point is of order $\eps$:
\begin{align} \label{def-eps} %
 \frac{\eps^2}{2} \ \simeq \ \frac{\theta^2}{2} \ \simeq \ 1-   \cos \theta   \ %
 = \ \frac {\gamma + \gam_{SL}-\gam_{SG}}{\gam}.
\end{align}
The limit model is a special form of the  thin-film equation. Assuming that the
height of the droplet is described by the graph $h(t,x)$, the evolution is given
by
\begin{align} \label{tfe} \left \{
    \begin{aligned}
      &h_t + \gam (h h_{xxx})_x = 0 \qquad && \text{in } \{ h > 0 \}, \\
      &h = 0, \quad |h_x| = \eps, \qquad && \text{on } \partial \{ h > 0 \}, \\
      &\tilde V = \gam h_{xxx} \qquad && \text{on } \partial \{ h > 0 \}
    \end{aligned}
  \right.
\end{align}
and where $\tilde V$ is the velocity of the moving contact points $\partial \{ h
> 0 \}$. Formal derivations of lubrication models of type \eqref{tfe} have been
e.g. given in \cite{Reynolds-1886}. We prove convergence of solutions of
\eqref{darcy} to solutions of \eqref{tfe}. This is the first rigorous
lubrication approximation in the case of partial wetting (non-zero contact
angle). Furthermore, it is the first convergence result in the framework of
classical solutions. A rigorous lubrication approximation in the framework of
weak solutions has been done by Giacomelli and Otto
\cite{GiacomelliOtto-2003}. Their approach is quite different to ours: In
particular, their result does not include well-posedness for the initial
model. Instead, the authors prove convergence to the limit model by only minimal
energy bounds, assuming existence of smooth solutions. In particular, the bounds
derived in \cite{GiacomelliOtto-2003} do not capture the slope of the profile at
the contact point. Indeed, the techniques used in \cite{GiacomelliOtto-2003} do
not seem to be applicable for the case of a non-zero contact angle at the moving
contact line as considered in this work.

\medskip

The main work lies in the derivation of bounds which are uniform in the
parameter $\eps > 0$. We give a short sketch of the strategy of our proof: In
order to perform the transition from \eqref{darcy} to \eqref{tfe}, we first
express \eqref{darcy} as a nonlocal evolution problem in terms of the profile
function $h(t,x)$. The corresponding equation can be seen as a nonlocal
parabolic evolution problem of third order. Equation \eqref{tfe} on the other
hand is a local fourth order degenerate parabolic equation. As the considered
models are higher order equations, the maximum principle cannot be used. Instead
we rely on their dissipative structures. Indeed, solutions of \eqref{darcy}
satisfy
\begin{align} \label{E-darcy} %
  \ddt \big( (|\partial \Omega \cap H|) - \alpha |\partial \Omega
  \cap \partial H) \big) \ = \ - \frac 1\gam \iint_{\Omega(t)} |\nabla p|^2 \
  dxdy,
\end{align}
where $\alp = \cos (\arctan \eps)$, see e.g. \cite{GiacomelliOtto-2001}. The
dissipation relation for solutions of \eqref{tfe} is
\begin{align} \label{E-tfe} %
  \ddt \big( \int_{\R} h_x^2 \ dx + \eps^2 |\{ h > 0 \}| \big) \ = \ - \gam
  \int_\R h h_{xxx}^2 \ dx.
\end{align}
One of the core issues of  the analysis is to find suitable norms which allow
for uniform bounds in the limit $\eps \to 0$. In \cite{KnuepferMasmoudi-2012a}, we
have investigated the linearizations of \eqref{darcy} and \eqref{tfe}. The
analysis in this work suggests to use sums of weighted Sobolev norms of the 
  type
\begin{align} \label{norm-intro} %
  \XM{f}{k} \ %
  &= \ \inf_{f=f_+ + f_-} \big( \nlt{x^{k+\delta} \partial_x^{4k+\delta} f_+} +
  \nlt{{\textstyle (\frac 1\eps)^{k}} x^\delta \partial_x^{3k+\delta} f_-}
  \big).
\end{align}
where $f = h_x - 1$ and where $k, \delta \geq 0$. In the limit $\eps \to 0$,
this norm turns from a sum of weighted Sobolev norms of order $4k+\delta$ and
$3k+\delta$ to a weighted Sobolev norm of order $4k+\delta$ (first term on the
right hand side of \eqref{norm-intro}).  This transition in the character of the
norm is reflected by a transition in the character of the equation: In the limit
$\eps \to 0$, the model (understood as an evolution equation for the profile
function, cf. \eqref{darcy-4}, \eqref{tfe-2}) changes
\begin{itemize}
\item from a third order to a fourth order evolution equation and
\item from a non-degenerate parabolic to a degenerate parabolic equation.
\end{itemize}
In particular, we will use norms of type \eqref{norm-intro} with $\delta =
1$. This seems to be the smallest integer value that is sufficient to control
the nonlinearity of the problem (the norms \eqref{norm-intro} are stronger for
larger $\delta$ as follows from Hardy's inequality). We also need to choose
norms which control the pressure $p$. As we will see, the norms for the pressure
do not have a real space representation, but are rather described in terms of
the Mellin transformed function. In fact, the choice of suitable norms for the
pressure turns out to be delicate in order to obtain uniform bounds in the small
$\eps$ parameter. We use radial variables with respect to the moving contact
point at the origin of the coordinate system. The norms, we use are weighted
Sobolev norms of supremum-type in the angular direction (in real space
variables) and of $L^2$-type in radial direction (in frequency variables).

\medskip

{\it Structure of the paper:} In Section \ref{sec-setting}, we transform the
problem on a fixed domain and we define the norms to control the profile and the
pressure. The main results of this work are stated and discussed in Section
\ref{sec-result}. In Section \ref{sec-overview}, we give an overview for the
proofs of the main theorems. In Section \ref{sec-estimates}, we prove estimates
for weighted spaces. In Section \ref{sec-wedge}, we derive estimates for the
nonlinear operator for a droplet supported in half-space. In Section
\ref{sec-local}, we derive corresponding localized estimates for compactly
supported droplets.

\section{Setting and norms} \label{sec-setting} %

By a change of dependent and independent coordinates, we reformulate the problem
on a fixed domain. We then formulate \eqref{darcy} as a nonlocal evolution
equation in terms of the profile function $h$. We also introduce norms to
control profile and pressure.

\subsection{Transformation onto a fixed domain} \label{ss-trafo} %

We define the nonlinear operator $\BB^h$ of Dirichlet-Neumann type by
\begin{align} \label{B-phys} % 
  \BB^h \kappa(x)  \ = \ \sqrt{1 + h_x^2} \ \left ( \partial_{n} p
  \right)_{|y = h(t,x)},
  &&\text{where}&& %
 \left \{
    \begin{array}{ll}
      \Delta p = 0  \ &\text{in  $\Omega(t)$} , \\
      p = \gam \kappa  &\text{on  $\partial_1 \Omega(t)$},  \\
      p_y = 0  &\text{at  $\partial_0 \Omega(t)$},
    \end{array} 
  \right.
\end{align}
where $\partial_1 \Omega(t) := \partial \Ome(t) \cap \{ y > 0 \}$ and
$\partial_0 \Ome(t) := \partial \Omega(t) \cap \{ y = 0 \}$ With the assumption
that the free boundary moves with the velocity, we have $h_t = w - v h_x$, where
$u = (v,w)$ is the velocity of the liquid. Suppose that the support of the
droplet stays an interval for some time, i.e. $\supp h(t) = (s_-(t),s_+(t))$ for
$t \in (0,\tau)$.  The evolution \eqref{darcy} can then be equivalently written
as nonlocal evolution for the profile $h$ by
\begin{align} \label{darcy-2} h_t + \BB^h \kappa = 0, \qquad && \text{for } x
  \in (s_-(t), s_-(t))
\end{align}
with boundary conditions $h_{|x=s_\pm} = 0$, $|h_{x|x=s_\pm}| = \eps$ and $\dot
s_{\pm}(t) = - \kappa_{x|x=s_\pm}$. By \eqref{B-phys} and since $|h_x| = \eps$ at the
triple point, the movement of the triple point is given by
\begin{align} \label{bc-kx} %
  \partial_t s_{\pm} = - p_{x|x=s_\pm} = - \gam \kappa_{x|x=s_\pm} \ %
  = \gam \Big((1+\eps^2)^{-\frac 32} \ h_{xxx} - 3\eps(1+\eps^2)^{-\frac
      52} \ h_{xx}^2 \Big)_{|x=s_\pm}.
\end{align}
Indeed, \eqref{bc-kx} follows since $p(t,x,h(t,x) ) = \kappa(t, x) $ and hence $
p_x + h_x p_y = \kappa_x $ at $x=s(t)$ and we conclude since $p_y=0$ for
$y=0$. We will assume that the support of $h$ at initial time is given by
$(0,1)$.  We next rescale time and space to get $\OO(1)$ quantities; furthermore
we fix the position of the moving contact point $s(t)$ by using moving
coordinates: We set
\begin{align} \label{def-MD} %
  M(t) \ = \ \tfrac 12 (s_+(t) + s_-(t)), && %
  D(t) \ = \ s_+(t) - s_-(t).
\end{align}
We introduce the new variables $(\tilde x, \tilde y, \tilde t)$ by
\begin{align} \label{newcoord} %
  \tilde x \ = \ \frac {x-M(t)}{D(t)} + \frac 12, && %
  \tilde y \ = \ \frac y{\eps D(t)}, && %
  \tilde t \ = \ \eps \gam \int_0^t \frac 1{D^{2}(s)} \ ds,
\end{align}
cf. \cite{Angenent-1988-1, GiacomelliKnuepfer-2010}.  In particular, $\partial_x
\tilde x = \frac 1D$, $\partial_x \tilde t = 0$, $\partial_t \tilde t = \frac
{\eps \gam}{D^2}$ and
\begin{align*}
  \partial_t \tilde x \ %
  &= \ - \frac {M_t}D - \frac{D_t (x-M)}{D^2} %
  %\lupref{newcoord}= - \frac 1D
  %\big(M_t + (\tilde x - {\frac 12}) \ D_t \big) %
   \lupref{newcoord}= - \frac {\eps \gam }{D^3} \big(\dot M + (\tilde x -
   {\frac 12}) \ \dot D \big) 
  \lupref{def-MD}= \ - \frac {\eps \gam }{D^3} \big((1- \tilde x) \dot s_- + x
  \dot s_+ \big),
\end{align*}
where the dot denotes differentiation in $\tilde t$. The dependent quantities
$\tilde h$ and $\tilde p$ are defined by
\begin{align}
  h(t,x) = \eps D(t) \tilde h(\tilde t, \tilde x), && %
  D(t) p(t,x,y) = \eps \gam \tilde p(\tilde t, \tilde x, \tilde y).
\end{align}
The transformed evolution --- after multiplication by $D^2/\eps \gam$ --- is
given by
\begin{align} \label{def-LL} %
  \big(D \tilde h \big)_{\tilde t} - %
  \big((1- x) \dot s_- + x \dot s_+ \big) \ h_{\tilde x} %
  - \BB_\eps^{\tilde h} \big( (1 + \eps^2 {\tilde h}_{\tilde x}^2)^{-\frac 32}
  \tilde h_{\tilde x \tilde x} \big) \ = \ 0 && \text{for $x \in (0,1)$}
\end{align}
with boundary conditions $h = 0$ and $|h_x| = 1$ at $x = 0,1$. The operator
$\BB_\eps^h$ is given by
\begin{align} \label{def-bt} %
  \BB_\eps^{\tilde h} \eta (\tilde x) \lupref{B-phys}= \left( -  {\tilde h}_{\tilde x}
    \tilde p_{\tilde x} + (\feps)^2 {\tilde p}_{\tilde y} \right)_{|
    \tilde y= \tilde h(\tilde t, \tilde x)}, && %
  \text{where} && %
  \left \{
    \begin{array}{ll}
      \Delta_\eps  \tilde p = 0 \quad &\text{in } \tilde \Omega, \\
      \tilde p = \eta &\text{on } \partial_1 \tilde \Omega, \\  
      {\tilde p_{\tilde y}} = 0 &\text{on } \partial_0 \tilde \Omega
    \end{array} 
  \right.
\end{align}
and where $\Delta_\eps = \partial_{x}^2 + {\textstyle (\frac
  1\eps)^2} \partial_{y}^2$. Here, $\tilde \Ome = \{ ((\tilde x, \tilde y) :
\tilde x \in (0,1), \tilde y \in (0,\tilde h(t,x))) \}$. We will also use the
notation $\tilde \Gamma = \partial_1 \tilde \Omega$ for the air-liquid interface
and $\Gamma_0 = \partial_0 \tilde \Omega = (0,1)$ for the liquid-solid
interface. By Proposition \ref{prp-p-nonlinear}, $\BB_\eps^h$ is well-defined.
In the following, we skip the tilde's in our notation. Let $h^* = x(1-x)$ and
let $f^* = h^*_x = 1-2x$.  Note that $h^*$ is an approximation of the stationary
solution for the Darcy flow which is the half-circle. We set $f = (h - h^*)_x =
h_x - f^*$. We also use the notation $B_\eps^f := \BB_\eps^h$. This yields the
following evolution model, defined on the fixed domain $Q_{\tau} := (0,\tau)
\times (0,1)$,
\begin{align} \label{def-LL1} %
  \left \{
  \begin{aligned}
    &\LL_\eps f := (D f)_t - %
    \Big[ \big((1- x) \dot s_- +
  x \dot s_+ \big) (f+f^*) + B_\eps^f
    \Big( \frac{f_x+f_x^*}{(1 +  \eps^2 (f-2)^2)^{\frac 32}} \Big)
    \Big]_x = 0, \\% &&\hspace{-1ex}\text{in } (0,1), \\
    &f = 0 \hspace{9ex}\text{at } x = 0,1, \\
    &\int_0^1 f \ dx = 0.
  \end{aligned}
  \right.
  \hspace{-4ex}
\end{align}
The contact point positions $s_\pm$ are given by $s_-(0) = 0$, $s_+(0) = 1$
together with the ODE
\begin{align} \label{speed-local} %
  \dot s_\pm(t) \ \lupref{bc-kx}= \ \Big( (1+\eps^2)^{-\frac 32} \ f_{xx} -
  3\eps^3(1+\eps^2)^{-5/2} \ (f_x+2)^2 \Big)_{|x=0,1};
\end{align}
$D(t)$ is accordingly defined by \eqref{def-MD}. We have transformed the
equation onto a fixed domain at the cost of the non-local operator $B_\eps^f$ as
and the nonlocal terms $s_\pm$, $D$.  The analogous transformations for the
thin-film equation \eqref{tfe} yield
\begin{align} \label{tfe-fixed} %
  \left \{
    \begin{aligned}
      &\LL_0 f := (Df)_t - \Big[ \big((1 - x) \dot s_- + x \dot s_+ \big) (f +
      f^*) %
      + \Big( \Big(\int_0^x f d\tilde x \Big)  f_{xxx} \Big)_x \Big]_{x} = 0, \\
      &f = 0, \hspace{9ex} \text{at } x=0,1, \\
      &\int_0^1 f \ dx = 0.
    \end{aligned}
  \right.
\end{align}
For $\eps = 0$, the functions $\dot s_\pm$ are defined by $s_-(0)=0$ and
$s_+(0)=1$ together with the ODE $\dot s_\pm(t) = f_{xx|x=0,1}$; $D(t)$ is
defined accordingly by \eqref{def-MD}.

\medskip

Due to the degeneracy of the evolution equation at the free boundary, special
attention needs to be directed at the boundary conditions: The above
transformations are such that the boundary conditions $h = 0$, $|h_x| = \eps$
respectively, are equivalent to the integral/boundary conditions $\int f = 0$,
$f = 0$ respectively for the transformed function $f$. In the subsequent part of
this work, we construct $f$ by a Lax-Milgram argument such that $f = 0$ indeed
holds at the boundary. The integral condition $\int f dx = 0$ holds
automatically for sufficiently smooth solutions. Indeed, suppose that $\int f dx
= 0$ is satisfied initially. For the Darcy flow, with the analogous calculation
as for \eqref{bc-kx}, it follows that
\begin{align*}
  \ddt \int_0^\infty D f \ dx  \ %
  &= \ \ - \Big[ \big((1- x) \dot s_- +
      x \dot s_+ \big) (f + f^*) - B_\eps^f
  \Big( \frac{f_x-2}{(1 +  \eps^2 (f+f^*)^2)^{\frac 32}} \Big)
  \Big]_0^1 \\
  &= \ \  \dot s_+ - \dot s_- - \dot s_+ + \dot s_- \ %
  = \ 0.
\end{align*}
Therefore, $\int Df dx = \int D f_{\rm in} dx = 0$ for all times (as long as $D
\neq 0$) and hence also $\int Df dx= \int D f_{\rm in}dx = 0$. An analogous
calculation applies also for the thin-film equation. Note that instead of
proposing a fixed contact angle which implies the speed of propagation, derived
in \eqref{bc-kx}), another option would be to consider a model with dynamic
contact angle while imposing a law relating contact angle and speed of
propagation, see e.g.  \cite{RenHuE-2010}.

\medskip

{\it The case of an infinite wedge. } Near the moving contact lines, the region
occupied by the liquid approximately has the shape of a wedge. This motivates to
linearize the evolution equation around an infinite wedge. We hence assume that
$h(t,x) \approx \eps (x - s(t))$ for $s(t) \in R$. We describe how the problem
is transformed onto the wedge, a more detailed derivation is given in
\cite{KnuepferMasmoudi-2012a}. Analogously to \eqref{newcoord}, the new variables
are defined by
\begin{align} \label{coord-global} %
  x - s(t) \ = \ \tilde x, \quad y \ = \ \eps \tilde y, \quad t \ = \ {\textstyle \frac{\tilde
      t}\eps} \quad \text{ and } \quad p \ = \ \eps \tilde p, \quad h \ = \ \eps
  \tilde h.
\end{align}
Correspondingly to \eqref{def-LL}, we get
\begin{align} \label{darcy-3} %
  {\tilde h}_{\tilde t} -  \dot s \tilde h_{\tilde x} + \BB_\eps^{\tilde h}
  \big((1 + \eps^2 \tilde h_{\tilde x}^2)^{-\frac
    32} \tilde h_{\tilde x \tilde x} \big) \ = \ 0 \qquad
  && \text{for } \tilde x \in (0, \infty).
\end{align}
in the following, we omit the '$^\sim$' in the notation.  Here, $s(t)$ is
defined as in \eqref{bc-kx}. We set $f := h_x - 1$. Taking one spatial
derivative, of \eqref{darcy-3}, we get
\begin{align} \label{we-get} %
  f_t - \big[ \dot s f + \BB_\eps^h \big( f_x (1+\eps^2(1+f)^2)^{-\frac 32}
  \big) \big]_x  \ = \ 0
\end{align}
with the single boundary condition $f_{|x=0} = 0$. We introduce the notation
$K^f = \{ (x, y): x > 0, 0 < y < h \}$, $\Gamma^f = \{ (x, y): x > 0, y = h
\}$. The linear operator $B_\eps := B_\eps^0$ is given by
\begin{align} \label{def-q} %
  B_\eps \eta(t,x) \ = \ \left( - q_x + (\feps)^2 \
    q_y \right)_{|y=x},
  &&\text{where}&&
  \left \{
    \begin{array}{ll}
      \Delta_\eps  q = 0 \quad &\text{ in } K, \\
      q = \eta  &\text{ on } \partial_1 K, \\  
      q_{\tilde y} = 0 &\text{ on } \partial_0 K,
    \end{array} 
  \right.
\end{align}
where $K := K^0$ and $\Gamma := \Gamma^0$.  Equation \eqref{we-get} can then be
equivalently expressed as
\begin{align} \label{darcy-4} 
  \left \{
    \begin{aligned}
      &f_t + \ A_\eps f = N_\eps(f) && \text{for } x \in (0, \infty), \\
      &f = 0 && \text{for } x = 0.
    \end{aligned}
    \right.
\end{align}
The main (linear) part of \eqref{we-get} is given by the operator
\begin{align} \label{def-A} 
  A_\eps \ := \ - \ (1 + \eps^2)^{-\frac 32} \ \partial_x B_\eps \partial_x.
\end{align}
The remaining terms in \eqref{we-get} are combined in the nonlinear operator
$N_\eps(f) = N_\eps(f,f)$,
\begin{align} \label{def-N} %
  N_\eps(\phi, f) &= \ \frac{\phi_x}{(1+\eps^2)^{\frac 32}} \Big(\frac{f_{xx} -
    3\eps^2 f_x^2}{1+\eps^2} \Big)_{|x=0} + \partial_x B_\eps^\phi \Big(
  \frac{f_x}{(1+\eps^2(1 + f)^2)^{\frac 32}}\Big) + A_\eps f.
\end{align}
The first term on the right hand side of \eqref{def-N} is related to the
movement of the triple point. The second and third term sum describe the error
which appears by replacing the domain $K^\phi$ by $K$ and by replacing the
curvature with $f_x$. Analogously, for the thin-film equation we apply the
coordinate transform $\tilde x = x - s(t)$. In the new coordinates, we obtain
\begin{align} \label{tfe-2}
  \left \{
    \begin{aligned}
      &f_t + \ A_0 f = N_0(f) \qquad && \text{for } x \in (0, \infty), \\
      &f = 0 && \text{for } x \in (0, \infty).
    \end{aligned}
    \right.
\end{align}
The linear and nonlinear part of the equation are given by
\begin{align} \label{def-A0} 
  \begin{aligned}
    A_0 f \ := \ (x f_{xx})_{xx},  
    &&N_0(\phi, f) \ := \ - \big( {\textstyle \int_0^x} \phi \ dx \ (f_{xx} -
    f_{xx|x=0}) \big)_{xx}
  \end{aligned}
\end{align}
and $N_0(f) = N_0(f,f)$. We also write $A_0 = - \partial_x
B_0 \partial_x$ and $B_0 = - \partial_x x \partial_x$.  Observe that
$N_0(\phi, f)$ is bilinear, while $N_\eps(\phi, f)$ is neither linear
in the first nor in the second argument. Also notice that
\eqref{tfe-2} has the scaling invariance $(x,t,f) \mapsto (\lambda x,
\lambda^3 t, f)$.

\subsection{Norms for the profile} \label{ss-X}

The initial problem \eqref{darcy} is non--degenerate parabolic on a non-smooth
moving domain, the limit problem \eqref{tfe} is degenerate parabolic on a smooth
domain. We use weighted Sobolev type spaces to capture the transition between
these two problems. Weighted spaces for the analysis of elliptic operators on
non-smooth domains have e.g. been used in
\cite{KozlovMazyaRossmann-Book}. Weighted spaces have also used to analyze
degenerate parabolic equations, see e.g.
\cite{DaskalopoulosHamilton-1998,Koch-1999,GiacomelliKnuepferOtto-2008}. Our
analysis connects these two applications of weighted spaces.

\medskip

Let $\EEEE = (0,1)$, $\EEEE = (0,\infty)$ or $\EEEE = (-\infty, 0)$ and let
$d(x) = \dist(x,\partial \EEEE)$. For $k \in \N$, our norms are given as the sum
of two weighted Sobolev norms:
\begin{gather} \label{norm-integer} %
  \XML{f}{\ell}{E} \ = \ \inf_{f=f_+ + f_-}
  \big(\nltL{d^{\ell+1} \partial_x^{4\ell+1} f_+}{E} %
  + \nltL{\frac {d}{\eps^\ell} \partial_x^{3\ell+1} f_-}{E} \big).
\end{gather}
In particular, in the limit $\eps \to 0$ the homogeneous norm in
\eqref{norm-integer} turns from a norm of order $3l+1$ to a norm of order
$4l+1$.  We furthermore set
\begin{align} \label{norm0ganz} % 
  \XNL{f}{k}{E}^2 \ = \ \sum_{\ell=0}^k \XML{f}{\ell}{E}^2.
\end{align}
We recall Hardy's inequality which holds for all $\bet \neq -\frac 12$ and all
$f \in \cciL{(0,\infty)}$:
\begin{equation} \label{hardy} \nltL{x^{\beta} f}{(0,\infty)} \ \leq \
  C_\beta \nltL{x^{\beta+1} f_x}{(0,\infty)}.
\end{equation}
In particular, for fixed $\eps > 0$ the second term on the right hand side of
\eqref{norm-integer} is estimated by the first one.

\medskip
 
Our estimates require a generalization of the above norms to the case of
fractional derivatives.  This generalization will be done with help of the
Mellin transform. This transform has been widely used for elliptic boundary
problems on conical domains (e.g. \cite{KozlovMazyaRossmann-Book}). For any $f
\in \cciL{(0,\infty)}$, its Mellin transform $\hat f$ is
\begin{align} \label{mellin} %
  \hat f(\lambda) \ = \ \int_{0}^\infty x^{-\lambda} f(x) \ \frac{dx}x \ %
  = \ \int_{\R} e^{-\lambda u} F(u) \ du.
\end{align}
Here and in the following we will frequently use the variables $u = \ln x$ and
$F(u) = f(x)$. By \eqref{mellin}, application of the Mellin transform on $f$
corresponds to application of the two-sided Laplace transform on $F$.  It is
easy to see that $\widehat{x f_x}(\lambda) = \lambda \hat f(\lambda)$ and
$\widehat{x^{-\beta} f}(\lambda) = \hat f(\lambda + \beta)$ for any $\beta \in
\R$. Furthermore, Plancherel's identity holds
\begin{align} \label{plancherel} %
  \nltL{x^{-\beta} f}{(0,\infty), \frac {dx}x} \ = \ \nltL{e^{-\beta u} F}{\R, du} \
  = \ \nltL{\hat f}{\Re \lambda = \beta}.
\end{align}
The {\it strip of convergence} is the set of $\lambda \in (\beta_1,\beta_2)
\times \R \subset \C$ where the integrand in \eqref{mellin} is absolutely
convergent. Note that if $f $ is such that $x^{-\beta_1} f $ and $x^{-\beta_2}
f$ are in $L^2((0,\infty), \frac{dx}x)$ for some $\beta_1<\beta_2$, then the
strip of absolute convergence contains the interval $(\beta_1,\beta_2)$ as can
be seen by applying H\"older's inequality.  For any $\beta$ in the strip of
convergence of $f$, the inverse Mellin transform of $f$ is
\begin{align} \label{mellin-inverse} %
  f(x) \ = \ \int_{\Re \lambda = \beta} x^\lambda \hat f(\lambda) \DI{\lambda} \
  = \ \int_{\Re \lambda = \beta} e^{\lambda u} \hat F(\lambda) \DI{\lambda},
\end{align}
where the line integral is taken in direction of increasing $\Im \lambda$. The
definition \eqref{mellin-inverse} does not depend on the choice of $\beta \in
(\beta_1, \beta_2)$ since $\hat f$ is analytic in the strip of convergence.

\medskip

We are ready to give a definition of the norm in terms of the Mellin
transform. The definition of the norm by Mellin transform and the definition by
\eqref{norm-integer} differ by a constants $C_k$ for integer $k$. In our
notation we do not differentiate between the two definitions of the norms. This
does not change the result since all out estimates depend on constants $C_k$. In
order to apply the Mellin transform, we first need to subtract the boundary
data. Consider $f \in \cciL{\overline \EEEE}$, where $\EEEE = (0,1)$ or $\EEEE =
(0,\infty)$, i.e. $f$ vanishes for $x \to \infty$.  Let $\zeta \in
\cciL{[0,\infty)}$ be a smooth cut-off function satisfying $\zeta = 1$ in
$[0,\frac 18]$, $\zeta = 0$ in $[\frac 14,\infty)$ and $\nco{D^k \zeta} \leq
C_k$ for all integer $k$.  Moreover, we assume that $\zeta(x) + \zeta(1-x) = 1 $
for all $x \in [0,\infty)$.  If $E = [0,1]$ then we set $f^L(x) = f(x) \zeta(x)$
and $f^R(x) = f(1-x) \zeta(x)$, in particular $\supp f^{L}, \supp f^{R}
\subseteq [0,\frac 34]$ and $f^L(x) + f^R(1-x) = f(x)(\zeta(x) + \zeta(1-x)) =
f(x)$. We now define for $k \geq 0$
\begin{align} \label{norm-E} %
  \XNL{f}{k}{[0,1]} \ := \ \XNL{f^L}{k}{(0,\infty)} + \XNL{f^R}{k}{(0,\infty)}.
\end{align}
It remains to define the norm for $E = (0,\infty)$: Given $k \geq 0$, let $n_k$
be the largest integer smaller than $3k - \frac 12$, i.e. $n_k =
\roundup{3k-\frac 32}$.  In particular, if $k \in \N_0$ then $n_k = 3k-1$. Let
$\PP_f$ be the Taylor polynomial of order $n_k$ of $f$ at $x=0$ (if $n_k =-1$,
then we choose $\PP_f = 0$).  We decompose $f = f_1 + f_0$, where $f_1 = \zeta
\PP_f$ and define
\begin{align} \label{norm} %
  \XNL{f}{k}{(0,\infty)} \ = \ \NNN{\PP_f}{} + \nlt{f_0} + \XM{f_0}{k},
\end{align}
where $\NNN{\cdot}{}$ is any fixed polynomial norm, e.g. the $\ell^2$-norm of
the coefficients. Here, the homogeneous norm $\XM{\cdot}{k}$, $k \geq 0$, is
given by
\begin{align} \label{norm-lambda} %
  \XM{f_0}{k} \ = \ \nltL{|\lambda|^{3k+1} \mu^{k} \hat f_0}{\Re \lambda = 3k
    -\frac 12}
\end{align}
with the notation $\mu = \min \{ |\lam|, \feps \}$. The equivalence of these
norms with the characterization \eqref{norm-intro} when $k$ is an integer
follows by application of \eqref{plancherel} and by repeated application of
Hardy's inequality; the proof is given in \cite{KnuepferMasmoudi-2012a}. We define
$\XS{k}$ as completion of $\ciiL{E}$ with respect to \eqref{norm}. The notation
$\XSLo{k}{E}$ indicates that the completion is taken in the subspace of
$\ciiL{E}$ where additionally $f = 0$ on $\partial \EEEE$. Note that the trace
of $f$ is controlled in $\XSL{k}{\EEEE}$ if $k > \frac 16$, in particular
$\XS{k} \neq \XSo{k}$ for $k > \frac 16$. Finally, we define $\XSoo{k}$ as the
completion of $\cciL{E}$ with respect to \eqref{norm-E}-\eqref{norm}. We will
also use corresponding parabolic norms and spaces: Generically, the norms are
defined for $ (t,x) \in [0,\infty) \times E =: Q$. For $0 \leq \eps \leq 1$, $k
\in \N_0$, we define
\begin{align} \label{def-PXN} %
  \PXNL{f}{k}{Q}^2 %
  = \sum_{1 \leq i+j \leq k} \TXNL{\partial_t^i f}{j}{Q}^2 + \sum_{0 \leq
      i+j \leq k-1} \CXNL{\partial_t^i f}{j+\frac 12}{Q}^2 +
  \NNN{f}{C_t^0 L_x^2(Q)}^2;
\end{align}
the corresponding spaces are called
$\PXS{k}$, $\PXSo{k}$, $\PXSoo{k}$, where as before the superscript `$^{\rm o}$'
indicates that the function also vanishes at the boundary $x=0$ and the
superscript `$^{\rm oo}$' denotes the space obtained by taking the closure of
$\cciL{[0,\infty) \times (0,\infty)}$.  If the domain of integration is $Q =
[0,\infty)^2$ or $E = [0,\infty)$, we sometimes omit the domain in the notation
of space and norm, i.e. we write $\XS{k}$ for $\XSL{k}{(0,\infty)}$ etc.

\medskip

Note that the choice of norms \eqref{norm-lambda} is supported by the
investigation of the linear operator in \cite{KnuepferMasmoudi-2012a}. In
particular, by \cite[Theorem 3.2]{KnuepferMasmoudi-2012a} for $k, \geq 0$, $\eps
\in [0, \frac {\pi}{3(2k+1)})$, we have
\begin{align} \label{est-MR-hom} %
  c\XM{A_\eps f}{k} \ \leq \ \XM{f}{k+1} \ \leq \ C\XM{A_\eps f}{k}.
\end{align}

\subsection{Norms for the pressure} \label{ss-Y}

We introduce norms and spaces to control the pressure $q$
(cf. \eqref{def-q}). Unlike the spaces $\XS{k}$, which can be expressed in both
real variables and Mellin variables (cf. \eqref{norm}), our norms for the
pressure can only be expressed in terms of Mellin variables. Roughly speaking,
we apply the Mellin transform in the radial direction (with respect to the tip
of the wedge), but not in the angular variable. The norm to control the pressure
is $L^\infty$ in the radial frequency variables and $L^\infty$ in angular
variables. Using the $L^\infty$-norm in angular variables enables us to obtain
estimates which are optimal in $\eps$. A standard approach using an
$L^2(L^2)$-norm in the pressure would not capture the optimal $\eps$-dependence
which is needed for the convergence to the limit model. There are several
technical difficulties connected with the fact that we have to take the supremum
in $v$. One of them is that the norms cannot be expressed in terms of physical
variables. Moreover, complex interpolation as is possible for the trace norm
(see Lemma \ref{lem-real}) cannot be directly used for the norms for the
pressure.
 
\medskip

We first define the space $\YS{k}$ on the wedge. We introduce a coordinate
transform which maps the wedge onto an infinite strip: We define the new
variables $(t,s)$ by
\begin{align} \label{def-ts} 
  \left \{
  \begin{array}{ll}
    x = e^u \cos(\eps v), \\ 
    y = \frac 1\eps e^u \sin(\eps v), 
  \end{array}
  \right. %
  && i.e \quad
  \begin{pmatrix}
    \frac{\partial x}{\partial u} & \frac{\partial x}{\partial v} \\ 
    \frac{\partial y}{\partial u} & \frac{\partial y}{\partial v}
  \end{pmatrix}
  \ =  \ 
  \begin{pmatrix}
    e^u \cos(\eps v) & - \eps e^u \sin(\eps v) \\
    \frac 1{\eps} e^u \sin(\eps v) &  e^u \cos(\eps v)
  \end{pmatrix}.
\end{align}
For later reference, we note that $dxdy = e^{2u}du dv$ and
\begin{align} \label{eq-ts} 
    &\partial_u = x \partial_x + y \partial_y, 
    &&\feps \partial_v = - \eps y
    \partial_x + \feps x\partial_y,  \\
    &\partial_x = e^{-u} \cos(\eps v) \partial_u - \feps e^{-u} \sin(\eps
    v)
    \partial_v, &&\feps\partial_y = e^{-u} \sin(\eps v) \partial_u -
    \feps e^{-u} \cos(\eps v) \partial_v. \notag
\end{align}
The coordinate transform \eqref{def-ts} can be understood as sequence of
the two transformations $(x, \eps y) = r (\cos \theta, \sin \theta)$ and
$(r, \theta) = (e^u, \eps v)$, see Fig. \ref{fig-trafo}.
\begin{figure}
  \centering
  \begin{minipage}{0.32\linewidth}
    \centering
    \includegraphics[width=2.1cm]{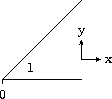}
  \end{minipage}
  \begin{minipage}{0.15\linewidth}
    \centering
    \includegraphics[width=1.7cm]{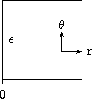}
  \end{minipage}
  \begin{minipage}{0.32\linewidth}
    \centering
    \includegraphics[width=1.7cm]{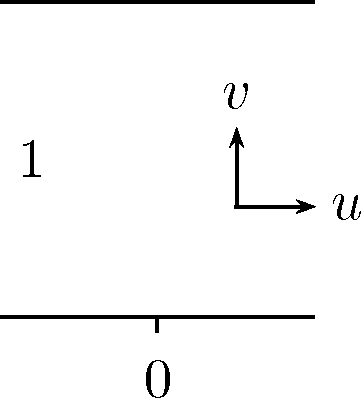}
  \end{minipage}
  \caption{Transformation from wedge to strip.}
  \label{fig-trafo}
\end{figure}
Let $\mu = \inf \{ \feps, |\lambda| \}$. Furthermore, for $q
\in \cciL{K \backslash (0,0)}$, let.  $\hat{q}$ be the Laplace transform of $q$
with respect to $u$, where $q(u,v) = q(x,y)$. Suppose that $q$ satisfies
\eqref{def-q} with $\phi = 0$ (the equation for the linearized pressure). In the
transformed variables, the equation for $\hat q$ has the form:
  \begin{align*}
    \lambda^2 \hat q(\lambda, v) + {\textstyle (\frac 1\eps)^2} \hat q_{v
      v}(\lambda,v) \ = \ 0 &&\text{ in } \R \times (0,1)
  \end{align*}
  with boundary conditions $\hat q(\lam,1) = (\lam+1) \hat f(\lam+1)$ and $\hat
  q_{v}(\lam,0) = 0$. This explicit expression of $\hat q$ in Mellin transformed
  variables motivates the definition of our norms. For any multi-index $\alpha =
  (\alp_1, \alp_2) \in \R \times \N_0$, we define $|\alpha| = \alp_1 + \alp_2$
  and $\Lambda^\alp = \lam^{\alp_1} (\frac 1\eps \partial_\nu)^{\alp_2}$.  For
  any $q_0 \in \cciL{K \backslash (0,0)}$ and for $\ell \geq 0$, we set
\begin{align} \label{Y-hom} %
  \begin{aligned}
    \YM{q_0}{\ell} \ %
    &= \ \sum_{|\alp|=3\ell} \nltL{\sup_{v \in (0,1)} \big| e^{\frac 12 \eps
        |\lambda|(1-v)} \Lambda^{\alp}\mu^\ell \widehat{q_0} \big|}{\Re \lambda
      = 3\ell-\frac 32}.
  \end{aligned}
\end{align}
For technical reasons that will be explained later, we will use these
homogeneous norms in particular for $\ell \in [\frac 23,k]$.  Suppose
that $f \in \XS{k}$ and suppose that $q$ is defined as in
\eqref{def-q} with $\eta = f_x$.  Since $q = f_x$ on $\partial_1 K$
and since by the definition of $\XS{k}$ the Taylor polynomial of $f$
of order $n_k$ is well-defined, we expect to have control on the
supremum norm for derivatives of $q$ up to order $n_k - 1$, where we
recall that $n_k = \roundup{3k-\frac 32}$.  Indeed, such an estimate
is given in Lemma \ref{lem-Y}.

\medskip

% Conversely, if all derivatives of order less or equal than $n_k-1$ of $q_0 \in
% \ciiL{K}$ vanish at $(0,0)$, then $q \in \YSoo{k}$.
For $q \in \ciiL{K}$, let $\PP_q$ be the Taylor polynomial of $q$ at $(0,0)$ of
order $n_k-1$ (if $k \in \N$ then $n_k-1 = 3k-2$). Let $\zeta : K \to \R$ be a
cut-off function such that $\zeta = \zeta(r)$ with $\zeta = 1$ in $[0,\frac
18]$, $\zeta = 0$ outside $[0,\frac 14]$ and such that $0 \leq \zeta \leq 1$. We
decompose $q := q_0 + q_1$ with $q_1 = \zeta \PP_q$ and define for $ k,\ell \in
\R$, the norm
\begin{align} \label{Y-norm} %
  \YN{q}{k} \ %
  = \ \NNN{\PP_q}{\PP} + \sup_{\frac 23 \leq \ell \leq k}
  \YM{q_0}{\ell}. %
\end{align}
Here, $\NNN{\cdot}{\PP}$ is any fixed polynomial norm, e.g. the
$\ell^2$-norm of the coefficients. We do not include the homogeneous
norms with $\ell < \frac 23$ since this would lead to 'negative
derivatives' in the definition of the following norm,
cf. \eqref{Z-norm}. The spaces $\YS{k}$ and $\YSoo{k}$ are defined by
completion with respect to functions $\cciL{K}$ respectively $\cciL{K
  \backslash (0,0)}$ as before.  By the above considerations, the
polynomial is uniquely defined and the norm is well--defined. The
space $\YSL{k}{\Omega}$ and its norm are defined by localizing the
above definitions (analogously as for the definition of
$\XSL{k}{[0,1]}$). The space $\PYS{k}$ and its norm are defined
analogously to $\PXS{k}$, i.e.
    \begin{align} \label{def-PYN} %
      \PYNL{p}{k}{Q}^2 %
      = \sum_{1 \leq i+j \leq k} \NNN{\partial_t^i
        p}{L^2(\YSL{j}{Q})}^2.
\end{align}
Let us remark that we believe that all homogeneous norms in \eqref{nochwas} for
all real $\ell \in [\frac 23, k]$ can be bounded by the two extremal homogeneous
norms ($\ell = \frac 23$, $\ell = k$). However, the proof of this interpolation
inequality does not seem to be straightforward, in particular since the
analyticity of the expressions is destroyed by the supremum in $v$ so that the
theory of complex interpolation does not seem to apply directly. This is the
reason, why we include the information about all the intermediate homogeneous
norms into the definition \eqref{Y-norm}. This is not necessary in the
definition of the norm for the space $\XS{k}$ in \eqref{norm} since there we
have an interpolation result at hand (see Lemma \ref{lem-real}).

\medskip

The space $\ZS{k}$ describes the regularity of functions $g = \Delta_\eps q$ for
$q \in \YS{k}$. In view of \eqref{Y-hom}, this suggests to consider for real
$\ell \in [\frac 23, k]$, the homogeneous norms of type
 \begin{align} %
   \ZM{g_0}{\ell} \ %
   &= \ \sum_{|\alp| = 3\ell-2} \nltL{\sup_{v \in (0,1)} e^{\frac 12 \eps
       |\lambda|(1-v)} \Lambda^{\alpha} \mu^\ell |\widehat{g_0}|}{\Re \lambda =
     3\ell-\frac 72}, \label{Z-hom}
 \end{align}
 where the sums are taken over $\alpha = (\alp_1,\alp_2) \in \R_+ \times \N_0$
 using the notation $|\alp| = \alp_1 + \alp_2$.  Note that $\ZMoo{g_0}{\ell} <
 \infty$ implies $\NNN{g_0}{L_y^\infty} = o(x^{(3\ell-7/2)/2})$ for all $g_0 \in
 \ZSoo{\ell}$. Correspondingly, we say $g \in \ZS{k}$ if there is a polynomial
 $\PP_g$ in $x,y$ of order $n_k - 3$ (if $k \in \N$ then $n_k-3 = 3k-4$) such
 that $g_0 := g - \zeta \PP_g \in \ZSoo{k}$ for some radial cut-off $\zeta =
 \zeta(r)$ with $\zeta = 1$ in $[0,1]$ and $\zeta = 0$ in $[2,\infty)$. The
 corresponding norm is given by
 \begin{align} \label{Z-norm} %
   \ZN{g}{k} \ = \ \NNN{\PP_g}{\PP} + \sup_{\frac 23 \leq \ell \leq k}
   \ZM{g_0}{\ell}.
\end{align}
The spaces $\ZSoo{k}$ and $\ZS{k}$ are defined by completion as
before.  The corresponding space $\ZSL{k}{\Omega}$ for the droplet
case and its norm are defined analogously as before by localizing the
above definitions. Note that the minimal value $\ell = \frac 23$ in
\eqref{Z-norm} is chosen such that the exponent $|\alp| = 3\ell-2$ in
definition \eqref{Z-hom} stays non-negative.

\subsection{Compatibility conditions}

Higher regularity for our solution requires compatibility conditions on the
initial data (for both Darcy flow and thin-film equation): Indeed, let $f \in
\PXS{k}$ be a solution of \eqref{darcy-4}, respectively \eqref{tfe-2}. Since $f
= 0$ at $x=0$, it follows that $\partial_t^k f_{|x=0}=0$. This translates to a
compatibility condition for the initial data.  It is obtained by consecutively
replacing the time derivatives in $\partial_t^k f_{|x=0}$ by the spatial
operators $\NN_\eps$ and $A_\eps$ using \eqref{darcy-4} resp.
\eqref{tfe-2}. The corresponding condition needs to be satisfied for the initial
data:
\begin{align} \label{compatibility} 
  f_{\rm in} \text{ satisfies compatibility conditions ensuring $\partial_t^l f_{|x=0} = 0$ at $t=0$}
\end{align}
for all $ 0 \leq l \leq k $.  For example, the condition for $k=1$ corresponds
to $N_\eps(f_{\rm in},f_{\rm in}) = A_\eps f_{\rm in}$. An analogous
compatibility condition needs to be satisfied for the linear evolution $f_t +
A_\eps f = g$.  In this case, we need $f_{\rm in}$ and $g $ to satisfy for all $
0 \leq l \leq k$
\begin{align} \label{linear-compatibility} f_{\rm in} \text{ and } g \text{
    satisfy compatibility conditions  ensuring $\partial_t^l f_{|x=0} = 0$ at
    $t=0$.}
\end{align}

\section{Statement and discussion of  the  results} \label{sec-result}

\subsection{Statement of results}

We have the following main results: We have well-posedness for the Darcy flow
with moving contact line. In the regime of lubrication approximation, we have
convergence of the solutions towards solutions of the thin-film
equation. Furthermore as a consequence, we obtain well-posedness for the
thin-film equation.

\begin{theorem}[Darcy flow] \label{thm-wedge} Let $k \geq 1$ be an  integer, $\eps
  \in (0, \frac {\pi}{3(2k+1)})$. Suppose that $f_{\rm in}^\eps \in \XS{k+1/2}$
  satisfies \eqref{compatibility} and suppose that $\XN{f_{\rm in}^\eps}{k+1/2}
  \leq \alpha_k$ for some (small) universal constant $\alpha_k$. Then there is a
  unique global in time solution $f^\eps \in \PXS{k}$ of \eqref{darcy-4}$_\eps$
  with initial data $f_{\rm in}^\eps$.  Furthermore,
  \begin{align} \label{est-wedge} %
    \PXN{f^\eps}{k+1} + \PYN{\overline p^\eps}{k+1} \ \leq \ C_k \XN{f_{\rm
        in}^\eps}{k+1/2},
  \end{align}
  where $\overline p^\eps = p^\eps \circ \Psi^\eps \in \YS{k}$ and $\Psi^\eps: K
  \to K^{f^{\eps}}$ is the coordinate transform in Lemma \ref{lem-psi}; see also
\eqref{norm}, \eqref{Y-hom} for the definition of the norms. The constant in
\eqref{est-wedge} is universal, in particular it does not depend on $\eps$.
\end{theorem}

The above well-posedness result can also be stated for the Darcy flow in terms
of the original variables: Suppose that the assumptions of Theorem
\ref{thm-wedge} hold. Then there exists a unique classical solution of
\eqref{darcy}.  In particular, if $f_{\rm in}^\eps$ is sufficiently small then
$h_{\rm in}^\eps$ defined by $\partial_x h_{\rm in}^\eps = 1 + f_{\rm in}^\eps$
satisfies $h_{\rm in}^\eps > 0$ for $x>0$.

\medskip

We also have convergence for solutions of the Darcy flow to solutions of the
thin-film equation. Furthermore, as suggested by the asymptotic expansion, in
the limit $\eps \to 0$, the pressure $p$ is independent of the vertical
direction:

\begin{theorem}[Convergence] \label{thm-convergence} %
  Suppose that the assumptions of Theorem \ref{thm-wedge} are satisfied.  Let
  $f^\eps$ be the solution of \eqref{darcy-4}$_\eps$ with initial data $f_{\rm
    in}^\eps$ and let $p^\eps$ be the corresponding pressure. Suppose that
  $\XN{f_{\rm in}^\eps - f_{\rm in}}{k+1/2} \to 0$ as $\eps \to 0$ for some
  $f_{\rm in} \in \XStfe{k+1/2}$. Then there exist $f$, $p$ and a
  subsequence $\eps_j \to 0$ such that
  \begin{align} \label{eq-convergence} %
    \PXN{f^{\eps_j} -f}{k} \to 0 \quad %
    \text{and} \quad %
    \PYN{\overline p^{\eps_j} - \overline p}{k} \to 0 && %
    \text{as $j \to \infty$,}
  \end{align} 
  where $\overline p^\eps = p^\eps \circ \Psi^\eps \in \YS{k}$, with $\Psi^\eps:
  K \to K^{f^{\eps}}$ defined in \eqref{psi-form}. Furthermore, $f \in
  \PXStfe{k+1}$ solves \eqref{tfe-2} with initial data $f_{\rm in}$. The limit
  pressure $p$ does not depend on the vertical direction, i.e. $p = p(t,x)$.
\end{theorem} 
As a consequence of Theorem \ref{thm-convergence}, the velocity field $U =
(V,W)$ in the limit $\eps = 0$ is horizontal and does not depend on $y$, i.e. $U
= (V(t,x),0)$. By Theorem \ref{thm-convergence} and by Proposition
\ref{prp-X}(2), the solutions converge also in terms of Sobolev norms. We
e.g. have
\begin{align*}
  \NNN{f^\eps - f}{L_t^2 C_x^{3k-1}(\R_+^2)} \to 0, %
  \quad \NNN{\nabla \overline p^\eps - \nabla \overline p}{L_t^\infty
    C_x^{3k-2}(\R_+ \times K)} \to 0 && %
  \text{ as $\eps \to 0$.}
\end{align*} 
For the case of a droplet as initial data, we have short-time
existence:
\begin{theorem}[Droplet] \label{thm-local} %
  Let $k \geq 1$ be an integer, $\EEEE = (0,1)$. Suppose that $h_{\rm
    in}^\eps \in H^1(\EEEE)$ with $h_{\rm in}^\eps > 0$ in $(-1,1)$
  and with $h_{\rm in}^\eps = 0$, $|h_{\rm in,x}^\eps| = 1$ on
  $\partial \EEEE$.  Suppose that $f_{\rm in}^\eps := [h_{\rm in}^\eps
  - \frac 12 (1-x^2)]_x \in \XSL{k+1/2}{\EEEE}$ satisfies the
  compatibility condition \eqref{compatibility}. Then there is a time
  $\tau > 0$ such that for every $\eps \in (0, \frac {\pi}{3(2k+1)})$,
  there is a unique short--time solution $h^{\eps}$ of \eqref{darcy}
  with initial data $h_{\rm in}^\eps$ (where $h^\eps$ describes the
  profile of the propagating liquid). Furthermore, $f^{\eps} :=
  [h^{\eps} - \frac 12 (1-x^2)]_x \in \PXSL{k+1}{(0,\tau) \times
    \EEEE}$ satisfies
  \begin{align} \label{est-local} %
    \PXNL{f^\eps}{k+1}{(0,\tau) \times \EEEE} \ \leq \ C_k \XNL{f^\eps_{\rm
        in}}{k+1/2}{\EEEE}.
  \end{align}
  The solution depends continuously on the initial data.  Furthermore
  $\PXN{f^{\eps_j} - f}{k} \to 0$ as $j \to \infty$ for a subsequence $\eps_j
  \to 0$ and $f$ solves \eqref{tfe-2}.
\end{theorem}
Theorem \ref{thm-local} also shows that any solution immediately assumes a
  regularity of order $\frac 1\eps$ where we recall that $\eps$ is related to
  the opening angle. This is the maximal regularity which can be expected in a
  non-smooth domain with opening angle of order $\eps$, see
  \cite{KozlovMazyaRossmann-Book}.
  \begin{corollary} \label{cor-local} %
    Any solution of $f$ as in Theorem \ref{thm-local} satisfies $f \in
    \XSL{K+\frac 12}{\EEEE}$ for any fixed positive time, where $K$ is the
    largest integer such that $K < \frac 12(\frac{\pi}{3\eps} - 1)$.
  \end{corollary}
  Indeed, this follows by a bootstrap argument using
  \eqref{est-local}: If $f_{\rm in } \in \XSL{k+1/2}{\EEEE}$, then we
  also have $\TXSL{k+1}{(0,\tau) \times \EEEE}$ which yields $f \in
  \XSL{k+1}{\EEEE}$ for almost every fixed positive time $t_0 >
  0$. Application of Theorem \ref{thm-local} then implies $f \in
  \TXSL{k+3/2}{(t_0,\tau) \times \EEEE}$ for all $t > t_0$. Now, for
  any $\delta$, we may repeatedly apply this argument for time steps
  of size $\frac \delta K$ which yields the assertion of Corollary
  \ref{cor-local}.

\medskip

Our analysis also yields the following new existence, uniqueness and regularity
result for classical solutions of the thin-film equation:
\begin{theorem}[Thin-film equation] \label{thm-tfe} %
  Let $k \geq 1$ be integer.
 \begin{enumerate}
 \item There is $\alpha_k > 0$ such that for any $f_{\rm in} \in \XStfe{k+1/2}$
   with $\XNtfe{f_{\rm in}}{k+1/2} \leq \alpha_k$ and such that
   \eqref{compatibility} is satisfied, there is a unique global in time solution
   $f \in \PXStfe{k}$ of \eqref{tfe-2} with initial data $f_{\rm
     in}$. Furthermore, $\PXNtfe{f}{k} \leq C_k \XNtfe{f_{\rm in}}{k+1/2}$.
 \item Let $f_{\rm in} \in \XStfe{k+1/2}$ and suppose that the analogous
   assumptions as in Theorem \ref{thm-local} holds.  Then there is a short time
   solution $f \in \PXStfe{k+1}$ of \eqref{tfe-2}. Furthermore,
   $\PXNLtfe{f_\eps}{k+1}{(0,\tau) \times \EEEE} \leq C_k \XNLtfe{f_{\rm
       in}}{k+1/2}{\EEEE}$.
  \end{enumerate}
\end{theorem}
Note that the 
 existence for weak solutions of the thin-film equation \eqref{tfe} in
the complete wetting regime (zero contact angle) is well understood, see e.g.
\cite{BerettaBertschDalPasso-1995,BertozziPugh-1996}; uniqueness of weak
solutions is still an open problem. Existence and uniqueness of classical
solutions has been shown in
\cite{GiacomelliKnuepferOtto-2008,GiacomelliKnuepfer-2010}. All the above
results address the case of complete wetting where the liquid attains a zero
contact angle at the triple point.  There are only few results for the partial
wetting regime where existence (but not uniqueness) of weak solutions is proved
\cite{Otto-1998,BertschGiacomelliKarali-2005}. Well-posedness for classical
solutions with non-zero contact angle for a related model has been shown in
\cite{Knuepfer-2011}.  In this paper, we give the first existence and uniqueness
result for \eqref{tfe} in the partial wetting regime. We hope that the
techniques developed in this paper can also be applied to more complicated
systems such as the Stokes flow with various boundary conditions at the
liquid-solid interface or for fluid models where the contact angle condition at
the triple point is different.

\medskip

In the following, we do not explicitly write $k$-dependence of
constants, i.e. we write $C = C_k$.

\subsection{Formal lubrication approximation} \label{ss-formal}

We formally show how the Darcy flow \eqref{darcy-4} converges to the thin-film
equation \eqref{tfe-2}. For this, we show convergence of both linear and
nonlinear operator in \eqref{darcy-4}, i.e. $A_\eps \to A_0$ and $N_\eps \to
N_0$ as $\eps \to 0$. The argument is based on an asymptotic expansion of the
($\eps$-dependent) pressure $p^\eps$ in $\eps y$ (cf. \eqref{def-bt}):
\begin{align*}
  p^{\eps}(t,x,y) = p_0(t,x) + \eps^2 p_2(t,x,{\textstyle \frac y{\eps}}) + \OO(\eps^4).
\end{align*}
Our aim is to solve \eqref{def-bt} up to first order in $\eps$, i.e.
\begin{align} \label{p-formal} 
  \left \{
    \begin{array}{ll}
      \partial_x^2 p^{\eps} + (\feps \partial_y)^2 p^{\eps} \ = \ \OO(\eps^2) \qquad &\text{ in } \Omega, \\
      p^{\eps}_{|\Gamma} \ = \ f_x, \quad p^{\eps}_{y|\Gamma_0} \ = \ 0.
    \end{array} 
  \right.
\end{align}
Indeed, the solution of \eqref{p-formal} has the asymptotic expansion
\begin{align} \notag % \label{p-expansion} 
  p^{\eps}(t,x,y) \ = \ f_x - {\textstyle \frac{\eps^2}2} \left (y^2 - h^2 \right) f_{xxx} +
  \OO(\eps^4).
\end{align}
Inserting this asymptotic expansion into \eqref{def-bt}, we obtain
\begin{align} \label{formal-DN} 
  B_\eps^\phi f_x(t,x) \ = \ \BB_\eps^h f_x(t,x) \ \lupref{def-bt}= \ - h_x f_{xx} - h f_{xxx} + \OO(\eps^2) \ = \ -
  (h f_{xx})_x + \OO(\eps^2).
\end{align}
where $h = x + \int_0^x \phi$.  The asymptotic expression of the linear operator
$A_\eps$ follows as a special case of \eqref{formal-DN} by setting $\phi = 0$ or
equivalently $h = x$:
\begin{align*}
  A_\eps f \ \lupref{def-A}= \ - (B_\eps f_x)_x \ \lupref{formal-DN}= \ (x
  f_{xx})_{xx} + \OO(\eps^2) \ \lupref{def-A0}= \ A_0 f + \OO(\eps^2),
\end{align*}
which implies $A_\eps \to A_0$.  The convergence $N_\eps \to N_0$ can be seen
similarly: With the notation $\Phi = \int_0^x \phi \ dx = h-x$, we have
\begin{align*}
  N_\eps(\phi, f) \ \ \ &\upref{def-N}= \ \phi_x f_{xx|x=0}
  + (B_\eps^\phi f_x)_x - (B_\eps f_x)_x + \OO(\eps^2) \\
  &\upref{formal-DN}= \ \Phi_{xx} f_{xx|x=0} - (h f_{xx})_{xx}
  + (x f_{xx})_{xx} + \OO(\eps^2) \\
  &= \   (\Phi f_{xx|x=0})_{xx} - (\Phi f_{xx})_{xx} + \OO(\eps^2) \ \ \ %
  \upref{def-A0}= \ N_0(\phi, f) + \OO(\eps^2)
\end{align*}
which formally proves the convergence $N_\eps \to N_0$.

\section{Proof of Theorems \ref{thm-wedge}--\ref{thm-tfe}} \label{sec-overview} %

We give the proof of the Theorems \ref{thm-wedge}-\ref{thm-tfe} already in this
section so that the reader may get an overview of the structure of the
proof. Some parts of the proof are based on results and estimates which are
given in detail in the later part of this work.

\subsection{Proof of Theorem \ref{thm-wedge}}

The proof of Theorem \ref{thm-wedge} proceeds by an application of a contraction
principle. It is based on maximal regularity for the linear operator $A_\eps$
and corresponding bounds for the operator $N_\eps$. In the following we drop the
superscript $\eps$ in the notation if the meaning is clear from the context,
e.g. $f = f^\eps$ and correspondingly for the other functions used. The maximal
regularity estimate is the following:
\begin{proposition} \label{prp-linear} %
  Let $k \geq 1$ be an integer, $\eps \in (0, \frac
  {\pi}{3(2k+1)})$. Suppose that $f_{\rm in}^\eps \in \XS{k+1/2}$ and
  $g \in \PXS{k-1}$ satisfy \eqref{linear-compatibility}. Then there
  is a unique global in time solution $f^\eps \in \PXS{k}$ of
    \begin{align} \label{darcy-4-lin} 
      \left \{
        \begin{aligned}
          &f_t + \ A_\eps f = g && \text{for } x \in (0, \infty), \\
          &f = 0 && \text{for } x = 0.
        \end{aligned}
      \right.
    \end{align}
  with initial data $f_{\rm in}$.
  Furthermore, for some uniform constant $C > 0$ and all $\tau > 0$
  and with $Q_\tau = (0,\tau) \times (0,\infty)$, we have
  \begin{align} \label{l-2} %
    \PXNL{f}{k+1}{Q_\tau} \ \leq \ C \ \big(\PXNL{g}{k}{Q_\tau} +
    \TXNL{g}{0}{Q_\tau} + \XNL{f_{\rm in}}{k+1/2}{(0,\infty)} +
    \NNN{f}{C_t^0 L_x^2(Q_\tau))}\big)
    \end{align}
\end{proposition}
The estimate on the nonlinear operator is stated in the following proposition:
\begin{proposition} \label{prp-nonlinear} %
  Suppose that the assumptions of Theorem \ref{thm-wedge} are satisfied. Suppose
  for $i=1,2$ that $ f_i \in \PXS{k+1}$ with $\PXN{f_i}{k+1}$. Then
  \begin{align} \label{nl-1} \hspace{3ex} & \hspace{-3ex} %
    \PXNL{N_\eps(f_1, f_1) - N_\eps(f_2, f_2)}{k}{Q_\tau} +   \TXNL{N_\eps(f_1, f_1) - N_\eps(f_2, f_2)}{0}{Q_\tau} \ \notag \\
    &\leq \ C \PXNL{f_1 - f_2}{k+1}{Q_\tau} (\PXNL{f_1}{k+1}{Q_\tau} +
    \PXNL{f_2}{k+1}{Q_\tau} \big).
   \end{align}
\end{proposition}
Proposition \ref{prp-linear} has been derived in
\cite{KnuepferMasmoudi-2012a}.  Note the extra term $ \NNN{f}{C_t^0
  L_x^2(Q_\tau))} $ on the right-hand side of \eqref{l-2} which is due
to the slightly different definition of the norms $
\PXNL{f}{k+1}{Q_\tau} $. This extra term is needed for the nonlinear
estimate \eqref{nl-1}. The proof of Proposition \ref{prp-nonlinear} is
shown in Section \ref{sec-wedge}.  The estimates in the above two
propositions hold for every chosen time interval. The constants do not
depend on this interval.

\medskip

The proof of Theorem \ref{thm-wedge} now follows by application of a contraction
argument: For $\delta > 0$ to be fixed later, we set
\begin{align}
  E \ := \ \{ f \in \PXS{k} \ : \ \PXN{f}{k} \ < \ \delta \}.
\end{align}
The operator $S_\eps(f)$ is defined for any $f \in E$ as the solution
of \eqref{darcy-4} with fixed initial data $f_{\rm in}$ and right hand
side $N_\eps (f,f)$. Let $f_1, f_2 \in E$, with the same initial data
and let $f := f_1 - f_2$. In particular, $S_\eps(f_1) - S_\eps(f_2) $
solves \eqref{darcy-4} with vanishing initial data and right hand side
$N_\eps(f_1,f_1) - N_\eps(f_2,f_2)$.  By standard interpolation, we
have for any $f \in E$,
\begin{align} \label{first-try} %
  \NNN{f}{C_t^0 L_x^2(Q_\tau)} \ \leq \ \nltL{f_{\rm in}}{(0,\infty)} %
  + C \tau \PXNL{f}{k+1}{Q_\tau}.
  \end{align}
  Hence, there is a small but universal constant $\tau$ such that for
  $\tau = \tau_0$, we can absorb the last term on the right hand side
  of \eqref{l-2} (increasing the constant in the estimate by some
  universal factor) to get
  \begin{align} \label{l-2b} %
    \PXNL{f}{k+1}{Q_\tau} \ \leq \ C \ \big(\PXNNL{g}{k}{Q_\tau} + \XNL{f_{\rm
        in}}{k+1/2}{(0,\infty)} \big),
  \end{align}
  where we use the notation $\PXNN{\cdot}{} := \PXN{\cdot}{} + \TXN{\cdot}{0}$.
  By \eqref{l-2b} and in view of \eqref{nl-1}, we hence get
 \begin{align*}
   \hspace{6ex} & \hspace{-6ex} %
   \PXN{S_\eps(f_1) - S_\eps(f_2)}{k+1} \
   \upref{l-2b}\leq \ C \PXNN{N_\eps(f_1, f_1) - N_\eps(f_2, f_2)}{k} \\
   &\upref{nl-1}\leq \ C \PXN{f_1 - f_2}{k+1} \ \big(\PXN{f_1}{k+1} + \PXN{f_2}{k+1} \big) \\
   &\leq \ C \delta \PXN{f_1 - f_2}{k+1}.
 \end{align*}
 Hence, $S_\eps$ is a contraction if $\delta > 0$ and $\tau$ are  chosen sufficiently
 small. Similarly, by \eqref{l-2b}, \eqref{nl-1} and since $N_\eps(0) = 0$, we
 get
\begin{align*}
  \PXN{S_\eps(f)}{k+1} \ \lupref{l-2b}\leq \ C \PXNN{N_\eps(f)}{k} + C
  \XN{f_{\rm in}}{k+1/2} \ \upref{nl-1}\leq \ C\alpha + C\delta^2
\end{align*}
and hence $S_\eps(E) \subseteq E$ for $\delta$ and $\alpha =
\alpha(\delta)$ sufficiently small. Therefore, application of Banach's
Fix-point Theorem yields existence and uniqueness of a solution of
\eqref{darcy-4} on the time interval $(0,1)$. In order to recover
long-time existence, we use dissipation of energy
\eqref{E-darcy}. Indeed, by \eqref{E-darcy}, we have
  \begin{align}
    \NNN{f}{C_t^0 L_x^2(Q_\tau)} \ \leq \ C \nltL{f_{\rm in}}{(0,\infty)} \ \leq
    \ C \XN{f_{\rm in}}{k+1/2}.
  \end{align}
  Using this estimate instead of \eqref{first-try}, we get estimates that are
  independent of the time interval. This shows long-time existence and also the
  estimate of $f$ in \eqref{est-wedge}.  In order to conclude the proof of
  Theorem \ref{thm-wedge}, it remains to prove the uniform bound
  \eqref{est-wedge} on $p$.  This estimate follows from the estimate in
  Proposition \ref{prp-p-nonlinear}.

\subsection{Proof of Theorem \ref{thm-convergence}} 

By Theorem \ref{thm-wedge}, we have the uniform bound $\PXN{f^\eps}{k+1} \leq C
\alpha$ for all $\eps > 0$. We use the optimal decomposition into high and low
frequencies $f^\eps = f^\eps_+ + f^\eps_-$ from \eqref{def-fpm}. By
\eqref{def-fpm}, this decomposition commutes with the time derivative. We have
uniform bounds on the norms
$\nltL{x^{2j+1} \partial_t^i \partial_x^{4j+1}f^\eps_+}{(0,\infty)^2}$ and
$\nltL{x \partial_t^i \partial_x^{3j+1}f^\eps_-}{(0,\infty)^2}$ for all $i$, $j$
with $i+j \leq k+1$. Standard compactness applied to both $f^\eps_\pm$ then show
that in particular there is $f = f^0_+ + f^0_- \in \PXStfe{k}$ and a subsequence
$\eps_j \to 0$ such that
\begin{align} \label{conv-1} %
    \PXNeps{f^{\eps_j} -f}{k}{\eps_j} \to 0 && %
    \text{as $j \to \infty$.}  
\end{align}
Now, let $p^{\eps_i}$, $p^{\eps_j}$ be the pressure related to $f^{\eps_i}$ and
$f^{\eps_j}$. By \eqref{ppf-2}, we then have
\begin{align*} % 
  \PYNeps{p^{\eps_i} - p^{\eps_j}}{k}{\eps_j} \ \leq \ C \PXNeps{f^{\eps_i} -
    f^{\eps_j}}{k}{\eps_j} (\PXNeps{f^{\eps_i}}{k}{\eps_j} +
  \PXNeps{f^{\eps_j}}{k}{\eps_j}) \ %
  \lupref{conv-1}\leq \ C\alpha \PXNeps{f^{\eps_i} - f^{\eps_j}}{k}{\eps_j} \to
  0
\end{align*}
as $i,j \to \infty$, $j \leq i$. This shows that $p^{\eps_j}$ converges in
$\PYS{k}$ thus concluding the proof of \eqref{eq-convergence}.

\medskip

It remains to show that $f$ solves the thin-film equation \eqref{tfe-2}.  By the
above uniform bounds, by \eqref{X-sob} and by \eqref{Y-sob}, we have $f^\eps \to
f$ in $L^2(H^4)$, $f^\eps_t \to f_t$ in $L^2(L^2)$ and $p^\eps \to p$ in
$L^2(H^1)$ with $\NNN{p_y^\eps}{L^2(L^2)} \leq C\eps \to 0$.  The boundary
condition $p^\eps = h_{xx}(1+\eps^2 h_x^2)^{-\frac 32}$ hence implies that in
the limit $\eps = 0$, we get $p = h_{xx}$.  We will show that in the limit
$\eps=0$, $h$ is a solution of the thin-film equation, where we recall that
$f^\eps = h^\eps_{x} - 1$. In particular, $h^\eps \to h$ in $L^2(H^4)$ and
$h_t^\eps \to h_t$ in $L^2(L^2)$. Correspondingly, we also have convergence of
the velocity $u^\eps = (v^\eps, w^\eps) = \nabla_\eps p^\eps \to u = (v, w) =
(p_x, 0)$ in $L^2(L^2)$. The transition to the limit now can be conveniently
done in terms of the continuity equation: By conservation of mass for
\eqref{darcy-4}, we have
\begin{align} \label{mass-eps} %
  h^\eps_t + \Big( \int_0^{h^\eps(t,x)} v^\eps \ d\tilde x \Big)_x \ =
  \ 0.
\end{align}
In the limit $\eps \to 0$ and in view of the above discussion, \eqref{mass-eps}
turns into
\begin{align} \label{mass-tfe} %
  h_t + \Big( \int_0^{h(t,x)} v \ d\tilde x \Big)_x \ = \ h_t + (h
  p_x)_x \ = \ h_t + (h h_{xxx})_x \ = \ 0.
\end{align}

\subsection{Proof of Theorem \ref{thm-local}} 

We prove this theorem by an application of the Inverse Function Theorem.  For
this, we linearize $ {\mathcal L}_\eps$ at an `approximate solution' $w$,
constructed with the help of an extension lemma. We then show boundedness and
differentiability for $ \LL_\eps$ and invertibility and maximal regularity for
its linearization $\delta \LL_\eps^w$ at $w$.  We keep the details
brief and refer to similar arguments in
\cite{DaskalopoulosHamilton-1998,GiacomelliKnuepfer-2010,Angenent-1988-1}.

\medskip

We define the `spatial part' $\AA_\eps$ of the operator $ \LL_\eps$ by $\LL_\eps
f= D f_t + \AA_\eps f$, i.e.
\begin{equation*} % \label{def-A2}
  \AA_\eps f \ := \ %
  D_t f - %
  \partial_x \Big[ \big((1- x) \dot s_- +
  x \dot s_+ \big) (f+(x(1-x))_x) + B_\eps^f
  \Big( \frac{f_x-2}{(1 +  \eps^2 (f+(x(1-x))_x)^2)^{\frac 32}} \Big)
  \Big],
\end{equation*}
see \eqref{def-LL1}.  We will use the following extension Lemma
\begin{lemma}[Extension Lemma] \label{lem-extension} %
  Let $k \in \N$ with $k \geq 1$. For any $f_{\rm in}\in \XSL{k+1/2}{\EEEE}$,
  $g_{\rm in} \in \XSL{k-1/2}{\EEEE}$ there exists $w \in \PXSL{k+1}{[0,\infty)
    \times E}$ such that $w_{|t=0} = f_{\rm in}$, $w_{t{|t=0}} = g_{\rm in}$ and
  \begin{align} \label{def-2-w} %
    \PXNL{w}{k+1}{[0,\infty) \times E} \ %
    \leq \ C \big(\XNL{f_{\rm in}}{k+1/2}{\EEEE} + \XNL{g_{\rm in}}{k-1/2}{\EEEE} \big).
  \end{align}
\end{lemma}
The extension can be constructed by gluing together solutions of the linear
equation given by Proposition \ref{prp-linear}. The methods used in
\cite{GiacomelliKnuepfer-2010} can also be used for our equation so that we will
not present the argument here. For $f_{\rm in} \in \XS{k}$, let $g_{\rm
  in}=-\AA_\eps f_{\rm in} $ and we choose $w \in \PXSL{k+1}{[0,\infty) \times
  E}$ as in Lemma \ref{lem-extension}. In particular,
\begin{equation} \label{app-sol} %
    \LL_\eps w |_{t=0} \ = \ D\partial_t w|_{t=0} + \AA_\eps f_{\rm in}  \ %
    \stackrel{D|_{t=0}=1}{=} \ g_{\rm in}-g_{\rm in} \ = \ 0
\end{equation} 
and $w$ may in this sense be called an approximate solution.  Let $\delta
\mathcal L_\eps^w$ be the linearization of $\mathcal L_\eps $ around $w$. We
have boundedness and differentiability of $\LL_\eps$ and boundedness of $ (
\delta \MM(w) )^{-1}$ for $\tau$ small enough:
\begin{proposition} \label{prp-local-linear} %
  Let $k \geq 1$, $k \in \N$ and suppose that $f_{\rm in} \in
  \XSL{k+1/2}{E}$ with $\int f_{\rm in} dx = 0$ and $g \in \PXSL{k}{Q_\tau}$
  satisfy \eqref{linear-compatibility}. Let $w \in \XSL{k+1}{[0,\infty) \times
    E}$ be defined as in Lemma \ref{lem-extension} with $g_{\rm in} :=
  g_{|t=0}$. Then for sufficiently small $\tau > 0$ and with the notation
  $Q_\tau = (0,\tau) \times E$, there exists a unique $f \in
  \PXSL{k+1}{Q_\tau}$, solution of $\delta \mathcal L_\eps(w) f \ = \ g$ in
  $Q_{\tau}$ with $f = 0$ on $(0,\tau) \times \partial \EEEE$ and with initial
  data $f = f_{\rm in}$. Furthermore,
  \begin{align} \label{est-linloc} %
    \PXNL{f}{k+1}{Q_\tau} \ \leq \ C \big(\PXNL{g}{k}{Q_\tau} + \XNL{f_{\rm
        in}}{k+1/2}{\EEEE} \big).
  \end{align}
\end{proposition}
Note that the condition $\int_0^1 f(x) dx = 0$ is preserved by the flow
generated by $\delta \mathcal L_\eps(w) f$ .  We also have differentiability of
$\LL_\eps$ in a neighborhood of $w$:
\begin{proposition} \label{prp-local-nonlinear} %
  Suppose that $f_{\rm in} \in \XSL{k+1/2}{E}$ satisfies \eqref{compatibility}
  and $\int_E f_{\rm in} dx = 0$ and let $w$ be defined as in Lemma
  \ref{lem-extension} with $g_{\rm in} = - \AA_\eps f_{\rm in}$. Then for
  sufficiently small $\tau > 0$ there is $\alp > 0$ such that $\MM_\eps:
  \PXSL{k+1}{Q_\tau} \to \XSL{k+1/2}{E} \times \PXSL{k}{Q_\tau}$ is bounded and
  continuously differentiable in the $\alp$-neighborhood of $w$ in
  $\PXSL{k+1}{Q_\tau}$.
\end{proposition}
The proof of the above two propositions is given in Section \ref{sec-local}.
Using the above two propositions, the proof of Theorem \ref{thm-local} follows
by application of the inverse function theorem: We claim that the operator
\begin{align*} %
  \MM_\eps : \PXSL{k+1}{Q_\tau} \to \PXSL{k+1/2}{\EEEE} \times \PXSL{k}{Q_\tau}
  && %
  \text{with} && %
  f \ \mapsto \ \MM_\eps(f) \ := \ (f_{|t=0}, \mathcal L_\eps f )
\end{align*}
is bounded, continuously differentiable near $w$ and $\delta \MM_\eps(w)$ is
invertible with bounded inverse for $\tau$ small enough.
  We define $v := \LL_\eps (w) $.  By the
inverse mapping theorem there is a neighborhood of $w$ and a neighborhood of
$(f_{\rm in}, v)$ where $\mathcal M_\eps$ is a diffeomorphism. By
\eqref{app-sol} we have $v_{|t=0} = 0$. Since $\cciL{\overline Q_\tau}$ is dense
in $\PXSL{k}{Q_\tau}$, it follows that $\PXNL{v}{k}{Q_\tau} \to 0$ for $\tau \to
0$. Hence, there is $\tilde \tau \in (0, \tau)$ and a function $\tilde v \in
Y_{\tau}$ with $\tilde v = 0$ for $t \in (0, \tilde \tau)$ and such that $(0,
\tilde v)$ is sufficiently near $(0,v)$.  Hence, there is $f \in X_\tau$ with
$\mathcal M_\eps(f) = (0,\tilde v)$.  The function $f$ is a solution of $
\LL_\eps f = 0 $ and hence $h(x) = x(1-x) + \int_0^x f(x') dx'$ is a solution of
\eqref{tfe-fixed} for $t \in (0,\tilde \tau)$, thus concluding the proof of
Theorem \ref{thm-local}.

\subsection{Proof of Theorem \ref{thm-tfe}} 

By Theorem \ref{thm-convergence}, we have existence and regularity of solutions
of \eqref{tfe-fixed} for initial data which are close to the infinite wedge. In
order to show Theorem \ref{thm-tfe}(1) it hence remains to prove uniqueness of
solutions. For this, it is enough to show that the corresponding results in
Propositions \ref{prp-linear} and \ref{prp-nonlinear} also hold in the case
$\eps=0$ and for the operators $A_0$ and $N_0$. The estimate for existence,
regularity and uniqueness in the case of half-space then follows by the same
fix-point argument as for Theorem \ref{thm-wedge}.

\medskip

The $\eps=0$ version of Proposition \ref{prp-linear} has been proved in
\cite{KnuepferMasmoudi-2012a}. The $\eps=0$-version of Proposition
\ref{prp-nonlinear} can be obtained by analogous estimates as the one's applied
in the proof of Proposition 8.1 in \cite{GiacomelliKnuepferOtto-2008}. In fact,
the estimate is easier in our situation since we only need for an estimate in
weighted Sobolev spaces, not the interpolation spaces used in
\cite{GiacomelliKnuepferOtto-2008}. Finally, we note that the local result in
Theorem \ref{thm-tfe} can be obtained by standard localization techniques. The
argument can be performed analogously as our localization argument in Section
\ref{sec-local}; the argument is easier since for $\eps = 0$, the norms
$\XN{\cdot}{k}$ are local. We also refer to a similar localization argument in
\cite{GiacomelliKnuepfer-2010}, performed for a thin-film equation in a H\"older
space setting.

\section{Estimates in weighted spaces} \label{sec-estimates}
 
 We recall some basic properties of the space $\XS{k}$ (see Proposition 2.3 of \cite{KnuepferMasmoudi-2012a}):
 \begin{proposition} \label{prp-X} %
   Let $k \in \R$ with $k \geq 0$ and let $f \in \XS{k}$.
  \begin{enumerate}
  \item For $0 \leq \eps \leq \eps'$, we have $\XS{k} \subseteq X_{\eps'}^{k}$
    and for all $f \in \XS{k}$
    \begin{align} \label{X-mon} %
      [f]_{X_{\eps'}^{k}} \ \leq \ \XM{f}{k}.
    \end{align}
  \item For any $\ell_1, \ell_1 \in \N_0$ with $0 \leq \ell_1 \leq 3k$ and $0
    \leq \ell_2 < 3k-\frac 12$, we have
    \begin{gather}   \label{X-sob} 
      \begin{aligned}
        \nlt{\partial_x^{\ell_1} f} + \nco{\partial_x^{\ell_2} f} \ \leq \
        C \XN{f}{k}.
      \end{aligned}
    \end{gather}
  \item If $f \in \XS{k}$, then $\partial_x f \in \XS{k-1/2}$,
    $\partial_x^2 f$, $\partial_x^3 f \in \XS{k-1}$ and
      \begin{align}
        \XN{f_x}{k-1/2} + \XN{f_{xx}}{k-1} + \XN{f_{xxx}}{k-1} \ \leq \ C \XN{f}{k}. \label{X-three}
      \end{align}
  \end{enumerate}
\end{proposition} 
The norm $\XN{\cdot}{k}$ controls a scale of weighted Sobolev spaces of
fractional order:
\begin{lemma} \label{lem-real} %
  Let $\eps \in [0, 1)$, $k \geq 0$ and $f_0 \in \XSoo{k}$. Then for every $0
  \leq \ell \leq k$, we have
  \begin{align} \label{est-real} %
    \nltL{\lambda^{3\ell + 1} \mu^{\ell} \widehat{f_0}}{\Re \lambda =
      3\ell-\frac 12} %
    + \nltL{\widehat{f_0}}{\Re \lambda = 3\ell-\frac 12} \ %
    \leq \ C \XN{f_0}{\ell}.
  \end{align}
\end{lemma} 
We have the following characterization of the homogeneous norms:
\begin{lemma} \label{lem-X-laplace} %
  Let $k \in \N_0$, $\eps \in [0, \frac{\pi}{3(2k+1)})$. Then for $f_0 \in
  \XSoo{k}$ we have
  \begin{align} \label{norm-eq2} %
    c \XM{f_0}{k} \leq \ \inf_{f_0 = f_+ + f_-}
    (\nltL{\lambda^{4k+1} \widehat {f_+}}{\Re \lambda = \bet_k} %
    + \nltL{  (\feps)^{k}  \lambda^{3k+1}
      \widehat {f_-}
    }{\Re \lambda = \bet_k} \big) \ %
    \leq C \XM{f_0}{k},
  \end{align}
  where $\beta_k = 3k - \frac 12$. Up to multiplication by a constant that
  only depends on $k$, for any even integer $M$ with $M \geq k$ the
  minimum in the right hand side of \eqref{norm-eq2} is achieved by
  \begin{align} \label{def-fpm} %
    \widehat{f_+}(\lambda) \ = \ (1 -(i\tan)^M(\eps
    \lambda)) \hat f_0(\lambda), && %
    \text{and} && %
    \widehat{f_-}(\lambda) \ = \ (i\tan)^M(\eps \lambda) \hat f_0(\lambda).
  \end{align}
  Furthermore,
    \begin{align} \label{est-fpm} %
      |\lambda|^k |\widehat{f_+}(\lambda)| \ \leq \
      C \mu^k |\hat f(\lambda)|, && ({\textstyle \frac 1{\eps}})^k
      |\widehat{f_-}(\lambda)| \ \leq \ C \mu^k |\hat f(\lambda)|.
    \end{align}
\end{lemma}
The proof of Proposition \ref{prp-X}, Lemma \ref{lem-X-laplace} and Lemma
\ref{lem-real} is given in \cite{KnuepferMasmoudi-2012a}.

\medskip

The space $\XS{k}$ is an algebra as is proved below. Note that there is a proof
for the algebra property in \cite{KnuepferMasmoudi-2012a}). The advantage of the
 result below  is that it also applies to the case $k = 1$, also the proof is
simpler than in \cite{KnuepferMasmoudi-2012a}:
\begin{lemma} \label{lem-algebra} %
  For $k \in \N$ with $k \geq 1$ and $f,g \in \XS{k}$, we have $fg \in \XS{k}$
  and
  \begin{align} \label{X-alg} %
    \XN{fg}{k} \ \leq \ C \XN{f}{k} \XN{g}{k}.
  \end{align}
\end{lemma}
\begin{proof}
  We decompose $f = f_1 + f_0$ and $g = g_1 + g_0$ where $f_1 = \PP_f \zeta$,
  $g_1 = \PP_g \zeta$ and where $\PP_f$, $\PP_g$ are the Taylor polynomials of
  $f$, $g$ of order $3k-1$. The cut-off function $\zeta$ is defined as for the
  definition of \eqref{norm}. In particular, $f_0, g_0 \in \XS{k}$. We need to
  estimate the products $f_1 g_1$, $f_0 g_1$, $f_1 g_0$ and $f_0 g_0$. Clearly
  $\XN{f_1 g_1}{k} \leq C \XN{f_1}{k} \XN{g_1}{k}$.  As in \eqref{def-fpm}, we
  decompose $f_0 = f_+ + f_-$ and $g_0 = g_+ + g_-$. The mixed term $f_1 g_0$
  is estimated as follows
   \begin{align*}
     \hspace{3ex} & \hspace{-3ex} %
     \XM{f_1 g_0}{k} \ %
     \leq \ C \nltL{x^{k+1}\partial^{4k} (f_1 g_+)}{(0,\infty)} +
     \nltL{{\textstyle (\frac 1\eps)}^k x \partial^{3k}(f_1
       g_-)}{(0,\infty)} \\
     &\leq \ C \XN{f}{k} \big( \nlt{x^{k+1}\partial^{4k} g_+} +
     \nlt{x \partial^{3k} g_-} \big) \ \leq \ C \XN{f}{k} \XN{g}{k},
   \end{align*}
   The second estimate follows since all derivatives of $f_1$ are supported in
   $(\frac 18,\frac 14)$ and since $x$ is of order $1$ in this interval. The
   estimate of $f_0 g_1$ proceeds analogously.

   \medskip

   It remains to show the estimate for the product $f_0 g_0 = f_+ g_+ + f_+ g_-
   + f_- g_+ + f_-g_-$. In particular, since $f_0, g_0 \in \XSoo{k}$, we have
   $|f_0|, |g_0| \leq Cx^{3k-1}$ for small $x$. It follows that $|f_0 g_0| \leq
   Cx^{6k-2} \leq Cx^{3k-2}$ (since $k \geq 1$), in particular $f_0g_0 \in
   \XSoo{k}$.  We show the estimate for the term high-high-frequency product
   $f_- g_-$.  Recall that the Mellin transform for the product of functions can
   be expressed as a convolution for the Mellin transformed functions. Hence,
   \begin{align} \label{ident} %
     \XM{f_- g_-}{k} \ &\lupref{norm-lambda}= \ \nltL{\lambda^{3k+1} \mu^k
       \int_{\Re \eta = \alpha} \hat f_-(\lambda-\eta) \hat g_-(\eta)
       \DI{\eta}}{\Re \lambda = 3k-\frac 12},
  \end{align} 
  where $\mu = \min \{ {\textstyle \frac 1\eps}, |\lambda| \}$ and where $\alpha
  \in \R$ is chosen such that the product $\hat f_-(\lam-\cdot) \hat g_-(\cdot)$
  is absolutely integrable on the line $\Re \eta = \alpha$.  Since $\mu \leq
  \feps$, the right-hand side of \eqref{ident} can be estimated  by
  replacing $\mu^k$ by $(\feps)^k$. Using the binomial formula $\lambda^{3k+1} =
  \sum_i c_{ki} \eta^i (\lambda-\eta)^{3k+1-i}$, \eqref{ident} is bounded by
  above by the sum of the terms
  \begin{gather*} %
    I_{k,i}(f_-,g_-) = \int_{\Re \lambda = 3k-\frac 12} \left| \int_{\Re \eta =
        \alpha} (\feps)^k (\lambda-\eta)^{3k+1-i} \widehat{f_-}(\lambda-\eta)
      \eta^{i} \widehat{g_-}(\eta) \DI{\eta} \right|^2
    \hspace{-1ex}\DI{\lambda},
   \end{gather*}
   where $0 \leq i \leq 3k+1$.  By symmetry, it is enough to estimate the terms
   with $i \leq \frac {3}2 k+\frac 12$. Note that the integrand of the inner
   integral above is analytic as a function of $\eta$. In particular, the value
   of the integral does not depend on $\alpha$.  This argument works since we
   have avoided to replace $\lam$ by $|\lam|$ in our proof.  By Young's
   inequality for convolutions and by the Cauchy-Schwarz inequality, we have
   \begin{align} \label{often-1} %
     \nltL{\hat F * \hat G}{\Re \lam = \bet} %
     \ &\leq \ C_\delta \nltL{(1+|\lam|^{\delta})\hat F}{\Re \lam = \bet_1}
     \nltL{\hat G}{\Re \lam = \bet_2},
     \end{align}
     which holds for all $\delta > \frac 12$.  With the choice $\alpha = i$ and
     $\delta = 1$, we get
   \begin{align}
     I_{k,i}(f_-,g_-) \ %
     &\upref{often-1}\leq \ C \nltL{(\feps)^{k-\frac i3} |\lambda|^{3k+1-i}
       \widehat{f_-}}{\Re \lambda = 3k-i-\frac 12} \nltL{(\feps)^{\frac i3}
       (1+|\lambda|^{i+1})
       \widehat{g_-}}{\Re \lambda = i} \notag \\
     &\upref{est-fpm}\leq \ C \nltL{\mu^{k-\frac i3} |\lambda|^{3k-i+1}
       \widehat{f_0}}{\Re \lambda = 3k-i-\frac 12} \nltL{\mu^{\frac i3} (1 +
       |\lam|^{i+1}) \widehat{g_0}}{\Re \lambda = i} \notag \\ %
     &\upref{est-real}\leq \ C \XN{f}{k} \XN{g}{k}. \notag
   \end{align}
   The last estimate follows from \eqref{est-real} using $0 \leq i \leq \frac 32
   k+\frac 12$ (and since $k \geq 1$). The estimate of the terms $f_+ g_-$, $f_-
   g_+$ and $f_+ g_+$ proceeds analogously (see also the related proof in
   \cite{KnuepferMasmoudi-2012a}). This concludes the proof of the Lemma.
 \end{proof}
 Also the space $\YS{k}$ is embedded into classical Sobolev spaces: For any
 multi-index $\alp \in \R_{\geq 0} \times \N_0$, we set $|\alp| = \alp_1 +
 \alp_2$. For $\alpha \in \N^2$ we also use the notation $D_\eps^\alp
 = \partial_x^{\alp_1} (\frac 1\eps \partial_y)^{\alp_2}$.
 \begin{lemma} \label{lem-Y} %
   Let $k \in \R$ with $k \geq 1$, $\eps \in (0, 1)$ and let $q \in \YS{k}$. Then
   for all $\ell_1$, $\ell_2 \in \N_0$ with $\ell_1 < 3k-1$ 
  and $\ell_2 \leq
   3k-\frac 32$ (or equivalently $\ell_2 \leq n_k - 1$), we have
   \begin{gather} \label{Y-sob} %
     \sum_{0 \leq |\alp| \leq \ell_1} \NNN{D_\eps^\alp q}{L_y^\infty L_x^2(K)} %
     + \sum_{0 \leq |\alp| \leq \ell_2} \sup_{y \in (0,x)} \NNN{D_\eps^\alp
       q}{L^\infty(K)} \ %
     \leq \ C \YN{q}{k},
  \end{gather}
  where the sums are taken over multiindices $\alp \in \Z^2$.
\end{lemma}
\begin{proof}[Proof of Lemma \ref{lem-Y}]
  Let $q=q_0 + q_1$ with $q_0 \in \YSoo{k}$ be the decomposition as in
  \eqref{Y-hom}. By definition, $q_1$ solves \eqref{Y-sob}. It   suffices to
  show the estimate \eqref{Y-sob} for $q_0 \in \YSoo{k}$. By application of
  Plancherel's identity, we have
  \begin{align*}
    \sum_{|\alp| = 3\ell-1} \sup_{y \in(0,x)} \nltL{D_\eps^\alp q_0}{(0,\infty)}
    \ %
    &\leq \ C \sum_{|\alp| = 3\ell-1} \sup_{v \in (0,1)}
    \nltL{\partial_u^{\alp_1} {(\textstyle \frac 1\eps \partial_v)}^{\alp_2}
      q_0}{(0,\infty)} \\ %
    &\upref{Y-hom}\leq \ C \sum_{|\alp| = 3\ell-1} \sup_{v \in (0,1)}
    \nltL{\Lam^\alp \widehat{q_0}}{\Re \lambda = 3\ell-\frac 32} \ %
    \leq \ C \YM{q_0}{k}. %
  \end{align*}
  This yields the $L^2$ estimate, the supremum estimate follows by standard
  interpolation.
\end{proof}

For any $f \in \PXS{k+1}$, its trace at $t=0$ is well-defined in $\XS{k+1/2}$:
 \begin{lemma} \label{lem-time-trace} %
   Let $k,\gamma \in \N_0$. For $f \in \PXS{k+1}$ we have
   \begin{gather} \label{est-time-trace} %
     \|\partial_t^i f\|_{C^0(\XS{k+1/2})} \ %
     \leq \ C \|\partial_t^{i+1} f\|_{L^2(\XS{k})} + C \|\partial_t^i
     f\|_{L^2(\XS{k+1})}.
   \end{gather}
\end{lemma}
\begin{proof}
  It suffices to give the proof for $i=0$ (for $i > 0$ consider $F
  = \partial_t^i f$ instead). By an approximation argument it is enough to
  consider $f \in \cciL{[0,\infty)^2}$ .  We decompose $f = f_0 + f_1$ where
  $f_1 = \PP_f \zeta$ and where $\PP_f$ is the Taylor polynomial  of order
  $n_{k+1/2} = 3k$ of $f$ at $x=0$; in particular, $f_0 \in \XSoo{k+1/2}$. In
  order to avoid fractional derivatives which appear in the definition of the
  norm for $\XS{k+1/2}$, we use the equivalence $c\XM{f_0}{k+1/2} \leq
  \XP{Af_0}{f_0}{k} \leq C\XM{f_0}{k+1/2}$ which holds for all $f_0 \in
  \XSoo{k+1/2}$. A proof of this equivalence is given for $k=0$ in \cite[Lemma
  4.6]{KnuepferMasmoudi-2012a}, the argument used there also applies for general
  $k \in \N$. The estimate of the homogeneous part is then easy: For $f, g \in
  \XSoo{k+1/2} \cap \XS{k+1}$, we have $\XP{A_\eps f_0}{g_0}{k} =
  \XP{f_0}{A_\eps g_0}{k}$. Hence, integrating in time from infinity (where $f =
  0$), we obtain
  \begin{align*}
    \sup_t \XM{f_0}{k+1/2}^2 \ %
    &\leq \ C \sup_t \XP{A_\eps f_0}{f_0}{k} \ %
    \leq \ C \int_0^\infty \Big| \XP{A_\eps f_0}{f_{0t}}{k} \Big| \ dt \\ %
    &\leq \ C \TXM{A_\eps f_0}{k} \TXM{f_{0t}}{k} \
    \leq \ C \TXM{f_0}{k+1} \TXM{f_{0t}}{k}. %
  \end{align*}
  It remains to give the corresponding estimate for $f_1$: That is, we
  need to estimate the coefficients of the Taylor polynomial
  $\PP_f$. We show the estimate for the highest order coefficient of
  $\PP_f$: Up to a constant it is given by $F(0)$ where $F
  := \partial^{3k} f$. We extend $F$ symmetrically as an even  function
  defined for all $t,x \in \R$ by setting $F(t,-x) := F(t,x)$.  We
  claim that
    \begin{align} \label{tc} %
      \sup_{t \in \R} |F(0)| \ \leq \ C \big( \nltL{F_t}{\R^2} + \nltL{F_{xxx}}{\R^2} + \nltL{F}{\R^2} \big).
    \end{align}
    Indeed, \eqref{tc} can be derived by taking the Fourier transform $\hat
    F(\eta,\xi)$ of $F(t,x)$
    \begin{align*}
      \sup_{t \in \R} |F(0)| \ \leq \ \iint |\hat F| \ %
      \leq \ \left(\iint \frac 1{(1 + |\eta|^2 + |\xi|^6)} \right)^{1/2}
      \left(\iint (1 + |\eta|^2 + |\xi|^6) |\hat F|^2 \right)^{1/2}.
    \end{align*}
    The first integral on the right hand side is bounded: One the one hand we
    have
    \begin{align}
      \iint_{|\xi| \leq 1 \ \text{or} \ |\eta| \leq 1} \frac 1{1 + |\eta|^2 +
        |\xi|^6} \ d\xi d\eta \ %
      \leq \ C \ < \ \infty.
    \end{align}
    Also, with the coordinate transform $\xi^6 = \eta^2 \lam^6$ and $d\xi =
    \eta^{1/3} d\lam$
    \begin{align}
      \int_1^\infty \int_1^\infty \frac 1{|\eta|^2 + |\xi|^6} \ d\xi d\eta \ %
      \leq \ C \int_1^\infty \eta^{-5/3} \int_0^\infty \frac 1{1 + |\lam|^6} \
      d\lam d\eta \ \leq C \ < \ \infty.
    \end{align}
    This concludes the proof of \eqref{tc} and thus of the lemma.
\end{proof}

\section{Uniform estimates for the operator in the half-space} \label{sec-wedge} %

\subsection{Linear pressure estimates}

We derive estimates for the pressure $p \in \YS{k}$
\begin{proposition} \label{prp-p-linear} %
  Let $k \in \R$ with $k \geq 0$, $\eps \in (0, \frac \pi{3(2k+1)})$. Then for any
  $f \in \XS{k}$ and any $g \in \ZS{k}$ there is a unique solution $p \in
  \YS{k}$ of
  \begin{gather} \label{p-wedge}
     \left \{
    \begin{array}{ll}
      \Delta_\eps p \ = \ g  \qquad&\text{ in } K, \\
      p \ = \ f_x  \qquad&\text{ on } \partial_1 K, \\
      p_y \ = \ 0  \qquad&\text{ on } \partial_0 K.
    \end{array}
    \right.
  \end{gather}
  Furthermore, we have
  \begin{align} \label{est-p} %
    \YN{p}{k} \ \leq \ C \big(\ZN{g}{k} + \XN{f}{k} \big).
  \end{align}
\end{proposition}
We first address the situation of homogeneous data $f \in \XSoo{k}$ and $g \in
\ZSoo{k}$: The coordinate transform \eqref{def-ts} leads to the following model
for $p$ understood as a function of $(u,v)$:
\begin{gather} \label{p-strip} \left \{
    \begin{array}{ll}
      \partial_u^2 p + (\feps \partial_v)^2 p \ = \ e^{2u} g
      &\text{ in } \R \times (0,1), \\
      p(u, \cdot) \ = \ e^{-u} f_u(u)  \qquad&\text{ for } v=1, \\
      p_v(u, \cdot)  \ = \ 0  \qquad&\text{ for } v=0.
    \end{array}
  \right.
\end{gather}
Application of the Laplace transform \eqref{mellin} in terms of $u$ yields
\begin{gather} \label{p-mellin} %
  \left \{
    \begin{array}{ll}
      \lam^2 \hat p(\lam,v) + (\feps \partial_v)^2 \hat p(\lam,v) \ = \ \hat g(\lam-2,v)
      &\text{ in } \R \times (0,1), \\
      \hat p(\lam, \cdot) \ = \ (\lam+1) \hat f(\lam+1)  \qquad&\text{ for } v=1, \\
      \hat p_v(\lam, \cdot)  \ = \ 0  \qquad&\text{ for } v=0.
    \end{array}
  \right.
\end{gather}
Explicit solution of \eqref{p-mellin} for $g = 0$ yields:
\begin{lemma}
  Suppose that the assumptions of Proposition \ref{prp-p-linear} hold. Suppose
  that $f \in \XSoo{k}$ and $g=0$. Then there is a unique solution $q:=p
  \in\YSoo{k}$ of \eqref{p-wedge}. It is given by
  \begin{align} \label{q-formula} %
    \hat q(\lambda,v) \ = \ \frac{\cos(\eps \lambda v)}{\cos(\eps \lambda)}
    (\lambda+1) \hat f(\lambda+1).
  \end{align}
  Moreover, $\YN{q}{k} \ \leq \ C \XN{f}{k}$.
\end{lemma}
\begin{proof}
  Clearly \eqref{q-formula} solves \eqref{p-mellin}.  The estimate $\YN{p}{k}
  \leq C\XNoo{f}{k}$ follows from
  \begin{align} \label{using} %
    \left|\frac{\sin(\eps \lambda v)}{\cos(\eps \lambda)} \right| \ %
    \leq \ \left|\frac{\cos(\eps \lambda v)}{\cos(\eps \lambda)}
    \right| \ %
    \leq \ C e^{\eps |\lambda| (v-1)} \ \leq \ C e^{\frac 12 \eps
      |\lambda| (v-1)}
  \end{align}
  for $v \in (0,1)$ which also implies uniqueness. Indeed, for $|\eps \lam| \leq
  1$, we have $\mu = |\lam|$, $|\frac {\cos (\eps \lam \nu)}{\cos(\eps \lam)}|
  \leq C$ and $|e^{|\eps \lam| (v-1)}| \geq e^{-1}$. For $|\eps \lam| \geq 1$,
  we have $|\cos(\eps \lam v)| \leq 4 e^{|\eps\lam|v}$ and $|\sin(\eps \lam)|
  \geq \frac 14 e^{|\eps\lam|}$. This proves \eqref{using}. See also
  \cite{KnuepferMasmoudi-2012a}.
\end{proof}
Explicitly solving \eqref{p-strip} for $f = 0$ yields:
\begin{lemma} \label{lem-w-hom} Suppose that the assumptions of Proposition
  \ref{prp-p-linear} hold. Suppose that $f = 0$ and $g \in \ZSoo{k}$. Then there
  is a unique solution $w:=p \in\YSoo{k}$ of \eqref{p-wedge}. It is given by
  \begin{align} \label{w-formula}
    \begin{aligned}
      \hat w(\l,v) \ &= \ \frac\eps{ \lambda  \cos(\e \l) } \cos(\eps \lambda v ) \int_v^1 \sin(\eps \lambda (z-1)  )   \hat g(\lambda-2,z) dz \\
      &\quad + \frac\eps{ \lambda \cos(\e \l) } \sin(\eps \lambda (v-1) ) \int_0^v \cos(\eps
      \lambda z ) \hat g(\lambda-2,z) dz.
    \end{aligned}
  \end{align}
  Moreover, $\YN{w}{k} \ \leq \ C \ZN{g}{k}$.
\end{lemma}

\begin{proof}
  The solution $w$ can be expressed in terms of the Green function $G(v,z)$ by
  \begin{align*} %\label{w3} %
    \hat w (v) \ = \ \int G(v,z) \eps^2 \hat g(\lambda-2, z) dz, \ \text{where} \ 
    \left \{
      \begin{array}{ll}
        \eps^2 \lambda^2 G(v,z)  +  \partial_v^2 G(v,z)  = \delta_{v=z} 
        &\text{in } \Omega, \\
        G = 0&\text{for } v=1, \\ 
        G_v = 0&\text{for } v=0.
      \end{array}
    \right.
  \end{align*}
  Since $G$ is harmonic away from $v=z$ and continuous at $v=z$, it must be of the  form
  % Hence,
  % \begin{gather*}
  %   G(s,z) \ = \ 
  %   \left \{
  %     \begin{array}{ll}
  %       \alpha \cos(\eps \lambda v )
  %       \quad&\text{if } \  v < z , \\
  %       \beta  \sin(\eps \lambda (v-1)  )
  %       \quad&\text{if } \  z < v ,
  %     \end{array}
  %   \right.
  % \end{gather*}
  % $G$ has to be continuous and hence, $\alpha = C \sin(\eps \lambda (z-1) )
  % $ and $\beta = C \cos(\eps \lambda z ) $
  \begin{gather*}
    G(v,z) \ = \  \left \{
      \begin{array}{ll}
        C \sin(\eps \lambda (z-1)  )    \cos(\eps \lambda v )
        \quad&\text{if } \  v < z , \\
        C  \cos(\eps \lambda z )     \sin(\eps \lambda (v-1)  )
        \quad&\text{if } \  z < v,
      \end{array}
    \right.
  \end{gather*}
  for some constant $C$ to be determined. Taking the derivative of the above
  equation in $v$, we get
  \begin{gather*}
    \partial_v G(v,z) \ = \ \left \{
      \begin{array}{ll}
        - C  \eps \lambda   \sin(\eps \lambda (z-1)  )        \sin(\eps \lambda s )
        \quad&\text{if } \  v < z , \\
        C\eps \lambda  \cos(\eps \lambda (v-1)  )\cos(\eps \lambda z ) 
        \quad&\text{if } \  z < v.
      \end{array}
    \right.
  \end{gather*}
  Since the jump of $\partial_v G(v,z) $ at $v=z$ is 1, namely $[\partial_v
  G(v,z) ] = 1 $, we deduce that
  \begin{align*}
    \frac1C \ = \ \cos(\eps \lambda (z-1) )\cos(\eps \lambda z ) +
    \sin(\eps \lambda (z-1) ) \sin(\eps \lambda z ) \ = \ \eps \lambda
    \cos(\e \l)
  \end{align*}
  which implies \eqref{w-formula}.  We next give the proof of the estimate: We
  will use that
  \begin{align*}
    \left| \sin(\eps \lambda (v-1) ) \right| \leq C e^{|\e \l| (1-v)}, \ \  %
    \left| \cos(\eps \lambda v ) \right| \leq C e^{|\e \l| v} \ \ \text{and}
    \ \ %
    \left| \frac{ \cos(\eps \lambda v ) }{ \cos(\e \l) } \right| \leq C e^{|\e
      \l| (v-1)}.
  \end{align*}
  Hence, using the notation $\hat g(\sigma) = \hat g(\lambda-2,\sigma)$, we
  infer that
  \begin{align} \label{sodala-1} %
    \Big| \int_v^1 \sin(\eps \lambda (\sigma-1) )
    \hat g(\sigma) d\sigma \Big| \ & \leq \ C \int_v^1 e^{|\e \l| (1-\sigma)}
    e^{-\frac{|\e \l|}2 (1-\sigma)} \sup_{\sigma \in (0,1)} \big| e^{\frac{|\e \l|}2
      (1-\sigma)} \hat g(\sigma)  \big|  d\sigma   \notag  \\
    &\leq \ \frac{C}{\e \l} e^{ \frac{|\e \l|}2 (1-v)}
    \sup_{\sigma \in (0,1)} \big|e^{\frac{|\e \l|}2 (1-\sigma)} \hat g(\sigma)
    \big|
  \end{align}
  and
  \begin{align} \label{sodala-2} %
    \Big| \int_0^v \cos(\eps \lambda \sigma ) \hat g(\sigma) d\sigma \Big| \
    &\leq \ C \int_0^v e^{|\e \l| \sigma} e^{-\frac{|\e \l|}2 (1-\sigma)}
    \sup_{\sigma \in (0,1)} \big| e^{\frac{|\e \l|}2
      (1-\sigma)} \hat g(\sigma) \big|  d\sigma   \notag  \\
    &\leq \ \frac{C}{\e \l} e^{ \frac{|\e \l|}2 ( 3 v - 1)} \sup_{\sigma \in
      (0,1)} \big| e^{\frac{|\e \l|}2 (1-\sigma)} \hat g(\sigma) \big|.
  \end{align}
  By \eqref{sodala-1}-\eqref{sodala-2} and in view of \eqref{w-formula}, we
  deduce that
  \begin{equation} \label{side-1} %
    |\hat w(\l,v) | \ \leq \ \frac{C}{|\lam|^2} e^{ \frac{|\e \l|}2 (v - 1 ) }
    \sup_{v \in (0,1)} \big| e^{\frac{|\e \l|}2 (1-v)} \hat g(\l-2, v) \big|.
  \end{equation} 
  We calculate the first derivative,
  \begin{align} \label{first-der} %
    \partial_v \hat w(\l,v) \ & = \ - \frac{\eps^2}{ \cos(\e \l) } \sin(\eps
    \lambda v )
    \int_v^1 \sin(\eps \lambda (\tilde v-1)  )   \hat g(\lambda-2,\tilde v) d\tilde v \notag \\
    &\quad \ - \frac\eps{ \lambda  \cos(\e \l) } \cos(\eps \lambda v )  \sin(\eps \lambda (v-1)  )   \hat g(\lambda-2,v) \notag  \\
    &\quad \ + \frac{\eps^2}{ \cos(\e \l) } \cos(\eps \lambda (v-1) )
    \int_0^v
    \cos(\eps \lambda \tilde v )   \hat  g(\lambda-2,\tilde v) d\tilde v \notag \\
    &\quad \ + \frac\eps{ \lambda \cos(\e \l) } \sin(\eps \lambda (v-1) )
    \cos(\eps \lambda v ) \hat g(\lambda-2,v).
  \end{align}
  A similar calculation as before shows that
  \begin{equation} \label{side-2} %
    |\feps \hat w_v(\l,v)| \ \leq \ \frac C{|\lam|} e^{ \frac{|\e \l|}2 (v - 1 )  } 
    \sup_{v \in (0,1)} \big| e^{\frac{|\e \l|}2  
      (1-v)} \hat g(\l-2, v) \big|. 
  \end{equation}
  Multiplication of both sides of \eqref{side-1}-\eqref{side-2} by $|\lam|^k$
  yields higher regularity in the radial variables. Higher regularity in $v$
  follows by \eqref{side-1}, \eqref{side-2} and repeated application of
  \begin{align} \label{nochwas} %
    |\lam^{k} {\textstyle (\frac 1\eps)^2} \hat w_{vv}(\lam,v)| \ %
    \lupref{p-mellin}\leq \ |\lam^{k+2} \hat w(\lam,v)| + |\lam^k \hat
    g(\lam-2,v)|.
  \end{align}
 Estimates \eqref{side-1}, \eqref{side-2} and \eqref{nochwas} imply
  \begin{equation*}
    \sum_{|\alp| = 3k+2}  \sup_{v \in (0,1)} | e^{ \frac{|\e \l|}2 (1-v)} \Lambda^{3k+2} \mu^{k} \hat w(\l,v)| \ % 
    \leq \ C \sum_{|\alp| = 3k}  \sup_{v \in (0,1)}  \big|e^{\frac{|\e \l|}2(1-v)} \Lam^{3k} \mu^{k}\hat g(\l-2, v) \big|.
  \end{equation*}
  Estimate \eqref{est-p} follows by taking the $L^2$-norm on the line $\Re
  \lambda = 3k-\frac 32$ on both sides.
\end{proof}

  \begin{proof}[Proof of Proposition \ref{prp-p-linear}]
    Let $\PP_f$ be the Taylor polynomial of $f$ at $x= 0$ of order
    $n_k=3k-1$; let $\PP_g$ be the Taylor polynomial of $g$ at
    $(x,y) = 0$ of order $n_k - 3$. Let $\PP_p$ be the polynomial
    solving \eqref{p-wedge} with $g$ and $f$ replaced by $\PP_g$ and
    $\PP_f$.  Existence and uniqueness of this polynomial solution
    follows from a straightforward calculation and furthermore
    \begin{align}
      \NNN{\PP_p}{\PP} \ \leq \ C \left(\NNN{\PP_f}{\PP} + \NNN{\PP_g}{\PP} \right),
    \end{align}
    see also the proof of Lemma 4.2 in \cite{KnuepferMasmoudi-2012a} as well  
 as the proof of Lemma \ref{lem-pressure-trace} below.  Let $\zeta :
    K \to \R$ be a cut-off such that $\zeta = \zeta(r)$ with $\zeta = 1$ in
    $[0,\frac 18]$, $\zeta = 0$ outside $[0,\frac 14]$ and such that $0 \leq
    \zeta \leq 1$ and let $\tilde \zeta ; (0,\infty) \to \R$ be given by $\tilde
    \zeta(x) = \zeta(|(x,x)|)$. We define
    \begin{align}
      p_1 \ := \ \PP_p \zeta, && %
      g_1 \ := \ \Delta_\eps p_1, && %
      f_1 \ := \ \tilde \zeta \int_0^x \PP_p \ d\tilde x,
    \end{align}
    and $g_0 := g - g_1 \in \ZSoo{k}$ and $f_0 := f - f_1 \in \XSoo{k}$.
    Furthermore, let $p_0$ be the solution of $\Delta p_0 = g_0$ in $K$ with
    $p_0 = \eta_0$ on $\partial_1 K$ and $\partial_y p_{0} = 0$ on $\partial_0
    K$.  By the previous two lemmas, there exists such a solution satisfying
    \begin{align}
      \YN{p_0}{k} \ \leq \ C \left(\ZNoo{g_0}{k} + \XNoo{f_0}{k} \right).   
    \end{align}
    Hence, $p := p_0 + p_1$ is a solution of \eqref{p-wedge} and satisfies the
    desired estimate.
\end{proof}
Since $\eps$ has the scaling of vertical length, one could expect that there is
only uniform control on the norm $\feps (
p_y)_{|\Gamma}$. But it turns out that even $({\textstyle \frac
  1\eps})^2(p_y)_{|\Gamma}$ is bounded uniformly for $\eps > 0$. The proof of
this statement is given in the next lemma:
\begin{lemma} \label{lem-pressure-trace} %
  Let $k \geq 1$, $\eps \in (0, \frac {\pi}{3(2k+1)})$. Then the solution $p \in
  \YS{k}$ of \eqref{p-wedge} satisfies
  \begin{align} \label{est-p-trace} %
    \XN{\partial_x(p_x)_{|\Gamma}}{k} + (\feps)^2
    \XN{\partial_x(p_y)_{|\Gamma}}{k} \ \leq \ C \big(\XN{f}{k+1} + \ZN{g}{k+1} \big).
  \end{align}
\end{lemma}
\begin{proof}
  Analogously as in the proof of Proposition \ref{prp-p-linear}, we decompose $p
  = p_1 + p_0$ where $p_1$ encodes the expansion at the boundary and where $p_0
  \in \YSoo{k+1}$ is the solution of \eqref{p-wedge} with corresponding
  homogeneous data $f_0 \in \XSoo{k+1}$ and $g_0 \in \XS{k+1}$. Furthermore, let
  $p_0 = q_0 + w_0$ where $q_0$ is the solution of \eqref{p-wedge} with boundary
  data $f_0$ and with right hand side $g=0$. Correspondingly, $w_0$ is the
  solution with right hand side $g_0$ and with boundary data $f=0$. In the
  sequel, we present give the corresponding estimates to \eqref{est-p-trace} for
  $q_0$, $w_0$ and $p_1$; together these estimates imply \eqref{est-p-trace}.

  \medskip

  {\it Estimate for $q_0$: } With the transformation \eqref{eq-ts}, we need to
  show
\begin{align} \label{q0-1} %
  \XN{e^{-2u} q_{0,uu|\Gamma}}{k} + (\feps)^2 \XN{e^{-2u} q_{0,vu|\Gamma}}{k} \
  \leq \ C \XN{f_0}{k+1}.
  \end{align}
  Here, and in the following, by a slight abuse of notation we understand $q_0$
  as a function of $(u,v)$.  By \eqref{q-formula}, since $|\sin(\eps
  \lam)/\cos(\eps \lam)| \leq \ C \mu$ and for $|\Re \lam| \in (\frac 12,
  4k-\frac 12)$ we obtain
  \begin{align*}
    |\widehat{e^{-2u} q_{0,vu}}(\lambda, 1)| \ %
    &= \ \big|\eps \lambda^2 \frac{\sin(\eps (\lambda+2))}{\cos (\eps
      (\lambda+2))} (\lambda+3) \hat f_0(\lambda + 3)\big| \ %
    \leq \ C |\eps^2 \lambda^3 \mu \hat f_0(\lambda + 3)|.
  \end{align*}
  Multiplying   this identity by  $|\lam|^{3k}\mu^{k+1}$ and  taking
  the $L^2$-norm on the line $\Re \lam = 3k - \frac 12$, we obtain the estimate
  for the second term on the left hand side of \eqref{q0-1}. The estimate for
  the first term proceeds analogously.

  \medskip

  {\it Estimate for $w_0$:} With the transform we need to show
  \begin{align} \label{fsfs} %
    \XN{e^{-2u} w_{0,uu|\Gamma}}{k} + (\feps)^2 \XN{e^{-2u} w_{0,vu|\Gamma}}{k}
    \ \leq \ C \ZN{g_0}{k+1}.
  \end{align}
  Evaluating \eqref{first-der} at $\nu=1$ we
  note that only the third term does not vanish, i.e.
  \begin{align} \label{2-der} %
    \partial_v \hat w(\l,1) \ \lupref{first-der}= \ \frac{\eps^2}{ \cos(\e \l) }
    \int_0^1 \cos(\eps \lambda \tilde v ) \hat g(\lambda-2,\tilde v) d\tilde v.
  \end{align}
  With the notation $\phi = (\feps)^2 e^{-2u} w_{0,vu|\Gamma}$ and since $\hat
  \phi(\lam) = \feps (\lam+2) \widehat{w_{0,v}}(\lam+2, 1)$, we get
  \begin{align} \label{todotodo} %
    \hat \phi(\lam) \ %
    \lupref{2-der}= \ \frac{\lambda+2}{\cos(\eps (\lambda+2))} \int_0^v
    \cos(\eps (\lambda+2) \tilde v ) \hat g_0(\lambda,\tilde v) d\tilde v.
  \end{align}
%  Using $|\cos(\e (\l+2))|^{-1} \leq C \eps \mu$ and 
  For $|\eps \lam| \leq 1$, we have $|e^{\frac{|\e \l|}2 (1-v)}| \leq 1$ and
  $|\cos (\eps \l)| \geq \frac 12$, $|\l| = \mu$ and therefore
  \begin{align}
    \big| \hat \phi(\lam) \big| \ %
    \leq \ C |\lambda| \sup_{v \in (0,1)} |\hat g_0(\lambda,v)| \ %
    \leq \ C \mu \sup_{v \in (0,1)} |e^{\frac{|\e \l|}2 (1-v)} \hat
    g_0(\lambda,v)| \ %
  \end{align}
  For $|\eps \lam| \geq 1$, application of \eqref{sodala-2} yields
  \begin{align} \label{sodala-3} %
    \big| \hat \phi(\lam) \big| \ %
    \leq \ \tfrac{C}{\e} \sup_{v \in (0,1)} \big| e^{\frac{|\e \l|}2
      (1-v)} \hat g(\lam, v) \big| %
    \leq \ C \mu \sup_{v \in (0,1)} \big| e^{\frac{|\e \l|}2 (1-v)}
    \hat g(\lam, v) \big|, %
  \end{align}
  where we also used that $\mu = \feps$ in this case. The above two inequalities
  together imply
  \begin{align*}
    |\lambda^{3\ell+1} \mu^\ell \hat \phi(\lam)| \ %
    \leq \ C \sup_{v \in (0,1)} \big|e^{\frac{|\e \l|}2 (1-v)} \lambda^{3\ell+1}
    \mu^{\ell+1} \hat g_0(\lambda,v) \big|
  \end{align*}
  for all $\ell \geq 0$.  Integrating the square of the above estimate on the
  line $\Re \lam = 3k - \frac 12$ yields
  \begin{align*}
    \XM{\phi}{\ell} \ %
    &= \ \nltL{\lambda^{3\ell+1} \mu^\ell \hat \phi}{\Re \lambda = 3\ell-\frac
      12} % \\%
    \leq \ C \nltL{\sup_{v \in (0,1)} \big|e^{\frac{|\e \l|}2 (1-v)}
      \lambda^{3\ell+1} \mu^{\ell+1} \hat g_0 \big|}{\Re \lambda = 3\ell-\frac
      12} \\ %
    &\leq \ C \ZN{g_0}{k+1},
  \end{align*}
  which concludes the estimate for the second term on the left hand side of
  \eqref{fsfs}. The estimate of the first term on the left hand side of
  \eqref{fsfs} proceeds similarly.

  \medskip

  {\it Estimate for $p_1$:} Let $\PP_p = \sum_{i+j \leq \ell} a_{ij}
  x^i y^j$ be the polynomial which solves \eqref{p-wedge} with
  boundary data $\PP_f$ and right hand side $\PP_g$, where $\PP_f =
  \sum_i b_i x^i$ is the Taylor polynomial of $f$ of order $3k-1$ and
  where $\PP_g = \sum_{ij} g_{ij} x^i y^j$ is the Taylor polynomial of
  $g$ of order $3k-4$. Analogously as in the proof of Lemma
  \ref{lem-w-hom}, we need to show that the coefficients of the
  polynomial $(\feps)^2 \PP_{p,y|\Gamma}$ are bounded by the
  coefficients of $f$: Indeed, since by the condition we have
  $\partial_x^i \PP_{p,y|y=0} = 0$, it follows that $a_{i,1}=0$ for
  all $i \geq 0$. With the equation, i.e.
  \begin{align} \label{theq} %
    \partial_y^2 \partial_x^i \PP_{p,y} \ = \ \eps^2  \left(\partial_x^i \PP_{g,y} -
      \partial_x^2 \PP_{p,y} \right),
  \end{align}
  we first get $|a_{i,3}| \leq \eps^2 \NNN{\PP_g}{}$ and then
  iteratively $|a_{i,2j+1}| \leq \eps^2 \NNN{\PP_g}{}$ for all $i,j
  \geq 0$. Furthermore since $\PP_{p|\Gamma} = f_x$ and again using
  \eqref{theq}, one can easily deduce that $|a_{i,2j}| \leq C \eps^2
  (|b_{i+j}| + \NNN{\PP_g}{})$ for all $i,j \geq 0$ (we e.g. have
  $a_{02} = \eps^2 (a_{20} + g_{00})$ and $a_{02} + a_{02} = b_2$). In
  particular, $|a_{ij}| \leq \eps^2 (\NNN{\PP_f}{\PP} +
  \NNN{\PP_g}{\PP})$ for all $i \geq 0, j \geq 1$ which yields the
  desired estimate.
\end{proof}

\subsection{Pull-back onto wedge} 

We need to measure the difference $p_1 - p_2$, where $p_1$ is the solution for
the pressure on the domain $\tilde K^{\phi_1}$ and $p_2$ is the corresponding
solution on $\tilde K^{\phi_2}$, see Proposition \ref{prp-p-nonlinear}.  For
this, for given profile function $\phi$, we introduce a pull-back from the
perturbed wedge $K^\phi$ to the unperturbed wedge $K$, see Fig.
\ref{fig-psi}a). The estimates are nonlinear due to the geometry of the domain.

\begin{figure}
  \centering
  \begin{minipage}{0.08\linewidth}
    \hfill
  \end{minipage}
  \begin{minipage}{0.22\linewidth}
    \centering
    \includegraphics[width=1.7cm]{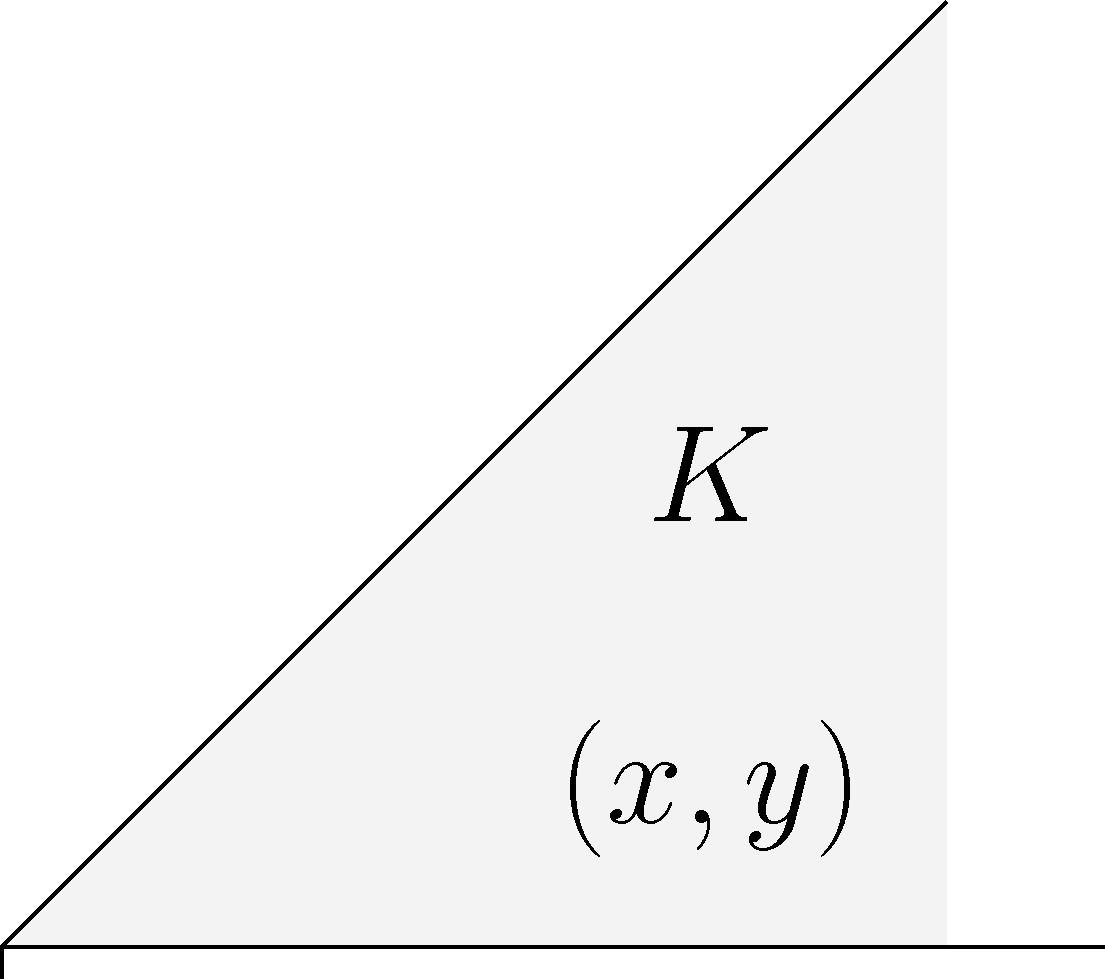}
  \end{minipage}
  \begin{minipage}{0.1\linewidth}
    \centering
    \includegraphics[width=1.7cm]{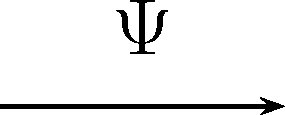}
  \end{minipage}
  \begin{minipage}{0.22\linewidth}
    \centering
    \includegraphics[width=1.7cm]{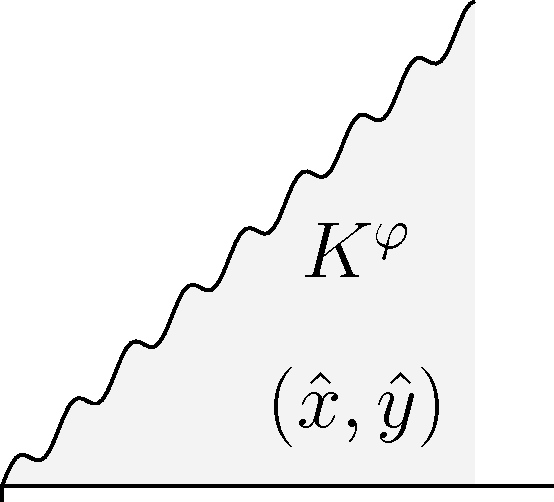}
  \end{minipage}
  \caption{Coordinate transform onto moving domain}
    \label{fig-psi}
\end{figure}

\begin{lemma} \label{lem-psi} %
  Let $k \in \N$ with $k \geq 1$, $\eps \in (0, \frac {\pi}{3(2k+1)})$. Then for any
  $\phi \in \XS{k}$ with $\XN{\phi}{k}$ sufficiently small, there is a
  diffeomorphism $\Psi: K \to K^\phi$ of the form
  \begin{align} \label{psi-form}  %
    \Psi(x,y) \ = \ (x,y) + (0, \psi(x,y)) \ = \ (\hat x, \hat y).
  \end{align}
  Furthermore, analogously as in the definition \eqref{Y-hom}, there is a
  decomposition $\psi = \psi_0 + \psi_1$ such that $\psi_1 = \zeta \PP$ and
  where $\PP$ is a polynomial of order $3k$ such that
  \begin{align} \label{est-psi} %
    \sup_{0 \leq \ell \leq k} \sum_{|\alp| = 3\ell+2} %
    \nltL{\sup_{0 \leq v \leq 1} \big| e^{\frac 12 \eps |\lambda| (1-v)}
      \mu^{\ell+1} \Lambda^{\alp} \widehat{\eps \psi_0} \big|}{\Re \lambda =
      3\ell+\frac 12} + \NNN{\PP}{\PP} \ %
    \leq \ C \XN{\phi}{k}
  \end{align}
  where we recall that $\mu = \inf \{ |\lam|, \frac 1\eps \}$. Furthermore, we
  have
  \begin{align} \label{psi-2} %
    \XN{\eps \psi_{x|\Gamma}}{k} + \XN{\psi_{y|\Gamma}}{k} \ \leq \ C
    \XN{\phi}{k}.
  \end{align}
\end{lemma}
\begin{proof}
  We construct $\psi$ to be the solution of
  \begin{gather} \label{def-psi} 
    \left \{
      \begin{array}{ll}
        \Delta_\eps \psi \ = \ 0 \qquad\qquad&\text{in } K, \\
        \psi \ = \ \int_0^x \phi dx, &\text{on } \partial_1 K, \\
        \psi  \ = \ 0, &\text{on } \partial_0 K.
      \end{array}
    \right.
  \end{gather}
  As in the previous section, the argument is based on a decomposition the right
  hand side into a polynomial part and a homogeneous remainder. Since this
  decomposition proceeds analogously as in the proof of Proposition
  \ref{prp-p-linear} we only consider the case of homogeneous data. That is we
  assume that $\phi \in \XSoo{k}$. With the transformation \eqref{def-ts},
  \eqref{def-psi} takes the form $\partial_u^2 \psi + ({\textstyle \frac
    1\eps \partial_v})^2 \psi = 0$ for $(u,v) \in \R \times (0,1)$. Furthermore,
  $\psi(u,1) = \int^{u} \phi e^{\tilde u} d\tilde u =: H(u)$ and $\psi(u,0) =
  0$. Application of the Laplace transform in $u$ and since $\lambda \hat
  H(\lambda) = \hat \phi(\lambda-1)$, the solution can be explicitly calculated
  as (cf. \eqref{p-strip}-\eqref{q-formula})
  \begin{align} \label{dertak} %
    \hat \psi(\lambda, v) \ %
    = \ \frac{\sin(\eps (\lambda) v)}{\sin(\eps (\lambda))} \hat H(\lam) \ %
    = \ \frac 1{\lam} \frac{\sin(\eps (\lambda) v)}{\sin(\eps (\lambda))} \hat
    \phi(\lam-1).
  \end{align}
  When taking derivatives of \eqref{dertak} in $v$, additional factors of $\eps
  \lam$ are created; furthermore the multiplier has either the cosinus or sinus
  in the denominator. We have for $v \in (0,1)$,
  \begin{align} \label{using-2} %
    \left|\frac{\sin(\eps \lambda v)}{\sin(\eps \lambda)} \right| \ %
    \leq \ \left|\frac{\cos(\eps \lambda v)}{\sin(\eps \lambda)} \right| \ %
    \leq \ \frac C{\eps \mu} e^{\eps |\lambda| (v-1)} \ %
    \leq \ \frac C{\eps \mu} e^{\frac 12 \eps |\lambda| (v-1)}. %
  \end{align}
    Indeed, for $|\eps \lam| \leq 1$, we have $\mu = |\lam|$, $|\frac {\cos
      (\eps \lam)}{\sin(\eps \lam)}| \leq \frac C{|\eps \lam|} = \frac C{\eps
      \mu}$ and $|e^{|\eps \lam| (v-1)}| \geq e^{-1}$. For $|\eps \lam| \geq 1$,
    we have $\mu = \frac 1\eps$ (and hence $\frac 1{\eps\mu} = 1$) $|\cos(\eps
    \lam v)| \leq 4 e^{|\eps\lam|v}$ and $|\sin(\eps \lam)| \geq \frac 14
    e^{|\eps\lam|}$. This proves \eqref{using-2}. Now, we multiply
    \eqref{dertak} by $|\lam|^{3k+1} \mu^k$, apply the $L^2$ norm on the line
    $\Re \lam = 3\ell - \frac 12$ and use \eqref{using-2}. This yields the
    estimate and in particular uniqueness.  This yields \eqref{est-psi} for
    (note that $1 \leq |\lam|$ and $1 \leq \mu$ for the considered values of
    $\lam$). This concludes the proof of \eqref{est-psi}. The estimate of
    \eqref{psi-2} follows directly by multiplication of \eqref{dertak} with
    $|\lam|$ and evaluation at $\Re (\lam+1)$ (corresponding to one derivative
    in $x$).
\end{proof}

\begin{lemma} \label{lem-remainder}  %
  Let $\Psi$ be the coordinate transform from Lemma \ref{lem-psi}. Then
  $P : K^\phi \to \R$ is a solution of \eqref{p-K} if and only if $p = P
  \circ \Psi : K \to \R$ satisfies
  \begin{align} \label{pRpsi} %
    \left \{
      \begin{array}{ll}
        \Delta_\eps p \ = \ R(p,\phi)  \qquad&\text{ in } K, \\
        p \ = \ f_x  \qquad&\text{ on } \partial_1 K, \\
        p_y \ = \ 0  \qquad&\text{ on } \partial_0 K,
      \end{array} 
    \right.
  \end{align}
  where the operator $R(p,\phi)$ is given by (using the notation $\qq := (1 +
  \psi_y)^{-1}$)
  \begin{align} \label{def-R}
    \begin{aligned}
      R(p,\phi) \ = \ &- \qq_x \psi_x p_y - \qq \psi_{xx} p_y - 2 \qq \
      \psi_x p_{xy} + \qq \qq_y \psi_x^2 p_y \\
      &+ \qq^2 \psi_x \psi_{xy} p_y + \qq^2 \psi_x^2 p_{yy} + ({\textstyle
        \frac 1\eps})^{2} \qq \qq_y p_y + (\feps)^{2}
      (\qq^2 - 1) p_{yy}.
    \end{aligned}
  \end{align}
\end{lemma}
\begin{proof}
  We can write the inverse coordinate transform $\Psi^{-1} : K^\phi \to K$ as
  $\Psi^{-1}(\hat x, \hat y) = (\hat x, \hat y + \eta(\hat x, \hat y))$ for some
  $\eta : K^\phi \to K$. In particular, $\hat y + \eta(\hat x, \hat y) +
  \psi(\hat x, \hat y + \eta(\hat x, \hat y)) = \hat y$.  By differentiating
  this equality in $\hat x$ and $\hat y$, we get $\psi_x + (1 + \psi_y)
  \eta_{\hat x} = 0$ and $(1 + \psi_y) (1 + \eta_{\hat y}) = 1$.  This implies
  $\frac {\partial x}{\partial \hat x} = 1$, $\frac{\partial x}{\partial \hat y}
  = 0$, $\frac{\partial y}{\partial x} = - \qq \psi_x$ and $\frac{\partial
    y}{\partial \hat y} = \qq$, in particular, $P_{\hat y} = 0 \Leftrightarrow
  P_y = 0$, justifying the boundary condition in \eqref{p-K}. Equation
  \eqref{def-R} follows from \eqref{p-K} together with
  \begin{align*}
    P_{\hat x \hat x} \ &= \ (\partial_x - \qq
    \psi_t \partial_y)(p_x - \qq \psi_x p_y) \ \\
    &= \ p_{xx} - \qq_x \psi_x p_y - \qq \psi_{xx} p_y - 2 \qq \ \psi_x p_{xy} +
    \qq \qq_y \ \psi_x^2 p_y + \qq^2 \psi_x \psi_{xy} p_y
    + \qq^2 \psi_x^2 p_{yy}, \\
    P_{\hat y \hat y} \ &= \ \qq (\qq \tilde p_y)_y \ = \ \qq^2 p_{yy} + \qq
    \qq_y p_y.
  \end{align*}
\end{proof}

\subsection{Estimates on the pressure} The main result of this section is:
\begin{proposition}[Shape dependence of $p$] \label{prp-p-nonlinear} %
  Let $k \in \N$ with $k \geq 1$, $\eps \in (0, \frac
  {\pi}{3(2k+1)})$ and $\phi \in \XS{k-1/2}$, $f \in \XS{k}$. Then
  there is a constant $\alp < 1$ such that if $\XN{\phi}{k-1/2} \leq
  \alp$, then there is a unique solution $P : K \to \R$ of
  \begin{align} \label{p-K} \left \{
      \begin{array}{ll}
        \Delta_\eps P = 0 \qquad &\text{ in } K^{\phi}, \\
        P = f_x \qquad &\text{ on } \partial_1 K^{\phi}, \\
        P_y = 0 \qquad &\text{ on } \partial_0 K^{\phi}.
      \end{array} 
    \right.
  \end{align}
  Furthermore, $p = P \circ \Psi \in \YS{k}$ satisfies
  \begin{align} \label{ppf} %
    \YN{p}{k} \ \leq \ C\XN{f}{k}.
  \end{align}
  Let $q$ be the solution of \eqref{p-wedge} with data $f$ and with right hand
  side $g = 0$. Furthermore, let $\tilde P$, $\tilde p$ and $\tilde q$ be the
  corresponding solutions with boundary condition $\tilde f$ instead of $f$ and
  let $w = p-q$ and $\tilde w = \tilde p - \tilde q$. Then
  \begin{align} \label{ppf-2} %
    \YN{w}{k} \ &\leq \ C \XN{\phi}{k-1/2} \XN{f}{k},  \notag \\
    \YN{p - \tilde p}{k} \ &\leq \ %
    C \XN{\phi - \tilde \phi}{k-1/2} (\XN{f}{k} + \XN{\tilde f}{k}) %
    + C \XN{f - \tilde f}{k},  \\
    \YN{w - \tilde w}{k} \ &\leq \ C \XN{\phi - \tilde\phi}{k-1/2} ( \XN{f}{k} +
    \XN{\tilde f}{k}) + C \XN{f - \tilde f}{k} (\XN{\phi}{k-1/2} +
    \XN{\tilde\phi}{k-1/2}). \notag
  \end{align}
\end{proposition}
Before we address the proof of Proposition \ref{prp-p-nonlinear}, we give an
estimate of the nonlinear term $R(p,\phi)$ defined in \eqref{def-R}:
\begin{proposition} \label{prp-RR} %
  Let $k \in \N$, $k \geq 1$, $\eps \in (0, \frac
  {\pi}{3(2k+1)})$. Let $\phi, \tilde \phi \in \XS{k-1/2}$ with
  $\XN{\phi}{k-1/2} \leq c_k$, $\XN{\tilde \phi}{k-1/2} \leq c_k$ for some
  sufficiently small universal constant $c_k$. Let $p \in \YS{k}$ be the
  solution of \eqref{p-K} on $K^\phi$ with boundary data $f$, let $\tilde p \in
  \YS{k}$ be the corresponding solution on $K^{\tilde \phi}$ with boundary data
  $\tilde f$. Then
  \begin{align} \label{einfach} %
    \begin{aligned}
      \ZN{R(p, \phi) - R(\tilde p, \tilde \phi)}{k} \ %
      &\leq \ C \YN{p - \tilde p}{k}(\XN{\phi}{k-1/2} +
      \XN{\tilde \phi}{k-1/2})  \\
      &\qquad\qquad+ C \XN{\phi - \tilde \phi}{k-1/2} (\YN{p}{k} + \YN{\tilde
        p}{k}),
    \end{aligned}
  \end{align}
  where $R(p,\phi)$ is defined in \eqref{def-R}.
\end{proposition}
\begin{proof}
  The proof uses some ideas of the proof of the algebra property in Lemma
  \ref{lem-algebra}.
%\cite[Proposition 1.2(4)]{KnuepferMasmoudi-2012a}
  There are two differences: We also need to take a supremum in the angular
  variable $v$. The factors in $R$ are controlled by different norms.
  Furthermore, the norm $g$ controls a discrete set of homogeneous norms
  (cf. \eqref{Z-norm}). We will show that
  \begin{align} \label{RR-est} %
    \ZN{R (p,\phi) }{k} \ \leq \ C \XN{\phi}{k-1/2} \YN{p}{k}
% \ \leq \  C \XN{\phi}{k-1/2} \XN{f}{k},
  \end{align}
  the proof of \eqref{einfach} follows by an analogous argument, using the
  multilinear structure of $R$.  We use the representation
  \eqref{def-R}. Recall the definition \eqref{def-psi} of the coordinate
  transform $\psi$.  We claim that \eqref{RR-est} follows by iterative
  application of
  \begin{align}\label{the-est-1} %
    \ZN{\rho g}{k} + \ZN{(1-\gamma) g}{k} %
    &\leq \ C \XN{\phi}{k-1/2} \ZN{g}{k}, \\
    \ZN{\rho_x w}{k} + \ZN{(\feps \rho_y) w}{k} \ %
    &\leq \ C \XN{\phi}{k-1/2} \big( \ZN{w_x}{k} + \ZN{\feps w_y}{k} \big), \label{the-est-2}
  \end{align}
  where $\rho$ is either one of the terms $\eps \psi_x$ or $\psi_y$ and where we
  recall that $\gamma = (1 + \psi_y)^{-1}$.

  \medskip

  Assuming that \eqref{the-est-1}--\eqref{the-est-2} hold, the proof of
  \eqref{RR-est} is easy:  Recall that with the definitions of the norms
  \eqref{Y-hom}, \eqref{Z-norm} and by \eqref{ppf} we have
  \begin{align} \label{nochnest} %
    \ZN{p_x}{k} + \ZN{\feps p_y}{k} + \ZN{p_{xx}}{k} + \ZN{\feps p_{xy}}{k} +
    \ZN{(\feps)^2 p_{yy}}{k} \ \leq \ C \YN{p}{k}.
  \end{align}
  We show how the estimate of \eqref{RR-est} proceeds for the term $\gam_x
  \psi_x p_y$ (the first term on the right hand side of \eqref{def-R}): In view
  of $\gam_{x} = - \gam^2 \psi_{xy}$, we have
  \begin{align*}
    \ZN{\gam_x \psi_x p_y }{k} \ %
    = \ \ZN{\gam^2 (\eps \psi_x) \psi_{xy} (\feps p_y)}{k} \ %
    \lupref{the-est-1}\leq \ C \ZN{\psi_{xy} (\feps p_y) }{k} %
    \lupupref{the-est-2}{nochnest}\leq \ C \XN{\phi}{k-1/2} \YN{p}. %
  \end{align*}
  Indeed, it can be easily checked that the estimate of the other terms in
  \eqref{def-R} proceeds analogously.  In order to conclude the proof of the
  Proposition, it hence remains to give the argument for
  \eqref{the-est-1}--\eqref{the-est-2}.

  \medskip

  {\it Proof of \eqref{the-est-1}:} In view of the Taylor expansion $\gamma - 1
  = \psi_y - \psi_y^2 + \ldots$, the estimate for the term $(1-\gamma) g$ in
  \eqref{the-est-1} hence follows by the corresponding estimate for the term
  $\psi_y g$ together with the estimate \eqref{est-psi}, $\XN{\phi}{k-1/2} \leq
  c_k$ and for $c_k$ sufficiently small. In order to see \eqref{the-est-1}, it
  hence remains to show
  \begin{align} \label{essence} %
    \ZN{\rho g}{k} \ \leq \ C \XN{\phi}{k-1/2} \ZN{g}{k}.
  \end{align}
  Let $\PP_\psi$ and $\PP_g$ be the Taylor polynomials at $(x,y)=(0,0)$ of
  $\psi$ and $g$ (of order $3k-1$ and $3k-4$, respectively).  We decompose $\psi
  = \psi_1 + \psi_0$ and $g = g_1 + g_0$, where $\psi_1 := \PP_\psi \zeta$ and
  $g_1 := \PP_g \zeta$ and where the radial cut-off function $\zeta =
  \zeta(|x,y)|) \in \cciL{[0,\infty)}$ satisfies $\zeta = 1$ in $[0,\frac 18]$
  and $\zeta = 0$ in $[\frac 14,\infty]$.  The terms $\rho_0$ and $\rho_1$ are
  defined correspondingly. Clearly, it is enough to show the corresponding
  estimate to \eqref{essence} for the products $\rho_1 g_1$, $\rho_1 g_0$,
  $\rho_0 g_1$ and $\rho_0 g_0$. In the following, we will give the estimate for
  $\rho_0 g_0$; the estimate of the other terms (related to finite dimensional
  Taylor expansion) proceeds analogously as in the proof of the algebra property
  of $\XS{k}$ in Lemma \ref{lem-algebra}. %\cite{KnuepferMasmoudi-2012a}.
  Furthermore, for simplicity of notation, we give the proof with $\Lam$
  replaced by $|\lam|$ in \eqref{Z-hom}, i.e. when only radial derivatives
  appear . The argument in the case of angular derivatives ${\frac
    1\eps} \partial_v$ proceeds by distributing the derivatives on the two
  factors using Leibniz' rule.  We will use the notation $\EEP{|\lam|} =
  \EP{\lam}$, in particular by the triangle inequality we have $\EEP{|\lam|}
  \leq \EEP{|\lam-\eta|}\EEP{|\eta|}$.

    \medskip

    Analogously to \eqref{def-fpm}, we decompose the functions $\psi_0$ and
    $g_0$ into their low and high frequency part, i.e. $\psi_0 = \psi_+ +
    \psi_-$ (with $\rho_\pm$ defined correspondingly) and $g_0 = g_+ + g_-$. In
    particular, as in \eqref{est-fpm}, we have
    \begin{align} \label{rp-freq} %
      |\lambda|^k |\widehat{g_+}| + \left(\feps \right)^k
      |\widehat{g_-}| \ \leq \ C \mu^k |\widehat {g_0}|, && %
      |\lambda|^k |\widehat{\rho_+}| + \left(\feps \right)^k
      |\widehat{\rho_-}| \ \leq \ C \mu^k |\widehat{\rho_0}|.
    \end{align}
    We need to estimate the terms $g_+ \rho_+$, $g_+ \rho_-$, $g_- \rho_+$ and
    $g_- \rho_-$. We show the estimate for the high frequency/high frequency
    product $g_- \rho_-$. Indeed, the estimate for the other two terms proceeds
    similarly (see also the algebra proof in \cite{KnuepferMasmoudi-2012a}).

    \medskip

    In view of \eqref{Z-hom} and \eqref{rp-freq}, we need to show for all $\ell
    \in [\frac 23,k]$,
    \begin{align}
      \ZMMoo{\rho_- g_-}{\ell} \ %
      &= \ C \nltL{ \sup_{v \in (0,1)} \big| \int_{\Re \eta = \gamma}
        \EEP{|\lam|} (\feps)^\ell \lambda^{3\ell - 2}
        \widehat{\rho_-}(\lambda-\eta) \widehat{g_-}(\eta) \DI{\eta}\big|}{\Re
        \lambda = 3\ell- \frac 72} \notag \\ %
      &\leq \ C \XN{\phi}{k-1/2} \ZN{g}{k}, \label{lhs-of}
  \end{align}
  for any $\gam \in \R$ of our choice such the above integral is defined. Let
  $\kappa \in [0,1)$ be the smallest number such that $N := 3\ell-2 + \kappa \in
  \N_0$. We show the corresponding stronger estimate to \eqref{lhs-of} where the
  $\lam^{3\ell - 2}$ is replaced by $\lam^N$. The estimate is stronger since
  $|\lam| \geq \frac 12$ in the line of integration in \eqref{lhs-of}. The
  advantage is that the binomial formula $\lambda^N = \sum_{j=0}^N \ C_{Nj}
  (\lambda-\eta)^j \eta^{N-j}$ with $N = 3\ell - 2 \in \N_0$ can be applied. We
  have
  \begin{align*} %\label{da-int-pq} %
    &\int_{\Re \lambda = 3\ell- \frac 72} \SP \left| \int_{\Re \eta = \gamma}
      \EEP{|\lam|} \lambda^{3\ell - 2+\kappa}
      \widehat{\rho_-}(\lambda-\eta)
      \widehat{g_-}(\eta) \DI{\eta}  \right|^2 \DI{\lambda}\\
    &= \int_{\Re \lambda = 3\ell- \frac 72} \SP \left| \sum_{i = 0}^{3\ell - 2}
      C_{Nj} \int_{\Re \eta = \gamma} \EEP{|\lam|} (\lambda-\eta)^{3\ell-2+\kappa-i}
      \hat \rho_-(\lambda-\eta) \eta^{i} \widehat{g_-}(\eta) \DI{\eta} \right|^2
    \DI{\lambda}.
    \end{align*}
    Note that the above inner integrand is analytic in $\eta$ and hence does not
    depend on the value of $\gamma$ as long as the integral is well-defined;
    hence we may choose $\gam$ freely as a function of $i$.  In turn, since
    furthermore $\EEP{|\lam|} \leq \EEP{|\lam-\eta|}\EEP{|\eta|}$, it is enough
    to estimate terms of the form
    \begin{align} \label{k-weg} %
      \hspace{6ex} & \hspace{-6ex} %
      \int_{\Re \lambda = 3\ell- \frac 72} \SP \left| \int_{\Re \eta = \gamma}
        \EEP{|\lam-\eta|} |\lambda-\eta|^{3\ell+\kappa-2-i}
        |\widehat{\rho_-}(\lambda-\eta)| \EEP{|\eta|} |\eta|^{i}
        |\widehat{g_-}(\eta)| \DI{\eta} \right|^2 \DI{\lambda}
    \end{align}
    for $\ell \in [\frac 23, k]$, all integers $i \in [0, 3\ell - 2 +\kappa]$
    and with our choice of $\gam \in \R$. We next apply the following variant of
    \eqref{often-1} which says that for all $\tilde \delta > \frac 12$, we have
    \begin{align} \label{often-2} %
      \begin{aligned}
        \nltL{\sup_{v \in (0,1)} |\hat F * \hat G|}{\Re \lam = \bet}\ %
        &\leq \ C_\delta \nltL{\sup_{v \in (0,1)} (1+|\lam|^{\tilde
            \delta})|\hat F|}{\Re \lam = \bet_1} \nltL{\sup_{v \in (0,1)} |\hat
          G|}{\Re \lam = \bet_2} \hspace{-1.9ex}
      \end{aligned}
    \end{align}
    if $\bet_1 + \bet_2 = \bet$ and as long as all integrals are
    well-defined. We introduce the short notation $\cnltL{\hat \phi}{\bet} =
    \nltL{\sup_{v \in (0,1)} \EEP{|\lam|} |\hat \phi|}{\Re \lam = \bet}$.  In
    view of \eqref{k-weg} and \eqref{often-2}, for the proof of \eqref{lhs-of}
    it suffices to show for all $\ell \in [\frac 23,k]$ and for all integers $i
    \in [0,3\ell-2+\kappa]$,
    \begin{align} \label{diesda-old} %
      % \min \Big \{ 
      \cnltL{(1+|\lam|^{3\ell+ \kappa-2-i +\tilde \delta})
        |\widehat{\rho_-}|}{\bet_1} \cnltL{ |\lam|^{i} \widehat{g_-}}{\bet_2} \
      \leq \ \XN{\phi}{k-1/2} \ZN{g}{k},
    \end{align}
    where we can arbitrarily choose $\tilde \delta > \frac 12$ and $\bet_1,
    \bet_2$ with $\bet_1 + \bet_2 = 3\ell - \frac 72$.  With the notation
    $\delta = \tilde \delta + (\kappa - 1) < \tilde \delta$, it is equivalently
    enough to show .
    \begin{align} \label{diesda} %
      % \min \Big \{ 
      \cnltL{(1+|\lam|^{3\ell-1-i +\delta}) |\widehat{\rho_-}|}{\bet_1}
      \cnltL{ |\lam|^{i} \widehat{g_-}}{\bet_2} \ \leq \ \XN{\phi}{k-1/2}
      \ZN{g}{k},
    \end{align}
    where we can arbitrarily choose $\delta \geq \frac 12$ and $\bet_1, \bet_2$
    with $\bet_1 + \bet_2 = 3\ell - \frac 72$. Both $\bet_1$, $\bet_2$ as well
    as $\delta$ are allowed to depend on $\ell$ and $i$. In fact, in the sequel
    we will always choose $\delta = \frac 12$.

    \medskip

    Recall that either $\rho_0 = \eps \psi_{0x}$ or $\rho_0 = \psi_{0y}$ and
    hence either $|\widehat{\rho_0}(\lam)| = |\lam \widehat{\psi_0}(\lam+1)|$ or
    $|\widehat{\rho_0}(\lam)| = |\partial_v \widehat{\psi_0}(\lam+1)|$. By
    \eqref{est-psi} and by the same argument as in the proof of Lemma
    \ref{lem-real}, we hence have
   \begin{align} \label{est-psi-2} %
     \cnltL{(1+ |\lam|^{3\ell+1} \mu^{\ell+1})
       |\widehat{\rho_0}|}{3\ell-\frac 12} \ \leq \ C
     \XN{\phi}{k-1/2}, && %
     \text{$\forall \ell \in [0,k-\tfrac 12]$.}
  \end{align}
    By \eqref{est-psi-2} and \eqref{rp-freq}, we have
    \begin{align} \label{control-} %
      \cnltL{(1 + (\feps)^{\ell+1}|\lam|^{3\ell+1}) |\widehat{\rho_-}|}
      {3\ell-\frac 12} \ %
      \leq \ C \XN{\phi}{k-1/2} && %
      \text{$\forall \ell \in [0, k-\tfrac 12]$.}
     \end{align}
     Furthermore, in view of \eqref{Z-hom} and by \eqref{rp-freq}, we also have
    \begin{align} \label{g-control} %
      \cnltL{(\feps)^\ell|\lam|^{3\ell-2}
        |\widehat{g_-}|}{3\ell-\frac 72} \ %
      \leq \ C \ZNP{g}{k}, && %
      \text{$\forall \ell \in [\tfrac 23, k]$}.
    \end{align}
    We prove \eqref{diesda} for the three extreme cases, i.e. the corners of the
    triangle in $(\ell, i)$ where $\ell \in [\frac 23, k]$ and $i \in
    [0,3\ell-2] \}$, the estimate of the other terms follows by 'interpolation'
    of these estimates:

    \medskip
    
    \noindent(a1) $\ell = \frac 23$, $i=0$ (and hence $\bet = 3\ell-\frac 72 =
    -\frac 32$): With the choice $\bet_1 = 3 \cdot \frac 16 - \frac 12 = 0$,
    $\bet_2 = 3 \cdot \frac 23 - \frac 72 = -\frac 32$ and $\delta = \frac 12$,
    we need to estimate
    \begin{align*} %
      \cnltL{(1 + |\lambda|^{\frac 32})|\widehat{\rho_-}|}{0} %
      \cnltL{(\feps)^{\frac 23} |\widehat{g_-}|}{-\frac 32} \ \leq \ C
      \XN{\phi}{k-1/2} \ZN{g}{k}.
    \end{align*}
    (b1) $\ell = k$, $i = 0$: With the choice $\bet_1 = 3(k-\frac 12) - \frac 12
    = 3k-2$, $\bet_2 = 3 \cdot \frac 23 - \frac 72 = -\frac 32$ and $\delta =
    \frac 12$, we need to estimate
   \begin{align*} %
     &\cnltL{(1 + |\lambda|^{3k-\frac 12} (\feps)^{k - \frac 23})
       |\widehat{\rho_-}|}{3k-2} \ %
     \cnltL{(\feps)^{\frac 23} |\widehat{g_-}|}{-\frac 32} \ %
     \leq \ C \XN{\phi}{k-1/2} \ZN{g}{k}.
   \end{align*}
   (c1) $\ell = k$, $i = 3k-2$: With the choice $\bet_1 = 3 \cdot \frac 16 -
   \frac 12 = 0$, $\bet_2 = 3k-\frac 72$ and $\delta = \frac 12$, we need to
   estimate
   \begin{align*} %
     \cnltL{(1 + |\lam|^{\frac 32}) |\widehat{\rho_-}|}{0} %
     \cnltL{|\lam|^{3k-2} (\feps)^k |\widehat{g_-}| }{3k-\frac 72} \ %
     \leq \ C \XN{\phi}{k-1/2} \ZN{g}{k}.
   \end{align*}
   Now, using \eqref{control-} and \eqref{g-control}, it can be easily checked
   that the above estimates (a1), (b1) and (c1) hold true for $k \geq 1$. This
   concludes the proof of \eqref{essence} and hence of \eqref{the-est-1}.

   \medskip

   {\it Proof of \eqref{the-est-2}:} The proof of \eqref{the-est-2} proceeds
   similarly to the proof of \eqref{the-est-1}. As before, we show the estimate
   for the crucial high frequency/high frequency case $\rho_{-,x} w_-$.  With the
   same arguments as before, we need to show (correspondingly to \eqref{diesda}):
   \begin{align} \label{diesda-2} %
     \hspace{6ex} & \hspace{-6ex} %
     (\feps)^k \cnltL{|\lam|^{3\ell-1-i+\kappa} |\widehat{\rho_-}|}{\bet_1} %
     \cnltL{(1 + |\lam|^{i+\delta})|\widehat{w_-}|}{\bet_2} \ %
     \notag\\
     &\leq \ C \XN{\phi}{k-1/2} \big( \ZN{w_x}{k} + \ZN{\feps w_y}{k} \big) \ %
     =: \ \mathcal R,
    \end{align}
    (we have put the $\delta$ on the second factor).  In the above inequality,
    we can arbitrarily choose $\bet_1, \bet_2$, $\delta$ as long as $\bet_1 +
    \bet_2 = 3\ell - \frac 52$ and $\delta \geq \frac 12$. Both $\bet_1$,
    $\bet_2$ as well as $\delta$ may depend on $\ell$, $i$.  We note that by
    \eqref{Z-hom}, we have
    \begin{align} \label{w-control} %
      \cnltL{(\feps)^\ell|\lam|^{3\ell-1} |\widehat{w_-}|}{3\ell-\frac 52} \ %
      \leq \ C \big(\ZNP{w_x}{k} + \ZNP{\feps w_y}{k} \big), && %
      \text{$\forall \ell \in [\tfrac 12, k]$}.
    \end{align}
    We prove \eqref{diesda-2} in the case when the maximum number of derivatives
    fall onto $\rho$, i.e. $\ell = k$ and $i = 0$ (in particular, $\kappa = 0$
    where we recall that $\kappa$ is defined as the smallest nonnegative integer
    such that $3\ell - 2 \in \N$). With the choice $\bet_1 = 3(k-\frac 12) -
    \frac 12 = 3k-2$, $\bet_2 = 3 \cdot 1 - \frac 52 = \frac 12$ and $\delta =
    1$, we hence need to estimate
   \begin{align*} %
     &\cnltL{(\feps)^{k} |\lambda|^{3k-1} |\widehat{\rho_-}|}{3k-2} \ %
     \cnltL{(1+|\lam|) (\feps)^{\frac 23} |\widehat{g_-}|}{\frac 12} \ \leq \ C
     \mathcal R.
   \end{align*}
   This estimate holds true as can easily be checked using \eqref{control-} and
   \eqref{w-control}. The estimate of the terms $\rho_{-,x} w_-$, $\rho_{-,x}
   w_-$ and $\rho_{-,x} w_-$ proceeds similarly. This concludes the estimate of
   \eqref{the-est-2} and hence of the proposition.
\end{proof}
% We turn to the proof of Proposition \ref{prp-p-nonlinear}:
\begin{proof}[Proof of Proposition \ref{prp-p-nonlinear}]
  We first show the existence of a solution $P$ of \eqref{p-K} on $K^\phi$. By
  Lemma \ref{lem-remainder}, for this we need to find a solution $p$ of
  \eqref{pRpsi}.  We will solve \eqref{pRpsi} using an iterative argument.  We
  set $p_0 := 0$ and iteratively define $p_{i+1}$ to be the solution of
  \eqref{p-wedge} with right hand side $g = R(p_{i-1}, \phi)$ and boundary
  data $f_x$.  By \eqref{est-p} and \eqref{einfach}, we get
  \begin{align*}
    \YN{p_{i+1} - p_i}{k} \ %
    &\upref{est-p}\leq \ C \ZN{R(p_i, \phi) - R(p_{i-1}, \phi)}{k} \ %
    \lupref{einfach}\leq \ C \YN{\phi}{k-1/2} \YN{p_{i+1} - p_i}{k}  \\
    &\leq \ \ C \alp \YN{p_{i+1} - p_i}{k} \ %
    \leq \ \ \tfrac 12 \YN{p_{i+1} - p_i}{k}
  \end{align*}
  for $\alp$ sufficiently small, $\{ p_i \}_{i \in \N}$ is a Cauchy sequence
  and converges to a solution $p$ of \eqref{pRpsi}. By \eqref{est-p}, $p$
  satisfies \eqref{ppf}. The estimates \eqref{ppf-2} now follow from the
  representation \eqref{pRpsi} of the solution together with the estimates
  \eqref{einfach} and \eqref{RR-est}.
\end{proof}
We also have the nonlinear version of the trace estimate in Lemma
\ref{lem-pressure-trace}:
\begin{lemma} \label{lem-trace-nonlinear} %
  Suppose that the assumptions of Proposition \ref{prp-p-nonlinear} are
  satisfied (in particular, $k \geq 1$). Then with the notation of Proposition
  \ref{prp-p-nonlinear}, we have
  \begin{align} \label{nl-trace-1} %
    \XN{\partial_x(p_x)_{|\Gamma}}{k-1} + (\feps)^2
    \XN{\partial_x(p_y)_{|\Gamma}}{k-1} \ \leq \ C \XN{f}{k}.
  \end{align}
\end{lemma}
\begin{proof}
  Indeed, \eqref{nl-trace-1} follows by application of \eqref{est-p-trace} and
  since $p$ satisfies \eqref{pRpsi}.
\end{proof}
Note that the corresponding 'bilinear' estimates for $w$, $p - \tilde p$ and $w
- \tilde w$ also hold (correspondingly as in Proposition \ref{prp-p-nonlinear}):
The estimates for the case of $x$ derivative are
\begin{align} \label{nl-trace-2} %
  \XN{\partial_x(w_x)_{|\Gamma}}{k-1} \ &\leq \ C \XN{\phi}{k-1/2} \XN{f}{k},  \notag \\
  \XN{\partial_x(p_x)_{|\Gamma} - \partial_x(\tilde p_x)_{|\Gamma}}{k-1} \ &\leq
  \ %
  C \XN{\phi - \tilde \phi}{k-1/2} (\XN{f}{k} + \XN{\tilde f}{k}) %
  + C \XN{f - \tilde f}{k},  \\
  \XN{\partial_x(w_x)_{|\Gamma} - \partial_x(\tilde w_x)_{|\Gamma}}{k-1} \ %
  &\leq \ C \XN{\phi - \tilde\phi}{k-1/2} ( \XN{f}{k} + \XN{\tilde f}{k}) \notag \\
  &\qquad+ C \XN{f - \tilde f}{k} (\XN{\phi}{k-1/2} +
  \XN{\tilde\phi}{k-1/2}). \notag
\end{align}
These estimates follow analogously to \eqref{nl-trace-1} but by using the
'bilinear' estimates \eqref{einfach} for $R$. The corresponding estimates to
\eqref{nl-trace-2} with $\partial_x(w_x)_{|\Gamma}$ and $\partial_x(\tilde
p_x)_{|\Gamma}$ replaced by $(\feps)^2 \partial_x(w_y)_{|\Gamma}$ and
$(\feps)^2 \partial_x(w_y)_{|\Gamma}$, respectively, also hold.

\subsection{Estimates for the profile} \label{ss-nonlinear}

In this section, we prove Proposition \ref{prp-nonlinear}. We show that
\begin{align} \label{nl-2} %
  \PXN{N_\eps(f)}{k} + \TXN{N_\eps(f)}{0} \ \leq \ C \PXN{f}{k+1}^2.
\end{align}
The proof of \eqref{nl-1} then follows by a straightforward extension
of this estimate (indeed the estimates \eqref{einfach} show that the
main part $\partial_x \BB_\eps \partial_x$ of the operator $\NN_\eps$
can be estimated like a bilinear operator).  We recall,
\begin{align*}
 N_\eps(f) &= \ \frac{f_x}{(1+\eps^2)^{\frac
      32}} \bigg(\frac{f_{xx} - 3\eps^2 f_x^2}{1+\eps^2} \bigg)_{|x=0}
  + \partial_x B_\eps^f \bigg( \frac{f_x}{(1+\eps^2(1 + f)^2)^{\frac 32}}\bigg)
  + A_\eps f,
\end{align*}
cf. \eqref{def-N}. The two main estimates in favor of \eqref{nl-2} are:
\begin{align} %
  \PXN{f_{xx|x=0} f_{x}}{k} + \TXN{f_{xx|x=0} f_{x}}{k} \ &\leq \ C \PXN{f}{k+1}^2, \label{todo-1} \\
  \PXN{\partial_x B_\eps^f \partial_x f + A_\eps f}{k} + \PXN{\partial_x
    B_\eps^f \partial_x f + A_\eps f}{k} \ &\leq \ C
  \PXN{f}{k+1}^2.\label{todo-2}
\end{align}
Two ``nonlinear'' corrections have been neglected in the estimates
\eqref{todo-1} and \eqref{todo-2}: One correction is related to the
$-3\eps^2f_x^2|_{|x=0}$ term, the other correction is related to the term $f_x
((1+\eps^2(1 + f)^2) - 1)$. Indeed, these corrections are easily controlled as
lower order terms for sufficiently small $f$; the estimate of these terms is
left to the reader. In the following, we will give the proof of \eqref{todo-1}
and \eqref{todo-2}.
 
\medskip

\noindent{\it Proof of \eqref{todo-1}: } In view of the definition
\eqref{def-PXN}, \eqref{todo-1} follows from
\begin{align} \label{with-time} %
  \TXM{\partial_t^i (f_{xx|x=0}f_{x})}{j} \ \leq \ C \PXN{f}{k+1}^2, &&
  \text{for $0 \leq i+j \leq k$.}
\end{align}
In order to see \eqref{with-time}, we note that by \eqref{X-sob}, we have
\begin{align} \label{as-1} %
  &\XN{f_{xx|x=0}f_{x}}{0} \ = \ \nlt{f_{xx|x=0}f_{x}} \ \leq \ \nlt{f_x}
  \nco{f_{xx}} \
  \lupref{X-sob}\leq \ C \XN{f}{1/2} \XN{f}{1}, \\
  &\XN{f_{xx|x=0}f_x}{j} \ \leq \ \XN{f_x}{j} \nco{f_{xx}} \
  \lupref{X-three}\leq \ C \XN{f}{j} \XN{f}{j+1} \qquad \text{for $1 \leq j \leq
    k$}.
  \label{as-2} %
\end{align}
We take the square of \eqref{as-1} and integrate the equation in time. We then
apply H\"olders inequality with $L^2$ on the $\XN{\cdot}{1}$ term and $L^\infty$
on the $\XN{\cdot}{1/2}$ term.  In view of \eqref{est-time-trace}, this yields
\eqref{with-time} for $i=j=0$. The estimate for $i > 0$ or $j > 0$ follows from
\eqref{as-2}: We take the square of \eqref{as-2} and integrate in time. We then
apply H\"older's inequality with $L^2$ on the $\XN{\cdot}{j+1}$ term and
$L^\infty$ on the $\XN{\cdot}{j}$ term. In view of \eqref{est-time-trace}, this
yields \eqref{with-time} for $i=0$ and $1 \leq j \leq k$. It remains to prove
\eqref{with-time} for $i > 0$: In this case we note that the time derivatives on
the left hand side of \eqref{with-time} are distributed according to Leibniz'
rule on the two factors. The loss of regularity due to the time derivation is
compensated by the fact that we only need to estimate the smaller norm $j = \ell
- i$. In view of the definition of $\PXN{\cdot}{k+1}$, this yields
\eqref{with-time} for all $0 \leq i \leq \ell$ and $j \geq 1$. This concludes
the proof of \eqref{todo-1}.

\medskip

\noindent{\it Proof of \eqref{todo-2}:} %
With the same argument as before, the proof of \eqref{todo-2} can be reduced to
the following estimate which does not involve time:
\begin{align} \label{no-time-proof} %
  \XM{\partial_x B_\eps^f \partial_x f + A_\eps f}{j} %
  \ \leq \ C \XN{f}{j+1/2} \XN{f}{j+1} && \text{for $0 \leq j \leq k$},
\end{align}
where we may assume that $\XN{f}{k+1/2} \leq \alp \leq 1$ (this corresponds to
the estimate $\CXN{f}{k-1/2} \leq \alp \leq 1$ for the corresponding time-space
estimate). Let $\psi$, $\gam$ be defined in Lemma \ref{lem-psi} and let $p$, $q$
be defined as in Proposition \ref{prp-p-nonlinear}. By \eqref{def-A},
\eqref{def-q} and \eqref{def-bt}, we have
\begin{align*}
  A_\eps f_x %
  \ &= \ (1 + \eps^2)^{-\frac 32} \partial_x \big[q_x - (\feps)^{2} q_y
  \big]_{|\Gamma}, \\
  - \partial_x B_\eps^f f_x \ %
  &= \ (1 + \eps^2)^{-\frac 32} \partial_x \big[(1+f) \ (p_x - \gamma \psi_x
  p_y) - (\feps)^{2}\gamma p_y \big]_{|\Gamma}.
\end{align*}
With the notation $w = p-q$ and $C_\eps = (1 +
\eps^2)^{-3/2} \leq 1$, we hence get
\begin{align} \label{dasda} %
  A_\eps f + \partial_x B_\eps^f f_x \ = \ C_\eps \partial_x \Big[ (\feps)^{2}
  w_y - w_x - f p_x + (1+f) \gamma \psi_x p_y + (\feps)^{2} (\gamma-1) p_y
  \Big]_{|\Gamma}.
\end{align}
We first note that for all $j \geq 0$, we have
\begin{align} \label{davor} %
  \XN{(w_{x|\Gamma})_x}{j} + \XN{(\feps)^2 (w_{y|\Gamma})_x}{j} \ %
  \lupref{nl-trace-2}\leq \ \ C \XN{f}{j+1/2} \XN{f}{j+1}.
\end{align} %
For the other terms on the right hand side of \eqref{dasda} and for $j \geq 1$
we additionally use the algebra property \eqref{X-alg}. The estimate of the term
$f (p_{x|\Gamma})_x$ proceeds as follows
\begin{align*}
  \XN{(f p_{x|\Gamma})_x}{j} \ %
  &\upref{X-alg}\leq \ \XN{f_x}{j} \XN{p_{x|\Gamma}}{j} + \XN{f}{j}
  \XN{(p_{x|\Gamma})_x}{j} \ %
  \lupupref{X-three}{nl-trace-1}\leq \ C \XN{f}{j+1/2} \XN{f}{j+1}, %
\end{align*} %
where we also used $\nlt{p_{x|\Gamma}} \leq C \YN{p}{1} \ \leq C
\XN{f}{1}$ which follows from \eqref{Y-sob} and \eqref{ppf}. We also
note that in view of \eqref{psi-2} we have $\XN{\eps
  \psi_{x|\Gamma}}{k+1/2} + \XN{\psi_{y|\Gamma}}{k+1/2} \leq C \alp <
1$ if the constant $\alp$ in the assumption of Theorem \ref{thm-wedge}
is chosen sufficiently small. In view of the Taylor expansion $1 -
\gamma = \psi_y - \psi_y^2 + \ldots$ and by \eqref{psi-2}, this
implies also $\XN{\gamma_{|\Gamma} - 1}{j} \leq
\XN{\psi_{y|\Gamma}}{j} \leq C$. For $j \geq 1$, the estimate of the
remaining terms in \eqref{dasda} then follows using these estimates
together with \eqref{X-alg}, \eqref{psi-2} and \eqref{nl-trace-1}.

\medskip

It remains to give the estimate for $j = 0$, where the algebra property
\eqref{X-alg} does not apply: We show the estimate for the nonlinear term $(f
p_{x|\Gamma})_x$, the estimate for the other terms proceeds analogously. Indeed,
we have
\begin{align*}
  \nlt{(f p_{x|\Gamma})_x} \ %
  \leq \ \nco{f} \nlt{(p_{x|\Gamma})_x} + \nlt{f_x} \nco{p_{x|\Gamma}} \ %
  \upref{nl-trace-1}\leq \ C \XN{f}{1/2} \XN{f}{1},
\end{align*}
where we also used $\nlt{p_{x|\Gamma}} \leq C \YN{p}{1} \ \leq C
\XN{f}{1}$ which follows from \eqref{Y-sob} and \eqref{ppf}. The
estimate of the other terms proceeds analogously. This concludes the
proof of \eqref{no-time-proof} and hence of the lemma.

\section{Localization argument} \label{sec-local}

We prove Propositions \ref{prp-local-linear}-\ref{prp-local-nonlinear} thus
concluding the proof of Theorem \ref{thm-local} in Section
\ref{sec-overview}. We will use the notation used in the proof of this
theorem. We first note that the derivative $\DLL_\eps^w$ of $\mathcal L_\eps$ at
$w$ is given by
\begin{align} \label{L-der} %
  \DLL_\eps^w f \ = \ L_\eps^w f + K_\eps^w f,
\end{align}
cf. \eqref{def-LL1}, where the operators $L_\eps$ (leading order) and $K_\eps^w$
(remainder) are given by
\begin{align}
  \displaystyle L_\eps^w f \ = \ f_t - \frac 1{(1+\eps^2)^{3/2}}\big(B_\eps^w
  f_{x} \big)_x \ %
  =: \ f_t + A^w_\eps f
\end{align}
and
\begin{align}
  K_\eps^w f \ &= \ (D^w - 1) f_t + \dot D^w f + \partial_t (\DD^w f w) \\
  &\qquad - \big[ \big((1 - x) \dot s_- + x \dot s_+ \big) f \big]_x - \big[
  \big((1 - x) \dot \sigma_{-}^w + x \dot \sigma_{+}^w \big) w
  \big]_x \\
  &\qquad - \Big[ B_\eps^w \big(\frac{3w_x (2 \eps^2 w_xf)}{2(1+\eps^2 (w +
    f^*)^2)^{\frac 52}} \big) \Big]_x %
  + \Big[ \delta B_\eps^w \big(f, \frac{w_x}{(1+\eps^2 w^2)^{\frac
        32}}\big) \Big]_x. \notag
\end{align}
Here, we have introduced the following notation: $s_\pm^{w}(t)$ is defined as in
\eqref{speed-local} with $f$ replaced by $w$; $D^w$ is defined as in
\eqref{def-MD} with $s_\pm(t)$ replaced by $s_\pm^{w}(t)$. Furthermore $\mathcal
D^w f = \delta D^w(f)$ and $\sigma_{\pm, w} f = \delta s_{\pm, w}(f)$ are the
linearizations of $D^w$, $\sigma_{\pm, w}$. Furthermore, $\delta B^w$ denotes
the shape derivative of $B^w$ (derivative in $w$ direction).  A dot on the top
of a symbol denotes the time derivative.

\medskip

The proof is based on two small parameters $\delta, \tau > 0$. The parameter
$\delta$ is used to localize the estimates near the boundary. Note that in the
interior where $C > h > c > 0$, the operator is uniformly parabolic. The
parameter $\tau$ signifies the time integral where the solution is defined. In
the course of the proof, we will choose $\delta$ and $\tau = \tau(\delta)$, in
this order, to be sufficiently small. In this section, we write $c$, $C$ for all
constants which depend only on $f_{\rm in}, k$, but neither depend on $\delta$
nor $\tau$.

\medskip

For $\delta > 0$, we define a covering of $E = (0,1)$ by setting
$\EEEE_{1\delta} := (0,2\delta), \EEEE_{2\delta} := (\delta,1-\delta),
\EEEE_{3\delta} := (1-2\delta, 1)$. Correspondingly, let $Q_{i\delta\tau} :=
(0,\tau) \times \EEEE_{i\delta}$. We choose a smooth partition of unity $\{
\psi_{i\delta} \}_{i=1,2,3}$, subordinate to this covering with $\psi_{i\delta}
\in C_c^{\infty}(\overline \EEEE_{i\delta}, [0,1])$, $\sum_i \psi_{i\delta} = 1$
on $\EEEE$ and also $\nco{\partial^j \psi_{i\delta}} \leq C \delta^{-j}$ for all
$j \geq 0$. We also define a second set of cut-off functions $\{ \tilde
\psi_{i\delta} \}_{i=1,2,3}$: Let $ \tilde \EEEE_{1\delta} := (0,4\delta),
\tilde \EEEE_{2\delta} := (3\delta,1-3\delta), \tilde \EEEE_{3\delta} :=
(1-4\delta, 1)$.  We choose $\psi_{1\delta} \in C_c^{\infty}(\overline
\EEEE_{i\delta}, [0,1])$ with $\tilde \psi_{i\delta} = 1$ in $E_{i\delta}$ and
in particular, $\tilde \psi_{i\delta} \psi_{i\delta} = \psi_{i\delta}$. Finally,
we also consider the cut-off function $\hat \psi_{1\delta} $, supported on an
even larger set such that $\hat \psi_{1\delta} \tilde \psi_{1\delta} = \tilde
\psi_{1\delta}$. The support of $\hat \psi_{1\delta} = 1$ is e.g. included in
$(0,5\delta)$, the functions $\hat \psi_{2\delta}, \hat \psi_{3\delta}$ are
defined analogously.

\medskip

Since $\XNL{f_{\rm in}}{3/2}{\EEEE} < C$ and $\PXNL{w}{2}{Q_\tau} < C$, it
follows that $f_{\rm in}$ and $w$ are H\"older continuous in time and
space. Therefore, by the boundary condition $|h_{\rm in,x}(0)| = |h_{\rm
  in,x}(1)| = 1$ and recalling that $h_{\rm in} = x(1-x) + \int_0^x f_{\rm
  in}(x') dx'$, there is $\delta_0 > 0$ such that for all $\delta < \delta_0$ we
have $|\partial_x h_0| \in (\frac 12, 2)$ in $\EEEE_{1\delta} \cap
\EEEE_{3\delta}$ and $h_0 \in (c, C)$ in $\EEEE_{2\delta}$. Similarly, there is
$\delta_0$ and $\tau_0$ such that the corresponding estimates hold for
$h^w(\tau) = x(1-x) + \int_0^x w(\tau,x') dx'$ for all $\delta < \delta_0$ and
every fixed time $\tau < \tau_0$. In the sequel, we will always assume $0 <
\delta, \tau < \min \{ \delta_0, \tau_0, 0.1 \}$.

\subsection{Proof of Proposition \ref{prp-local-linear}}

\begin{proof}[Proof of Proposition \ref{prp-local-linear}]
  In view of the extension Lemma \ref{lem-extension}, we only need to consider
  the case of zero initial data so that in the following we may assume $f_{\rm
    in} = 0$. We begin with the proof of maximal regularity estimate
  \eqref{est-linloc}. That is we assume that $f$ satisfies
  \begin{align} \label{fgeq} %
    \delta \LL_\eps^w f \ = \ g \ \ \text{in } Q_{\tau} && \text{and } && f \ =
    \ 0 \quad \text{on $(0,\tau) \times \partial E$}
  \end{align}
  and with initial data $f = 0$ and we will show
  \begin{align} \label{locest} %
    \PXNL{f}{k+1}{Q_\tau} \ \leq  \ C \PXNL{g}{k}{Q_\tau}.
  \end{align}
With the partition of unity $\psi_{i\delta}$, $i=1,2,3$ and by the triangle inequality, we have
  \begin{align} \label{split} %
    \PXNL{f}{k+1}{Q_\tau} \ \leq \ \PXNL{\psi_{1\delta} f}{k+1}{Q_\tau} +
    \PXNL{\psi_{2\delta} f}{k+1}{Q_\tau} + \PXNL{\psi_{3\delta} f}{k+1}{Q_\tau}.
  \end{align}
  We begin with the estimate for $\psi_{1\delta} f$ (related to the left
  boundary of the domain). The idea is to use that on $Q_{1\delta\tau}$ (that is
  near the left boundary), $\partial \LL$ (and also $L_\eps^w$) are approximated
  by $L_\eps$, where $L_\eps f := f_t + A_\eps f$ and where $A_\eps$ is defined
  in \eqref{def-A}. % By Proposition \ref{prp-linear} we get
  % \begin{align} \label{E-1} \PXNL{\psi_{1\delta} f}{k+1}{Q_\tau} &\leq \
  %   C \PXNL{L_\eps' (\psi_{1\delta} f)}{k}{Q_\tau}.
  % \end{align}
  We first claim that
  \begin{align} \label{byloc} %
    \PXNL{\psi_{1\delta} f}{k+1}{Q_\tau} &\leq C_\delta \PXNL{\tilde
      \psi_{1\delta} L_\eps (\psi_{1\delta} f)}{k}{Q_\tau } + C_\delta \PXNL{
      f}{k}{Q_\tau} + \frac1{10} \PXNL{ f}{k+1}{Q_\tau},
  \end{align}
  where we recall $\tilde \psi_{1\delta} \psi_{1\delta} = \psi_{1\delta}$ and
  hence $(1-\tilde \psi_{1\delta}) \psi_{1\delta} = 0$.  In order to see
  \eqref{byloc}, we first note that by \eqref{l-2}, we get $\PXNL{\psi_{1\delta}
    f}{k+1}{Q_\tau} \leq C \PXNL{L_\eps \psi_{1\delta} f}{k}{(0,\tau) \times K}$
  (the $L^2$-term is estimated by the other term since the support of the right
  hand side is bounded). Furthermore,
  \begin{align*}
    L_\eps (\psi_{1\delta} f) \ & %
    = \ \tilde \psi_{1\delta} L_\eps
    (\psi_{1\delta} f) + (1 - \tilde \psi_{1\delta}) L_\eps (\psi_{1\delta} f) \\ %
    &= \ \tilde \psi_{1\delta} A_\eps (\psi_{1\delta} f) + \psi_{1\delta} f_t +
    (1 - \tilde \psi_{1\delta}) A_\eps (\psi f).
  \end{align*}
  Furthermore, by \eqref{def-A} and since $\tilde \psi_{1\delta}$ depends only
  on $x$, by a short calculation we obtain
  \begin{align*}
    A_\eps(\psi_{1\delta} f ) &= \partial_x (-p'_x + \frac1{\eps^2} p'_y), \\
    (1 - \tilde \psi_{1\delta}) A_\eps(\psi_{1\delta} f ) &= \partial_x (-
    \tilde p'_x + \frac1{\eps^2} \tilde p'_y ) - \partial_{xx} \tilde
    \psi_{1\delta} p' - 2 \partial_x \tilde \psi_{1\delta} p'_x +
    \frac1{\eps^2}\partial_x \tilde \psi_{1\delta} p'_y,
  \end{align*}
  where $p' : K \to \R$ and $\tilde p' := (1 - \tilde \psi_{1\delta}) p'$ are
  the solutions of
  \begin{align} \label{p'-def1} %
    \left \{
      \begin{array}{ll}
        \Delta_\eps p' = 0  \ &\text{ in }  K, \\
        p' = (\psi_{1\delta} f)_x  &\text{ on } \Gamma = \partial_1 K , \\
        p'_y = 0  &\text{ on } \partial_0 K.
      \end{array}
    \right. && %
    \hspace{-2ex}
    \left\{
      \begin{array}{ll}
        \Delta_\eps \tilde p'  =  
        -2 \partial_x \tilde \psi_{1\delta} p'_x - \partial_x^2 \tilde \psi_{1\delta} p'   \  & \text{ in }  K, \\
        \tilde p' = 0  & \text{ on } \partial_1 K,     \\
        \tilde p'_y = 0  & \text{ on } \partial_0 K.
      \end{array}
    \right.
    \hspace{-2ex}
  \end{align}
  In the following, we use the following Rellich-type estimate for lower order
  terms which can be obtained by standard interpolation: For any given $\delta >
  0$, we have
  \begin{align} \label{small-large} %
    \XN{f_{xx}}{k-1}  \ \leq \ \delta \XN{f}{k}  + C_\delta  \XN{f}{k-1}.
  \end{align}
  Using this inequality, estimate \eqref{byloc} follows by application 
  of Proposition \ref{prp-p-linear}.

  \medskip

  Let $[L_\eps^w,\psi_{1\delta}] := L_\eps^w \psi_{1\delta} -
  \psi_{1\delta}L_\eps^w$ denote the commutator of $L_\eps^w$ and
  $\psi_{1\delta}$, then
  \begin{align}  \label{E-2} %
    \hspace{6ex} & \hspace{-6ex} %
    \PXNL{ \tilde \psi_{1\delta} L_\eps (\psi_{1\delta} f)}{k}{Q_{\tau}} %
    \leq \PXNL{ \tilde \psi_{1\delta} (L_\eps-\delta
      \LL_\eps^w)(f\psi_{1\delta})}{k}{Q_\tau}
    +  \PXNL{ \tilde \psi_{1\delta} \delta \LL_\eps^w (\psi_{1\delta} f)}{k}{Q_\tau}  \notag \\
    &\leq \PXNL{ \tilde \psi_{1\delta} (L_\eps-\delta
      \LL_\eps^w)(f\psi_{1\delta})}{k}{Q_\tau} + \PXNL{ \tilde \psi_{1\delta}
      [\delta \LL_\eps^w,\psi_{1\delta}]f}{k}{Q_\tau} %
    + \PXNL{\psi_{1\delta} g}{k}{Q_\tau},
  \end{align}
  where we have used $\psi_{1\delta} \tilde \psi_{1\delta} =
  \psi_{1\delta}$. The estimate of the first and second term on the right-hand
  side of the above estimate is given in Lemmas \ref{lem-difference} and
  \ref{lem-commutator}. Choosing $\delta$ sufficiently small, we hence obtain
  \begin{align*}
    \PXNL{\psi_{1\delta} f}{k+1}{Q_\tau} %
    &\leq \ C_\delta \PXNL{g}{k}{Q_\tau} + C_\delta \PXNL{f}{k}{Q_\tau} +
    \frac16 \ \PXNL{f}{k+1}{Q_\tau}.
  \end{align*}
  The third term on the right hand side of \eqref{split} can be estimated
  analogously. For the middle term (corresponding to the interior of the
  domain), we note that in the interior, our weighted Sobolev norms are
  equivalent to standard Sobolev norms (with equivalence depending on $\delta$)
  and an analogous estimate to the one above can be achieved for the middle term
  using standard parabolic estimates, see also
  \cite{Eidelman-Book,Krylov-Book}. Altogether, these estimates yield
  \begin{align*}
    \PXNL{f}{k+1}{Q_{\tau}} \ \upref{split}\leq \ C_\delta \PXNL{g}{k}{Q_{\tau}} +
    C_\delta \PXNL{f}{k}{Q_\tau} + \frac 12 \PXNL{f}{k+1}{Q_{\tau}} .
  \end{align*}
  Then using that $f_{|t=0} = 0$, we deduce that $\PXNL{f}{k}{Q_\tau} \leq C
  \tau^{1/2} \PXNL{f}{k+1}{Q_{\tau}}$ and we choose $\tau $ such that $C_\delta
  C \tau^{1/2} < \frac 16$. Hence,
  \begin{align*}
    \PXNL{f}{k+1}{Q_{\tau}} \ \upref{split}\leq \ C_\delta \PXNL{\DLL_\eps^w
      f}{k}{Q_{\tau}} + \frac 12 \PXNL{f}{k+1}{Q_{\tau}},
  \end{align*}
  which yields \eqref{locest} by absorbing on the left hand side. 

  \medskip

  The existence part is similar to the proof of Lemma 3.4 of
  \cite{GiacomelliKnuepfer-2010}. Indeed, the argument used there can be
  generalized since only very little of the particular structure is used: The
  main ingredient in this argument is existence and maximal regularity for the
  linearized operator at the boundary. The second ingredient is the fact that
  existence of a solution together with estimates in the interior follows by
  standard parabolic theory. We have already proved these properties for our
  operator. Finally, the argument also requires that the operator can be
  localized in the sense that the long-range interaction of the solution
  operator only yields a lower order contribution. Indeed, we have used this
  idea already in the proof of \eqref{byloc}.
\end{proof}

In the following we give the estimate for the right-hand side of \eqref{E-2}. We
use the notations and assumptions of the proof of Proposition
\ref{prp-local-linear}.
\begin{lemma}[Estimate of difference] \label{lem-difference} %
  For $ 0< \tau < 1$, we have
  \begin{gather} \label{est-dif} %
    \PXNL{ \tilde \psi_{1\delta} (L_\eps - \delta \LL_\eps^w) (\psi_{1\delta}
      f)}{k}{Q_\tau} \leq \ C_\delta \PXNL{ f }{k}{Q_\tau} + C \delta
    \PXNL{f}{k+1}{Q_\tau}.
  \end{gather}
\end{lemma}
 \begin{proof}
  Let $p'$ and $P$ be the solutions of
  \begin{align} \label{P-def} %
     \left \{
       \begin{array}{ll}
         \Delta_\eps p' = 0  \qquad &\text{ in }  K, \\
         p' = (\psi_{1\delta} f)_x  &\text{ on } \partial_1 K , \\
         p'_y = 0  &\text{ on } \partial_0 K,
       \end{array}
     \right.
    && %
    \left \{
      \begin{array}{ll}
        \Delta_\eps P = 0  \qquad &\text{ in }  \Omega^w, \\
        P = (\psi_{1\delta} f)_x  &\text{ on } \partial_1 \Omega^w, \\
        P_y = 0  &\text{ on } \partial_0 \Omega^w,
      \end{array} 
    \right.
  \end{align}
  where $\Omega^w = \{ (x,y) \, | \, 0< x < 1, \, \hbox{and }\, 0< y < h^w(x)
  \}$ and $\partial_1\Omega^w = \graph h^w$ and $h^w = x(1-x) + \int_0^x w(x')
  dx'$.  The existence of a solution $p'$ follows from Proposition
  \ref{prp-p-linear}, the existence of a solution $P$ can be shown with similar
  arguments as in the proof of Proposition \ref{prp-p-nonlinear}; they will not
  be detailed here. The reason to introduce these two functions is that we have
  $L_\eps - L_\eps^w = A_\eps - A_\eps^w$. Furthermore,
  $A_\eps(\psi_{1\delta} f ) = \partial_x (-p'_x + \frac1{\eps^2} p'_y ) $ and $
  A_\eps^w (\psi_{1\delta} f ) = \partial_x ( - h_x^w P_x + \frac1{\eps^2}
  P_y ) $.  Note that $A_\eps$ and $A_\eps^w$ are not defined on the same
  interval in $x$. To compare them, we therefore use the cut-off function
  $\tilde \psi_{1\delta}$. Indeed, $ \tilde \psi_{1\delta} P $ can be seen as a
  function in the domain $K^{\hat \psi_{1\delta} (w -\frac{x}2) } $ since the
  domains $K^{\hat \psi_{1\delta} (w -\frac{x}2) } $ and $\Omega^w $ coincide on
  the support of $ \tilde \psi_{1\delta} P $ (recall that $\hat \psi_{1\delta}
  \tilde \psi_{1\delta} = \tilde \psi_{1\delta}$).
 % Moreover, $\tilde \psi_{1\delta} P $ and $ \tilde
 % \psi_{1\delta} p' $ solve systems similar to \eqref{P-def} and \eqref{p'-def}
 %  with some extra terms on the right-hand side which are lower order.
 More
  precisely, $\tilde \psi_{1\delta} P $ solves
\begin{align} \label{P-def1} %
    &\left \{
      \begin{array}{ll}
        \Delta_\eps  (\tilde \psi_{1\delta} P)  = 
        2 \partial_x  \tilde \psi_{1\delta}  \partial_x P +  
        \partial_x^2   \tilde \psi_{1\delta} \,  P     \qquad &\text{ in } 
  K^{\hat  \psi_{1\delta}  (w -\frac{x}2) } , \\
        ( \tilde \psi_{1\delta} P   )  = (\psi_{1\delta} f)_x  &\text{ on } \partial_1 
 K^{\hat  \psi_{1\delta} (w -\frac{x}2)  }, \\
        ( \tilde \psi_{1\delta} P   )_y = 0  &\text{ on } \partial_0 K^{\hat  \psi_{1\delta}  (w -\frac{x}2) }.
      \end{array} 
    \right.
  \end{align}
  We use Lemma \ref{lem-psi} to construct a coordinate transform $\Psi =
  (\psi,id)$ from $K$ to $K^{\hat \psi_{1\delta} (w -\frac{x}2) }$.  Hence,
  arguing as in the proof of Proposition \ref{prp-p-nonlinear}, we deduce that $
  p = P \circ \Psi $ solves %a system similar to \eqref{pRpsi}
  \begin{align} \label{p-def2} %
    \left \{
      \begin{array}{ll}
        \Delta_\eps  (\tilde \psi_{1\delta}  p ) \ = \   R(p,\psi) +\KK_1(p, \psi)  \qquad &\text{ in }  K, \\
        \pi  \ = \ 0  &\text{ on } \partial_1 K , \\
        \pi_y \ = \ 0  &\text{ on } \partial_0 K.
      \end{array}
    \right.
  \end{align} 
  where $R(p,\psi)$ is given in \eqref{def-R} and where $\KK_1(p, \psi)$
  is a lower order term involving at most one derivative of $P$ and such that
  $\supp \KK_1 \subseteq \supp \partial_x \tilde \psi_{1\delta}$.  Arguing as in
  the proof of Proposition \ref{prp-p-nonlinear}, we deduce that $ \tilde
  \psi_{1\delta} p \in Y^1_\eps $ and hence $\tilde \psi_{1\delta} P \in
  Y^1_\eps (\Omega^w)$. Similarly, it can also be shown that $ (\tilde
  \psi_{2\delta} + \tilde \psi_{3\delta} ) P \in Y^1_\eps (\Omega^w) $.  One can
  then start a bootstrap argument and prove that $ P \in Y^{k+1}_\eps (\Omega^w)
  $, $ \| P \|_{ Y^k_\eps (\Omega^w) } \leq C(w) \| \psi_{1\delta} f
  \|_{X^k_\eps} $ and $ \| P \|_{ Y^{k+1}_\eps (\Omega^w) } \leq C(w) \|
  \psi_{1\delta} f \|_{X^{k+1}_\eps} $. Note that $\tilde \psi_{1\delta} p'$
  satisfies the same system as \eqref{P-def1} when replacing $P$ by $p'$ and
  $K^{\hat \psi_{1\delta} (w -\frac{x}2) }$ by $K$. Taking the difference, $\pi
  = p - p' $, we hence obtain
  \begin{align} \notag % \label{pp'-def1} %
    \left \{
      \begin{array}{ll}
        \Delta_\eps  (\tilde \psi_{1\delta}  \pi ) \ = \ R(p,\psi) +\KK_2(p,p',\psi)  \qquad &\text{ in }  K, \\
        \pi \  = \ 0  &\text{ on } \partial_1 K , \\
        \pi_y \ = \ 0  &\text{ on } \partial_0 K.
      \end{array}
    \right.
  \end{align} 
  where $\KK_2(p,p',\psi)$ is a lower order term with at most one
  derivative of $p$ or $p'$.  We have
  \begin{align*}
    \tilde \psi_{1\delta} A^w_\eps(\psi_{1\delta} f ) & = \big( - h^w_x ( \tilde
    \psi_{1\delta} P )_x
    + \frac1{\eps^2} ( \tilde \psi_{1\delta} P )_y  +  h^w_x P 
 (\tilde \psi_{1\delta} )_x \big)_x   -\partial_x  \tilde \psi_{1\delta}  ( - h^w_x  P_x + \frac1{\eps^2} P_y  ), \\
    \tilde \psi_{1\delta} A_\eps(\psi_{1\delta} f ) & = \big( - ( \tilde
    \psi_{1\delta} p' )_x + \frac1{\eps^2} ( \tilde \psi_{1\delta} p' )_y + p'
    (\tilde \psi_{1\delta} )_x \big)_x - \partial_x \tilde \psi_{1\delta} (
    -p'_x + \frac1{\eps^2} p'_y ).
  \end{align*}
  Hence and since $|h^w_x - 1 | \leq C \delta$ on $\supp \tilde \psi_{1\delta}$,
  we obtain
  \begin{align}
    \hspace{6ex} & \hspace{-6ex} %
    \tilde \psi_{1\delta} (L^w_\eps- L_\eps )(\psi_{1\delta} f ) \ %
    = \ \tilde \psi_{1\delta}    (A^w_\eps-   A_\eps )(\psi_{1\delta} f ) \nonumber \\
    & = \big( - ( \tilde \psi_{1\delta} \pi )_x + \frac1{\eps^2} ( \tilde
    \psi_{1\delta} \pi )_y + \pi (\tilde \psi_{1\delta} )_x \big)_x - \partial_x
    \tilde \psi_{1\delta} ( -\pi_x + \frac1{\eps^2} \pi_y ) +
    \KK_3, \label{3terms}
  \end{align}
  where the remainder term $\KK_3$ satisfies $ \PXNL{\KK_3}{k}{Q_\tau} \leq C
  \delta \PXNL{f}{k+1}{Q_\tau}$. Notice that the second term on the right-hand
  side on \eqref{3terms} consists of lower order term. Hence, we only need to
  estimate the first term.  We have
  \begin{align} \notag %
    \ZN{\KK_2}{k+1} \leq C_\delta \XN{ f}{k} + \delta \XN{ f}{k+1}
    &&\text{and}&& %
    \ZN{ R(p,\psi) }{k+1} \leq C \delta \XN{ f}{k+1} .
  \end{align}
  Hence, $ \| \pi \|_{ Y^{k+1}_\eps } \leq C_\delta \| f \|_{X^{k}_\eps} + C
  \delta \| f \|_{X^{k+1}_\eps} $ and \eqref{est-dif} follows easily.  In a
  sense we proved that the difference between $A^w_\eps$ and $ A_\eps$ comes
  from either terms which are small or terms which are more regular.  This shows
  the estimate of $L_\eps - L_\eps^w$, the estimate of $K_\eps^w$ follows
  similarly also using \eqref{small-large}. This concludes the proof of the
  Lemma.
\end{proof}

\begin{lemma}[Estimate of commutator] \label{lem-commutator} %
  For $ 0< \tau < 1$, we have 
  \begin{gather*}
    \PXNL{ \tilde \psi_{1\delta} [\delta \LL_\eps,\psi_{1\delta}]f}{k}{Q_\tau}
    \leq \ C_\delta \PXNL{f}{k}{Q_\tau} + C \delta \PXNL{f}{k+1}{Q_\tau}.
  \end{gather*}
\end{lemma}

\begin{proof}
  Let $P$ and $Q$ be the solutions of
  \begin{align} \label{P-def12} %
    \left \{
      \begin{array}{ll}
        \Delta_\eps P = 0  \qquad &\text{ in }  \Omega^w, \\
        P = (\psi_{1\delta} f)_x  &\text{ on } \partial_1 \Omega^w, \\
        P_y = 0  &\text{ on } \partial_1 \Omega^w,
      \end{array} 
    \right.  && \left \{
      \begin{array}{ll}
        \Delta_\eps Q = 0  \qquad &\text{ in }  \Omega^w, \\
        Q =  f_x  &\text{ on } \partial_1 \Omega^w, \\
        Q_y = 0  &\text{ on } \partial_0 \Omega^w.
      \end{array} 
    \right.
  \end{align} 
  Arguing as above, we can prove that $ Q, P \in Y^{k+1}_\eps (\Omega^w) $ and
  that $ \| Q \|_{ Y^{k+i}_\eps (\Omega^w) }, \| P \|_{ Y^{k+i}_\eps
    (\Omega^w) } \leq C(w) \| f \|_{X^{k+i}_\eps} $ for $i=0,1$.  Moreover,
  \begin{align}
    \tilde \psi_{1\delta} [L_\eps,\psi_{1\delta}]f \ &= \ \tilde \psi_{1\delta}
    \big[ ( - w_x P_x + \frac1{\eps^2} P_y )_x -
    \psi_{1\delta}  ( - w_x  Q_x + \frac1{\eps^2} Q_y  )_x     \big] \notag  \\
    &= \ \tilde \psi_{1\delta} \big[ ( - w_x (P- \psi_{1\delta} Q )_x +
    \frac1{\eps^2} (P - \psi_{1\delta} Q )_y )_x \big] + \KK, \label{commu}
  \end{align}
  where the terms of lower order are collected in $\KK$. Note that $P
  -\psi_{1\delta} Q $ solves
  \begin{align} \label{PQ-def} %
    \left \{
      \begin{array}{ll}
        \Delta_\eps (P - \psi_{1\delta}Q)  = 
        -2 \partial_x   \psi_{1\delta}  \partial_x Q -   
        \partial_x^2   \psi_{1\delta} \,  Q     \quad &\text{ in }  \Omega^w, \\
        P - \psi_{1\delta}Q =  \psi_{1\delta,x} f  &\text{ on } \partial_1 \Omega^w, \\
        (P - \psi_{1\delta}Q) _y = 0  &\text{ on } \partial_0 \Omega^w.
      \end{array} 
    \right.
  \end{align}
  The right-hand side of \eqref{PQ-def} is more regular and allow us to estimate
  $P -\psi_{1\delta} Q $ by $ \YN{P -\psi_{1\delta} Q }{k} \ \leq \ C \XN{ f}{k}
  +\frac1{10} \XN{ f}{k+1} $. Moreover, the operator $K$ on the right hand side
  of \eqref{commu} only involve one derivative of $Q$ and hence can be estimated
  similarly. This concludes the proof of the Lemma.
\end{proof}

\subsection{Proof of Proposition \ref{prp-local-nonlinear}} 

A localization of the estimate in Lemma \ref{lem-time-trace} yields the
following estimate: For all $k,\gamma \in \N_0$, $\tau > 0$ and $f_0 \in
\PXSLoo{k+1}{Q_\tau}$ we have
\begin{gather} \label{est-trace-local} %
  \|\partial_t^i f_0\|_{C^0(\XSL{k+1/2}{E})} \ \leq \ C \big(\|\partial_t^{i+1}
  f_0\|_{L^2(\XSL{k}{E})} + \|\partial_t^i f_0\|_{L^2(\XSL{k+1}{E})} \big).
\end{gather}
In view of Lemma \ref{lem-time-trace}, it remains to show boundedness and continuous
differentiability of $\LL_\eps$, defined in \eqref{def-LL1}. The highest order
term of $\LL_\eps$ is given by the operator $\partial_x B_\eps^f$, see
\ref{def-bt}. Boundedness of this operator follows from a localization of the
estimates in Proposition \ref{prp-p-nonlinear} and Lemma
\ref{lem-pressure-trace}: For $p$ be defined by \eqref{def-bt}, we have
\begin{align*}
  \XNL{ \partial_x B_\eps^f f_x}{k}{E} \ %
  &\leq \ C \big(\XNL{\partial_x(p_x)_{|\Gamma}}{k}{E} + (\feps)^2
  \XNL{\partial_x(p_y)_{|\Gamma}}{k}{E} \big) \\ %
  &\leq \ C \big(\XNL{f}{k+1}{E} + \ZNL{g}{k+1}{E} \big),
\end{align*}
where $g$ is a lower order term. Indeed, near the left boundary of $E$, $g$ is
the same term as in the right hand side of \eqref{p-def2}. In the center and
near the right boundary of $E$, $g$ is defined analogously. By a localization of
\eqref{RR-est} and since $\XNL{f}{k+1}{Q_\tau} \leq C$, we hence obtain
\begin{align*}
  \XN{\partial_x B_\eps^f f_x}{k} \ \leq \ C \XNL{f}{k+1}{E}.
\end{align*}
Furthermore, also using that $\XNL{f}{k+1}{Q_\tau} \leq C$, we have
\begin{align*}
  \XN{\dot s_\pm (f+f^*)}{k} %
  &\leq C \nco{\dot s_\pm} \big(\XN{f}{k}+ \XN{f^*}{k} \big) %
  \leq C \XN{f}{1} \big(\XN{f}{k}+ \XN{f^*}{k} \big) %
  \leq C \XN{f}{k+1}.
\end{align*}
It remains to show boundedness and continuity of the first derivative
$\DLL_\eps^w$: In view of \eqref{L-der}, the estimate of $\DLL_\eps^w$ leads
to analogous terms as the one for the estimate of $\LL_\eps$. The estimate hence
follows analogously and is not more difficult than the estimate of $\LL_\eps$
itself. Similarly, continuity of the first derivative follows by a bound on the
second derivative. As before, in view of the structure of the operator it is
clear that these terms do not impose any other difficulties than the one's
already when estimating the operator $\LL_\eps$ itself.  Hence, a bound on the
second derivative can be obtained similarly to the calculations before.

%\renewcommand{\appendixtocname}{Appendix}
%\addappheadtotoc
%\setcounter{section}{1}
%\setcounter{equation}{0}

%{\small
%\bibsep=1pt
%\bibliographystyle{spmpsci}
\bibliographystyle{abbrv} \bibliography{all} % name your BibTeX data base
%}

\end{document}